\documentclass{amsart}
\usepackage{amssymb}
\usepackage{hyperref}
\usepackage{graphicx} 
\usepackage{float}
\usepackage[matrix, arrow,all,cmtip,color]{xy}

\newtheorem{theorem}{Theorem}[section]
\newtheorem{nontheorem}[theorem]{Non-Theorem}

\newtheorem{fact}[theorem]{Fact}
\newtheorem{lemma}[theorem]{Lemma}
\newtheorem{corollary}[theorem]{Corollary}
\newtheorem{claim}[theorem]{Claim}
\newtheorem{proposition}[theorem]{Proposition}

\newtheorem{introtheorem}{Theorem}

\theoremstyle{definition}
\newtheorem{definition}[theorem]{Definition}
\newtheorem{example}[theorem]{Example}
\newtheorem{non-example}[theorem]{Non-Example}

\newtheorem{assumption}[theorem]{Assumption}
\newtheorem{remark}[theorem]{Remark}

\newcommand{\acl}{\operatorname{acl}}
\newcommand{\dcl}{\operatorname{dcl}}

\newcommand{\tp}{\operatorname{tp}}
\DeclareMathOperator{\bd}{bd}
\DeclareMathOperator{\inter}{int}

\newcommand{\groupconf}[6]{
	\xy <1.5cm,0cm>: (1,0.2)*{#4}; (2,0.2)*{#5}; (0.8,-0.8)*{#6};
	(-0.2,0)*{#1}; (-0.2,-1)*{#2}; (-0.2,-2)*{#3}; (1,0)*{}; (0,-2)*{}
	**\dir{-}; (0,-2)*{}; (0,0)*{} **\dir{-}; (0,0)*{}; (2,0)*{}
	**\dir{-}; (2,0)*{}; (0,-1)*{} **\dir{-};
	\endxy }

\makeatletter
\let\@wraptoccontribs\wraptoccontribs
\makeatother

\begin{document}

\title{Topologically 1-Based T-minimal Structures}
\author{Benjamin Castle}
\address{Department of Mathematics, University of Illinois Urbana-Champaign}
\email{btcastl2@illinois.edu}

\author{Assaf Hasson}
\address{Department of Mathematics, Ben Gurion University of the Negev, Be'er-Sheva, Israel}
\email{hassonas@math.bgu.ac.il}

\contrib[With an appendix by]{Will Johnson}
\address{School of Philosophy, Fudan University,
Shanghai 200433, P.R. China}
\email{ willjohnson@fudan.edu.cn}

\thanks{Castle was supported by NSF grant DMS-2452735. Hasson was supported by ISF grant No. 555/21.}

\maketitle

\begin{abstract}
    We prove group existence and structure theorems in a general setting of tame topological theories. More precisely, we identify a linear/non-linear dividing line -- called \textit{topological 1-basedness} -- among the class of t-minimal theories with the independent neighborhood property. This is a wide class including all visceral theories, as well as all dense weakly o-minimal and C-minimal theories (even those where exchange fails). 
    
    Now assume $\mathcal M$ is highly saturated and t-minimal with the independent neighborhood property. We show that if $\mathcal M$ is non-trivial and topologically 1-based, it admits a type-definable abelian group $(G,+)$ with $G$ an open subset of $M$. Moreover, we can ensure that $G$ is a topological group with the subspace topology inherited from $M$; and in this case, we show that the induced structure on $G$ satisfies an appropriate topological analog of the Hrushovski-Pillay classification of 1-based stable groups. 
\end{abstract}

\tableofcontents

\section{Introduction}

In this paper we introduce the notion of a \textit{topologically 1-based t-minimal structure}, and prove that when such a structure is non-trivial, it admits an infinite type-definable abelian topological group, whose (local) structure we then proceed to study. These results can be viewed as a topological analogue of the literature on 1-based stable structures: on the one hand, the existence of a type-definable group resembles Hrushovski's characterization of locally modular regular types (\cite{Hr4}, following earlier works of Zilber \cite[\S 3]{ZilBook} and Cherlin-Harrington-Lachlan \cite{CHL}). On the other hand, our structural analysis of the group closely resembles the Hrushovski-Pillay analysis of 1-based stable groups (\cite{HrPi87}). 

Our main theorem reads as follows (see Theorem \ref{T: main 1-based}; the relevant terms are defined in Definitions \ref{D: tmin}, \ref{D: inp}, \ref{D: nontrivial}, and \ref{D: 1-based type}):

\begin{introtheorem}\label{T: main intro}
    Let $\mathcal M$ be a sufficiently saturated t-minimal structure with the independent neighborhood property. Assume that $\mathcal M$ is non-trivial and topologically 1-based. Then there are a countable parameter set $A$, and an $A$-type-definable abelian group $G$, such that the following hold:
    \begin{enumerate}
        \item $G$ is open in $\mathcal M$.
        \item $G$ is a topological group with the topology inherited from $M$.
        \item $G$ is locally linear.
    \end{enumerate}
\end{introtheorem}

We explain the terminology more below; for now, we note that the above theorem encompasses all dense weakly o-minimal and C-minimal structures; in particular, as far as we know, it is the first group-existence theorem for (non-o-minimal) weakly o-minimal structures.

\subsection{Background} To recall, 1-basedness in stability is an abstract notion viewed as capturing those stable theories of \textit{linear} complexity. Classical results in stability show that in non-trivial 1-based theories, any instance of non-triviality can be explained in terms of type-definable abelian groups (e.g., \cite[Corollary 5.4.13]{PillayBook}). The key notion allowing the construction of such groups is the \textit{germ} of a map on a complete (stationary) type -- where the word \textit{germ} here indicates that maps are viewed up to agreement on a non-forking extension of the domain.

Soon after the aforementioned works on 1-basedness in stable theories, it became clear that \textit{o-minimality} provided a suitable tame (yet unstable) setting where many stability-like results could be developed. In particular, o-minimal structures were seen to exhibit a similar dimension theory to strongly minimal structures (and thus many notions could be translated from stability to o-minimality). To that end, early works of Peterzil and Loveys (\cite{Pe},\cite{LovPet}) developed a version of linearity for o-minimal structures, and obtained similar group-existence and structure theorems (which ultimately became part of the o-minimal trichotomy theorem, \cite{PeStTricho}). The key notion here was again the germ of a function -- but in the o-minimal setting, germs retain their original topological meaning (one treats functions up to agreement on an interval).

Many aspects of Peterzil's group-existence proof (\cite{Pe}) are quite general, but the definition of linearity is rather specific to o-minimality. In a later work, Maalouf \cite{MaaCmin} adapted the definition of linearity to the related context of \textit{C-minimality}, and proved a similar group existence theorem -- however, Maalouf needed to assume the \textit{exchange lemma}, which can fail in a general C-minimal structure. 

The above works on o-minimality and C-minimality naturally lead to the following question: is there a general notion of linearity for `topologically tame' structures, and can one use it to prove abstract group-existence and structure theorems? For example, what can be said for \textit{weakly o-minimal} structures, and for C-minimal structures without exchange? 

\subsection{T-minimality and Work of Johnson}

Answering the above question would require a general setting for topological tameness, and this general setting should admit a suitable dimension theory for definable sets and types. Recently, Johnson \cite{t-minimal-johnson} has shown that \textit{t-minimality} provides such a setting. Recall that a (sufficiently saturated) structure is t-minimal if (1) it admits a uniformly $\emptyset$-definable Hausdorff topology, and (2) a unary definable set is infinite if and only if it has non-empty interior with respect to the given topology. T-minimality encompasses all dense weakly o-minimal and C-minimal structures (regardless of exchange), in addition to many valued fields and expansions thereof (e.g. all 1-h-minimal fields, and all unstable dp-minimal expansions of fields).

Johnson's paper develops a rudimentary dimension theory assuming only t-minimality, then proceeds to develop stronger geometric properties (e.g. generic continuity of functions) assuming a stronger condition called \textit{viscerality} (saying roughly that the topology is `uniform'). In our setting, we do not want to assume viscerality: the issue is that in many cases, viscerality follows from the existence of a group structure, and we want to work a priori without a group structure. For example, while all o-minimal groups are visceral, there are also non-visceral linear o-minimal structures (and our results should apply to them).

Instead, we augment t-minimality with a weaker condition, which we call the \textit{independent neighborhood property}: roughly speaking, this property says that one can always `zoom in' on a tuple without changing its dimension. The same property was already studied in the dp-minimal visceral context in \cite{HaHaPeVF}, where it proved quite useful. The key features of the independent neighborhood property in our case are:

\begin{itemize}
    \item It encompasses all visceral, weakly o-minimal, and C-minimal structures.
    \item It implies many of the same geometric properties (including generic continuity) that Johnson's paper proves for visceral structures.
\end{itemize}

In light of the above discussion, our main theorem shows that the known group-existence and structure theorems for o-minimal and C-minimal structures are special cases of a general result, valid for all t-minimal structures with the independent neighborhood property. Our proof follows a similar (but not identical) outline to the original proof of \cite{Pe}, with various new challenges encountered throughout. In particular, note that we \textit{do not} assume exchange (and this causes quite a few difficulties); as a tradeoff, our version of non-triviality needs to be a bit stronger than the usual one (see Definition \ref{D: nontrivial}).

\subsection{Topological 1-basedness}

The main task in our proof is to identify and study a linearity notion -- namely \textit{topological 1-basedness} -- in the absence of either exchange or an order (recall that Peterzil and Loveys' o-minimal version used both exchange and the order). Let us now discuss our version in more depth. 

Just as \textit{germs} play a key role in the stable and o-minimal group existence theorems, they also play a key role in our setting. In fact, it follows from Johnson's paper that in a t-minimal structure, every complete type $\tp(a/A)$ has a well-defined germ at any realization $a$ (see Lemma \ref{L: germs exist}), and this germ can be coded as an imaginary element $\operatorname{germ}(a/A)$.

When constructing groups in the stable setting, germs are manifested definably as \textit{canonical bases}. Thus, it is natural to view the imaginary $\operatorname{germ}(a/A)$ as a sort of `topological canonical base' of $\tp(a/A)$, localized to the point $a$. There is good evidence for this line of thinking -- for example, assuming the independent neighborhood property, if $B\supseteq A$ then one has $\dim(a/B)=\dim(a/A)$ if and only if $\operatorname{germ}(a/B)=\operatorname{germ}(a/A)$ (see Lemma \ref{L: weakly ind}).

In light of this analogy, our definition of topological 1-basedness essentially amounts to replacing canonical bases in the stability-theoretic definition with germs. There is one difference, however, which makes our definition a bit more general. Recall that in stability, 1-basedness requires the canonical base of a type to be algebraic over a single realization. Since the canonical base is imaginary, the word `algebraic' here refers to the operator $\acl^{eq}$. In geometric structures, Gagelman extends the dimension theory to imaginary elements, where the $0$-dimensional closure may be a proper coarsening of $\acl^{eq}$. In our view, it is this 0-dimensional closure operator (rather than $\acl^{eq}$) which correctly captures linearity in our setting. Of course, 0-dimensionality only makes sense assuming exchange (i.e. in geometric structures). However, one can also characterize it topologically. 
That is, in geometric t-minimal structures, 0-dimensional imaginaries are those associated with definable equivalence relations whose classes are \textit{open} (or at least those which have interior in some ambient set). Thus, dropping the assumption of exchange and keeping only $t$-minimality (with independent neighbourhoods), our notion of linearity is the following:
say that $\tp(a/Ab)$ is \textit{topologically 1-based over $A$} if the following equivalent conditions hold (see Lemma \ref{L: 1-based iff constant germ}):
\begin{itemize}
    \item $\dim(b/A\operatorname{germ}(a/Ab))=\dim(b/Aa)$.
    \item The map $y\mapsto\operatorname{germ}(a/Ay)$ is constant on a neighborhood of $b$ in $\tp(b/Aa)$.
\end{itemize}

Then say that $\mathcal M$ is topologically 1-based if $\tp(a/b)$ is topologically 1-based over $\emptyset$ for all real tuples $a$ and $b$.

Note that this definition generalizes known versions of linearity in the literature. For example, our main structure theorem for topologically 1-based types, Theorem \ref{T: partition}, shows that topological 1-basedness is equivalent to an appropriately localized version of Pillay's notion of weak normality. Outside of the stable context, in, e.g.  \cite{BerVas}, Berenstein and Vassiliev develop a version of linearity in the related context of \textit{geometric structures}, which they call \textit{weak 1-basedness}. When restricted to o-minimality, weak 1-basedness is shown to capture exactly the linear o-minimal structures in the sense of Peterzil and Loveys. Now suppose $\mathcal M$ is t-minimal with the independent neighborhood property, and $\mathcal M$ moreover satisfies exchange. Then $\mathcal M$ is a geometric structure. In Corollary \ref{C: weak 1b}, we show that such an $\mathcal M$ is topologically 1-based if and only if it is weakly 1-based as a geometric structure -- thus, topological 1-basedness also generalizes linearity in o-minimal structures (see Corollary \ref{C: omin}).

\subsection{Infinitesimals and Regularity} Once we have the notion of topological 1-basedness, our task is to use it to build type-definable groups. There are two key notions involved here. The first, following earlier works such as Peterzil's o-minimal group configuration of \cite{PetGconf}, is the \textit{infinitesimal neighborhood} $\mu(a/A)$ of a tuple $a$ over a parameter set $A$. This refers to the set of realizations of $\tp(a/A)$ from a highly saturated elementary extension which are `infinitesimally close' to $a$ in the sense of the original model (see Definition \ref{D: inf neigh}). Infinitesimal neighborhoods allow us to replace finite-to-finite correspondences, locally, with homeomorphisms. For example, suppose $a$ and $b$ are interalgebraic over $A$; then using generic continuity, we show that $\mu(ab/A)$ is the graph of a homeomorphism $\mu(a/A)\rightarrow\mu(b/A)$ (see Lemma \ref{L: mu generic cont}). By repeated applications of this fact, it follows that if $a$ and $b$ are interalgebraic over $At$, then $\mu(abt/A)$ is the graph of a $\mu(t/A)$-indexed family of homeomorphisms $\mu(a/A)\rightarrow\mu(b/A)$. Thus, we have a convenient tool for constructing families of homeomorphisms -- and such objects will ultimately be the building blocks of a type-definable group action.

The other key notion is \textit{regularity}. Using infinitesimal neighborhoods, one can construct (for example) three points $a,b,c\in M$, a set $A$, and three compatible type-definable families of homeomorphisms $\mu(a/A)\rightarrow\mu(b/A)$, $\mu(b/A)\rightarrow\mu(c/A)$, and $\mu(a/A)\rightarrow\mu(c/A)$ (compatible meaning that any composition $\mu(a/A)\rightarrow\mu(b/A)\rightarrow\mu(c/A)$ agrees with one of the given maps $\mu(a/A)\rightarrow\mu(c/A)$). One then needs to find additional conditions assuring that (1) these families of maps arise from a group (or rather, at this level, a groupoid) and (2) the resulting groupoid is type-definable in an appropriate sense.

The key notion is regularity (inspired by regular group actions): say that a family of maps $f_t:X\rightarrow Y$ is \textit{regular} if for any $(x,y)\in X\times Y$ there is $t$ with $f_t(x)=y$, and if any two such maps either agree everywhere or agree nowhere. Regular families of maps can be seen as a manifestation of linearity (or 1-basedness), in the sense that a single point determines an entire map in the family. In Sections 7 and 8, we study type-definable systems of compatible regular families of maps in more detail, showing that in many cases such configurations give rise to type-definable group actions (this is in some sense a completely abstract group configuration theorem -- see Theorem \ref{T: combinatorial group conf}). Then in Section 9, we show that such collections of regular families arise naturally under the assumption of topological 1-basedness. The key input here is a structure theorem for topologically 1-based types (see Theorem \ref{T: partition}), which roughly says that such types always look regular in an infinitesimal neighborhood (as mentioned above, this is a topological analog of the \textit{weakly normal} decomposition in \cite{HrPi87}):

\begin{introtheorem}\label{IT: partition} Let $\mathcal M$ be a sufficiently saturated t-minimal structure with the independent neighborhood property. Let $a$ and $b$ be real tuples, and $A$ a parameter set. Then the following are equivalent: 
\begin{enumerate}
    \item $\tp(a/Ab)$ is topologically 1-based over $A$.
    \item Any two fibers in the projection $\mu(ab/A)\rightarrow\mu(b/A)$ are either equal or disjoint. 
\end{enumerate}
\end{introtheorem}
So, condition (2) of the above theorem is how we prove that the families of homeomorphisms we build are regular, and thus gives rise to a type-definable group. 

\subsection{Summary of the Paper}

The paper is roughly organized into four main parts:

\begin{itemize}

\item Sections 2-4 study t-minimal theories and the independent neighborhood property in general, focusing on germs and generic continuity. These sections do not involve topological 1-basedness or groups. Section 2 constructs and studies germs in arbitrary t-minimal theories; Section 3 adds the independent neighborhood property to the background assumptions and develops generic continuity; and Section 4 rewrites the key points of the previous sections in the language of infinitesimal neighborhoods.

\item Sections 5-6 introduce topological 1-basedness. Section 5 is concerned mainly with the definition and basic properties. For example, we prove that topological 1-basedness is preserved under adding and removing constants (Lemma \ref{L: constants}) and taking reducts (Lemma \ref{L: reducts}), and that a topologically 1-based structure cannot admit an infinite definable field (Lemma \ref{L: no field}). Then Section 6 is devoted to the main structure theorem for topologically 1-based types, Theorem \ref{IT: partition} above.

\item Sections 7-8 leave the topological setting entirely and focus on a completely abstract version of group configurations, which we call \textit{regular groupoid spines}. This is something like a regular groupoid, but where only certain ordered pairs of objects have morphisms between them. Section 7 shows that any regular groupoid spine with at least 3 objects arises from a genuine regular groupoid (this is really a group configuration theorem). Then Section 8 shows that if a regular groupoid spine is type-definable in an appropriate sense, so is the groupoid it arose from (and thus so is the automorphism group at any of its objects).

\item Finally, Sections 9-11 combine the previous two parts to prove our main theorems. Suppose that $\mathcal M$ is non-trivial and topologically 1-based. In Section 9 we use infinitesimal neighborhoods to construct a type-definable regular groupoid spine with three objects in $\mathcal M$ (precisely, every required property other than regularity comes from the general theory of infinitesimal neighborhoods, and regularity comes from Theorem \ref{IT: partition}). Thus, by the results of Sections 7-8, we obtain a type-definable group in an elementary extension of $\mathcal M$. Then in Section 10, we show that the group can be transferred down to $\mathcal M$, and that -- by virtue of generic continuity -- it can be tweaked slightly to make it a topological group with the induced topology from $\mathcal M$. Finally, Section 11 shows that the group is locally abelian and locally linear. The main tool here is again Theorem \ref{IT: partition}. The idea is that any two infinitesimal translates of an infinitesimal neighborhood (in an appropriate sense) are equal or disjoint -- and this implies that the given infinitesimal neighborhood is a coset of a subgroup. Thus, the local structure of definable sets can be analyzed entirely in terms of subgroups and their cosets, in a similar manner to the Hrushovski-Pillay analysis of 1-based stable groups (\cite{HrPi87}).

We can also give some local analysis. Namely, using the notion of topologically 1-based types, we define what it means for a one-dimensional type to be \textit{topologically locally modular} (even if the ambient structure is not topologically 1-based as a whole). In this case (assuming non-triviality), we are able to build a type-definable topological group on an open subset of the realizations of the type (see Theorem \ref{T: top gp loc mod}). However, the structural results for the group (local linearity and local abelianity) need full topological 1-basedness, at least for all tuples from the group itself.
\end{itemize}

Finally, we give a brief appendix devoted to examples of t-minimal structures with the independent neighborhood property,  covering some non-visceral structures, most notably weakly o-minimal and C-minimal structures. We also include a complementary appendix by Will Johnson, proving that visceral theories have the independent neighborhood property.  

\subsection{Notation and model theoretic generalities}
We use standard model-theoretic notation and terminology. Structures are denoted by calligraphic capital letters ($\mathcal M, \mathcal N$ etc.), and their universes are denoted by the corresponding Latin letters $M$, $N$. All tuples are assumed to be finite. As is customary, by \emph{real} tuples, we refer to tuples in $M^n$ (some $n$), as opposed to \textit{imaginary} tuples which may reside in $\mathcal M^{eq}$.  

Our structures will typically be assumed to be \emph{sufficiently saturated}. To clarify, when we say $\mathcal M$ is \textit{sufficiently saturated}, we mean it is $\kappa$-saturated and $\kappa$-strongly homogenous for some uncountable cardinal $\kappa$ strictly larger than the cardinality of the language of $\mathcal M$. Here, $\kappa$ is unspecified but fixed to go along with the structure $\mathcal M$. In this case, the word \textit{small} always means `of size less than $\kappa$' -- and by a \textit{parameter set}, we mean a small subset of $\mathcal M^{eq}$. Thus, smallness and parameter sets are sensitive to the particular choice of model. For example, when studying infinitesimal neighborhoods, we will be working in an $|M|^+$-saturated model $\mathcal N\succ \mathcal M$. So $M$ is small with respect to $\mathcal N$, but obviously not with respect to $\mathcal M$. 

By a \textit{type-definable set}, we mean the intersection of a small collection of definable sets. We use throughout that the projection of a type-definable set is again type-definable (over the same parameters). A complete type over a parameter set $A$ is a maximally consistent type (in the language $\mathcal L$ expanded with constants symbols for all $a\in A$). For every tuple $a$ and parameter set $A$ we denote by $\tp(a/A)$ the unique complete type over $A$ realized by $a$. \\

\section{T-minimality and Germs}

\textbf{From now until the end of Section 6, we fix a sufficiently saturated first-order structure $\mathcal M$ and a Hausdorff topology $\tau$ on $M$ with a $\emptyset$-definable basis.} Note that there is also a $\emptyset$-definable basis for the product topology on $M^n$, given by products of $n$ basic open sets in $M$.

We will frequently use the following easy consequence of compactness:

\begin{fact}\label{F: int of opens} Let $X\subset M^n$. Then the open subsets of $X$ are closed under small intersections.
\end{fact}
\begin{proof}
    Let $\{U_t:t\in T\}$ be a definable basis for the topology on $M^n$, and let $\{V_i:i\in I\}$ be a small collection of open subsets of $X$. Suppose we are given $a\in\bigcap V_i$. For each $i$, choose $t_i$ with $X\cap U_{t_i}\subset V_i$. Consider the partial type saying that $y\in T$, $a\in U_y$, and $U_y\subset U_{t_i}$ for each $i$. This partial type is finitely satisfiable and small, so it has a realization in $\mathcal M$. This shows that $\bigcap (X\cap V_i)$ is open in $X$.
\end{proof}

\subsection{T-minimality}

Recall the following definition (originally appearing in \cite{Matth}):

\begin{definition}\label{D: tmin}
    $\mathcal M$ is \textit{t-minimal} if for all definable sets $X\subset M$, $X$ is infinite if and only if it has non-empty interior. 
\end{definition}

The t-minimal structures form a fairly wide class, including most known instances of topological tameness in model theory. In particular, t-minimality encompasses all weakly o-minimal and C-minimal structures, as well as Henselian valued fields of characteristic zero.

Johnson (\cite[Definition 2.26, 2.29]{t-minimal-johnson}) has developed a general dimension theory in t-minimal structures. In particular, for any tuple $a\in M^n$ and any set $A$, one has a dimension $\dim(a/A)$, defined as the length of any minimal tuple $b\subseteq a$ with $a\in\acl(Ab)$ (such a tuple is called a \textit{basis} of $a$ over $A$; one of the results of Johnson's paper -- \cite[Proposition 2.25]{t-minimal-johnson} -- is that this is well-defined, i.e. any two bases have the same length). As usual, one then defines the dimension of an $A$-definable set $X$ (or an $A$-type-definable set) as the maximum value of $\dim(a/A)$ for $a\in X$ (and as usual, this does not depend on $A$, as long as it is small). We caution that the dimension theory nevertheless exhibits wild behavior in general t-minimal structures, and refer the reader to Johnson's paper for some examples. For now, let us list some crucial properties of the dimension theory that we will use (all of these are in \cite[Proposition 2.31 and Theorem 2.37]{t-minimal-johnson}):

\begin{fact}\label{F: tmin} Assume $\mathcal M$ is t-minimal, and consider real tuples $a$ and $b$, parameter sets $A$ and $B$, and type-definable sets $X$ and $Y$.
\begin{enumerate}
    \item $\dim(M^n)=n$ for all $n$.
    \item If $X\subset Y$ then $\dim(X)\leq\dim(Y)$.
    \item $\dim(X\times Y)=\dim(X)+\dim(Y)$.
    \item If $f:X\rightarrow Y$ is relatively definable and surjective with finite fibers, then $\dim(X)=\dim(Y)$.
    \item If $A\subset B$ then $\dim(a/B)\leq\dim(a/A)$.
    \item\label{subad} $\dim(ab/A)\leq\dim(b/A)+\dim(a/Ab)$.
    \item If $\acl(Aa)=\acl(Ab)$ then $\dim(a/A)=\dim(b/A)$.
    \item $\dim(a/A)$ is the smallest dimension of an $A$-definable set containing $A$.
\end{enumerate}
\end{fact}

Fact \ref{F: tmin}(\ref{subad}) is known as \textit{sub-additivity}, and will be referenced as such throughout the paper.

\begin{assumption} \textbf{From now until the end of Section 6, assume $\mathcal M$ is t-minimal.}
\end{assumption}

\subsection{Basic Witnesses}

Suppose $a\in M^n$ is a real tuple, and $A$ is a parameter set. Let $b\in M^d$ be a basis for $a$ over $A$, and let $\pi:M^n\rightarrow M^d$ be a corresponding projection sending $a$ to $b$. Since $a\in\acl(Ab)$, there is an $A$-definable set $X\subset M^n$ such that $a\in X$ and $\pi$ is finite-to-one on $X$. We give this a separate name:
\begin{definition}
    Let $a\in M^n$, $A$ a parameter set, and $b\in M^d$ a basis for $a$ over $A$.
    \begin{enumerate}
        \item By an $a$-witness for $b$ over $A$, we mean an $A$-definable set $X\subset M^n$ with $a\in X$ and admitting a finite-to-one projection $\pi:X\rightarrow M^d$ with $\pi(a)=b$.
        \item Let $X$ be an $a$-witness for $b$ over $A$, witnessed by the projection $\pi:M^n\rightarrow M^d$. We say that $X$ is a \textit{minimal $a$-witness for $b$ over $A$} if the fiber size $|\pi^{-1}(b)\cap X|$ is minimal among all such witnesses. That is, let $Y$ be another $a$-witness for $b$ over $A$, also witnessed by the projection $\pi$. Then $|\pi^{-1}(b)\cap X|\leq|\pi^{-1}(b)\cap Y|$.
        \item A set $X$ is a \textit{basic witness} (resp. \textit{mimimal basic witness}) of $a$ over $A$ if there is some basis $b$ of $a$ over $A$ such that $X$ is an $a$-witness (resp. minimal $A$-witness) for $b$ over $A$.
    \end{enumerate}
\end{definition}

So for any $a$ and $A$, there is a minimal basic witness of $a$ over $A$; and the basic witnesses (resp. minimal basic witnesses) form a cofinal family of sets in $\tp(a/A)$.

\subsection{Germs} One of our main tools will be the existence of \textit{germs} of types. This was developed in a slightly different setting in (\cite[\S 3.2]{CasHasYe}). In that paper, the existence of germs followed immediately from the frontier inequality (i.e. $\dim(\operatorname{Fr}(X))<\dim(X)$ for definable $X$). In the present setting, the frontier inequality can fail (Johnson's paper gives precise examples, including even some with the independent neighborhood property (Definition \ref{D: inp})). However, germs still exist for a different reason -- and in many ways, they will serve as a replacement for the frontier inequality for the rest of the paper. 

\begin{definition}\label{D: germ}
    Let $a\in M^n$.
    \begin{enumerate}
        \item If $X\subset M^n$, the \textit{germ of $X$ at $a$}, denoted $\operatorname{germ}_a(X)$, is the equivalence class realized by $X$ among all subsets of $M^n$ modulo agreement in a neighborhood of $a$. In this case, denoting $g=\operatorname{germ}_a(X)$, we say that $X$ \textit{realizes} $g$.
        \item A \textit{definable germ} at $a$ is a germ at $a$ realized by some definable subset of $M^n$.
        \item If $g$ is a definable germ at $a$, then the \textit{dimension} of $g$, denoted $\dim(g)$, is the smallest value of $\dim(X)$ where $X$ is a definable set realizing $g$.
        \end{enumerate}
\end{definition}

\begin{remark} A restatement of Definition \ref{D: germ}(3) is as follows: let $g$ be a definable germ at $a$, realized by the definable set $X$. Then $\dim(g)$ is the \textit{local dimension of $X$ at $a$} -- that is, the smallest value of $\dim(X\cap U)$ where $U$ is a definable neighborhood of $a$.
\end{remark}

\begin{lemma}\label{L: germs exist} Let $a\in M^n$, and $A$ a parameter set. Then the formulas in $\operatorname{tp}(a/A)$ have a cofinal germ at $a$. Precisely, let $X$ be any minimal basic witness of $a$ over $A$. Then for every $A$-definable $Y\subset X$ containing $a$ we have $\operatorname{germ}_a(X)=\operatorname{germ}_a(Y)$. 
\end{lemma}

\begin{proof}
    The case when $a$ is independent (that is, $\dim(a/A)=n$) is, essentially, due to Johnson. More precisely, in this case, the lemma just says that if $X$ is $A$-definable and contains $a$, then $X$ contains a neighborhood of $a$. But if not, then $a$ belongs to $X-\operatorname{Int}(X)$, and thus $\dim(X-\operatorname{Int}(X))=n$. By Theorem 1.7(3) of Johnson's paper \cite{t-minimal-johnson}, this implies that $X-\operatorname{Int}(X)$ has non-empty interior, which is absurd.
        
    Now in the general case, let $X$ be any minimal basic witness. So $X$ is a minimal $a$-witnesses for some basis $b$, witnessed by the projection $\pi:M^n\rightarrow M^d$ with $\pi(a)=b$.
    
    Let $Y\subset X$ be any other $A$-definable set containing $a$. Since $Y\subset X$, $Y$ is also an $a$-witness for $b$ over $A$ (also witnessed by $\pi$). Since the fiber size above $b$ in $X$ is minimal among all such sets, this means $X$ and $Y$ agree at all points in the fiber above $b$. By the previous paragraph (since $b$ is independent over $A$), there is a neighborhood $U$ of $b$ such that $X$ and $Y$ agree on all fibers above points in $U$. In particular, $X$ and $Y$ agree on the preimage of $U$ in $M^n$, which is a neighborhood of $a$.
\end{proof}

\begin{definition}
 Let $a\in M^n$ and let $A$ be a parameter set. We define the \textit{germ of $a$ over $A$}, denoted $\operatorname{germ}(a/A)$, to be $\operatorname{germ}_a(X)$ for any set $X$ as in Lemma \ref{L: germs exist}. 
\end{definition}

Thus $\operatorname{germ}(a/A)$ is always a definable germ at $a$. Moreover, by Fact \ref{F: int of opens}, $\operatorname{germ}(a/A)$ is precisely the germ at $a$ of the set of realizations of $\tp(a/A)$.

\subsection{Operations on Germs}

Germs come with certain natural relations. We review these now. 

\begin{definition}
    Let $g$ and $h$ be germs at a point $a\in M^n$. We say that $g$ is a \textit{subgerm} of $h$, denoted, $g\leq h$ if for some $X$ realizing $g$ and $Y$ realizing $h$, we have $X\subset Y$. Equivalently, for any $X$ realizing $g$ and $Y$ realizing $h$, we have $X\cap U\subset Y\cap U$ for some neighborhood $U$ of $a$.
\end{definition}

\begin{definition}
    Let $g$ and $h$ be germs at points $a\in M^m$ and $b\in M^n$. The \textit{product of $g$ and $h$}, denoted $g\times h$, is the germ at $(a,b)\in M^{m+n}$ realized by $X\times Y$ whenever $X$ realizes $g$ and $Y$ realizes $h$.
\end{definition}

\begin{definition}
    Let $a\in M^m$ and $b\in M^n$, and let $g$ be a germ at $(a,b)\in M^{m+n}$. The \textit{fiber germ of $g$ over $b$}, denoted $g_b$, is the germ at $a$ realized by the fiber $X_b\subset X^n$ whenever $X$ realizes $g$.
\end{definition}

One easily has:

\begin{lemma}\label{L: germ facts}
    Let $a\in M^m$ and $b\in M^n$.
    \begin{enumerate}
        \item The relation $\leq$ gives a partial order on germs at $a$.
        \item If $g$ and $h$ are germs at $a$ and $b$, respectively, then $(g\times h)_b=g$.
        \item If $g$ and $h$ are germs at $(a,b)$, and $g\leq h$, then $g_b\leq h_b$.
        \item If $g$ and $h$ are definable germs at $a$ and $b$, respectively, then $$\dim(g\times h)=\dim(g)+\dim(h).$$
    \end{enumerate}
\end{lemma}

The following facts are also apparent:

\begin{lemma}\label{L: type germ facts}
    Let $a\in M^m$, $b\in M^n$, and $A\subset B$ parameter sets.
    \begin{enumerate}
        \item $\operatorname{germ}(a/B)\leq\operatorname{germ}(a/A)$.
         \item $\operatorname{germ}(ab/A)\leq\operatorname{germ}(a/A)\times\operatorname{germ}(b/A)$.
        \item If $X$ is $A$-definable and contains $a$, then $\operatorname{germ}(a/A)\leq\operatorname{germ}_a(X)$.
    \end{enumerate}
\end{lemma}

\subsection{Codes of Germs}

We remark here that definable germs can be viewed as elements of $\mathcal M^{eq}$. In particular, this will apply to $\operatorname{germ}(a/A)$ for any $a$ and $A$. 

Note that since the topology on $\mathcal M$ is $\emptyset$-definable, every automorphism $\sigma$ of $\mathcal M$ is a homeomorphism. It follows that if $g$ is a germ at $a\in M^n$, then $\sigma(g)$ is a well-defined germ at $\sigma(a)$. Now we define:

\begin{definition}\label{D: germ code}
    Let $a\in M^n$, and let $g$ be a germ at $a$. We say that $c\in\mathcal M^{eq}$ is a \textit{code} for $g$ if the following holds: let $\sigma$ be any automorphism of $\mathcal M$. Then $\sigma(c)=c$ if and only if $\sigma(g)=g$ (so in particular $\sigma(a)=a$).
\end{definition}

\begin{remark}
    An immediate consequence of Definition \ref{D: germ code} is that if $g$ is a germ of $a$ with code $c$, then $a\in\dcl(c)$.
\end{remark}

Now we show:

\begin{lemma}\label{L: germ codes}
    Let $a\in M^n$ and let $g$ be a definable germ at $a$. Then there is a code for $g$ in $\mathcal M^{eq}$. In particular, for any parameter set $A$, there is a code for $\operatorname{germ}(a/A)$ in $\mathcal M^{eq}$.
\end{lemma}
\begin{proof}
    Let $X$ be a definable set realizing $g$. Write $X=Y_b$ for some $\emptyset$-definable family $\{Y_t:t\in T\}$ of sets. Let $\sim$ be the equivalence relation on $M\times T$ given by $(x,s)\sim(y,t)$ if $x=y$ and $Y_s$ and $Y_t$ agree in a neighborhood of $x$. Then $\sim$ is $\emptyset$-definable, so $T/\sim$ is a sort in $\mathcal M^{eq}$. Now let $c\in T/\sim$ be the equivalence class of $(a,b)$. Then one checks easily that $c$ is a code of $g$. 
\end{proof}

Definition \ref{D: germ code} should remind the reader of canonical bases in stability theory, particularly when considering germs of the form $\operatorname{germ}(a/A)$. In fact, germs will serve as a replacement for canonical bases when interpreting a group later on. The main difference between the two notions is that germs at $a$ contain the data of the point $a$, while canonical bases do not. This boils down to which equivalence relation is absorbed in the coding process: one should think of canonical bases (when they exist) as coding formulas in $\operatorname{tp}(a/A)$ up to agreeing \textit{almost everywhere}, while germs code formulas in $\tp(a/A)$ up to agreement \textit{in a neighborhood of $a$}.

Note that there can be several codes of $\operatorname{germ}(a/A)$, but any two are interdefinable over $a$. For this reason -- as is often the case with canonical bases -- we will abuse notation by identifying all codes of a germ with the germ itself. Thus, if $g$ is a code of $\operatorname{germ}(a/A)$, we will simply write $g=\operatorname{germ}(a/A)$.

\subsection{D-approximations} Following (\cite[Definition 3.16]{CasHasYe}), we also define:

\begin{definition}
    Let $a\in M^n$ and let $A$ be a parameter set. A \textit{d-approximation} of $\tp(a/A)$ at $a$ is a definable set $X$ satisfying $a\in X$, $\dim(X)=\dim(a/A)$, and $\operatorname{germ}_a(X)=\operatorname{germ}(a/A)$. 
\end{definition}

\begin{remark} Note that a $d$-approximation of $\operatorname{tp}(a/A)$ at $a$ need not be $A$-definable.
\end{remark}

For example, if $a\in M^n$ and $\dim(a/A)=n$, then a d-approximation of $\tp(a/A)$ at $a$ is just a definable set containing a neighborhood of $a$ in $M^n$. In general, one should think of a d-approximation as analogous to a definable set of Morley degree one in an $\omega$-stable theory: in that context, any complete type $\tp(a/A)$ is (up to parallelism and passing to $\acl(A)$) the unique generic type of some definable set $X$ of degree 1. In our context, any complete type $\tp(a/A)$ is the unique `generic at $a$' type of any d-approximation -- where `generic at $a$' means that no formula in the type decreases the germ at $a$. 

We now give a couple of easy facts about d-approximations. First, a trivial but useful fact is the following:

\begin{lemma}\label{L: d-app locality} Let $a\in M^n$, let $A$ be a parameter set, and let $X$ be a d-approximation of $\tp(a/A)$ at $X$. Then any $A$-definable property which holds of $a$, also holds on a neighborhood of $a$ in $X$.
\end{lemma}
\begin{proof}
    Let $Y$ be an $A$-definable set containing $a$. By Lemma \ref{L: germ facts}(3), we compute: $$\operatorname{germ}_a(Y)\geq\operatorname{germ}(a/A)=\operatorname{germ}_a(X),$$ which is equivalent to the statement of the lemma.
\end{proof}

A restatement of Lemma \ref{L: germs exist} is the following:

\begin{corollary}\label{C: minimal witness is d-app}
    Let $a\in M^n$ and $A$ a parameter set. Then any minimal basic witness of $a$ over $A$ is a d-approximation of $\tp(a/A)$ at $a$.
\end{corollary}
\begin{proof}
    Let $X$ be a minimal basic witness, corresponding to the projection $\pi:M^n\rightarrow M^d$ sending $a$ to the basis $b=\pi(a)$. Lemma \ref{L: germs exist} gives that $X$ realizes $\operatorname{germ}(a/A)$. Now $\pi$ is finite-to-one on $X$, so by \cite[Theorem 1.7(10)]{t-minimal-johnson}, $\dim(X)\leq d$; while since $a\in X$, we also have $\dim(X)\geq d$. So $\dim(X)=d$, and thus $X$ is a d-approximation.
\end{proof}

\section{Independent Neighborhoods}

We will need a natural strengthening of t-minimality, which we call the \textit{independent neighborhood property}. This property has already been studied in the more restrictive context of SW-uniformities and related structures (\cite{HaHaPeVF}, \cite[\S 3.1]{HaHaPeGps}).  It is also the key geometric idea underlying the definition of \emph{vicinic sets} in $\S 4$ of the same paper. In this section, we introduce the property in the abstract t-minimal setting and study its basic features. The theme will be that independent neighborhoods allow for stronger geometric properties of definable sets and types.

\subsection{Definition and Examples}

\begin{definition}\label{D: inp}
    $\mathcal M$ has the \textit{independent neighborhood property} if for any tuple $a\in M^n$, any parameter set $A$, and any neighborhood $U$ of $a$, there is a neighborhood $V$ of $a$ such that $V\subset U$ and $V$ is definable over a parameter $t$ with $\dim(a/At)=\dim(a/A)
    $.  
\end{definition}

Definition \ref{D: inp} only requires the new parameters to preserve the dimension of the original tuple $a$. However, it follows that one can preserve the dimension of any tuple containing $a$:

\begin{lemma}\label{L: inp for bigger tuple}
    Suppose $\mathcal M$ has the independent neighborhood property. Let $a\in M^m$, $b\in M^n$, let $A$ be a parameter set, and let $U$ be a neighborhood of $a$. Then there is a neighborhood $V\subset U$ of $a$ defined over a tuple $t$ with $\dim(ab/At)=\dim(ab/A)$.
\end{lemma}
\begin{proof}
    The set $U\times M^n$ is a neighborhood of $(a,b)$. By the independent neighborhood property, there is a neighborhood $W\subset U\times M^n$ of $(a,b)$ defined over a parameter $t$ with $\dim(ab/At)=\dim(ab/A)$. Let $\pi:M^m\times M^n\rightarrow M^m$ be the projection. Then $\pi(W)\subset U$ is a $t$-definable neighborhood of $a$.
\end{proof}

Most natural examples of t-minimal theories have the independent neighborhood property. For instance, it was shown in (\cite[Corollary 3.12]{HaHaPeVF}) that all SW-uniformities have the independent neighborhood property,  covering all dp-minimal expansions of topological groups. More generally, Johnson shows in Appendix \ref{A: visceral} that all visceral theories have the independent neighborhood property, and thus so do \textit{all} t-minimal topological groups. Finally, in Appendix A, we show that all C-minimal and (dense) weakly o-minimal theories have the independent neighborhood property.

The simplest example of a t-minimal theory without the independent neighborhood property is likely one pointed out by Johnson. Namely, consider a real closed field with the Sorgenfrey topology (generated by the half-open intervals of the form $[a,b)$). This gives a t-minimal theory. However, consider a generic singleton $a$, and any $b>a$. If $V\subset[a,b)$ is any neighborhood of $a$, then $a$ is definable over the canonical parameter of $V$ (as the minimum element of $V)$. It follows that no such $V$ can be defined independently of $a$. 

\begin{assumption} \textbf{From now until the end of Section 6, assume $\mathcal M$ has the independent neighborhood property.}
\end{assumption}

\subsection{Geometric Consequences of Independent Neighborhoods}

Independent neighborhoods imply stronger geometric properties of definable sets. We now give various examples. The first lemma below shows that the local dimension of a d-approximation of $\tp(a/A)$ at $a$ is always $\dim(a/A)$. In contrast, without the independent neighborhood property one can have infinite discrete definable subsets of $M^2$ (consider $\{(x,-x)\}\subset\mathbb R^2$ with the Sorgenfrey topology; in that case, if $a\in\mathbb R$ is generic then $\operatorname{germ}(a,-a/\emptyset)$ is a single point).

\begin{lemma}\label{L: local dimension}
    Let $a\in M^n$, let $A$ be a parameter set, and let $X$ be a d-approximation of $\tp(a/A)$ at $a$.
    \begin{enumerate}
        \item Let $Y\subset X$ be any definable set containing a neighborhood of $a$ in $X$. Then $Y$ is also a d-approximation of $\tp(a/A)$ at $a$. Thus, the local dimension of $X$ at $a$ is $\dim(a/A)$.
        \item In particular, $\dim(\operatorname{germ}(a/A))=\dim(a/A)$.
    \end{enumerate}
        \end{lemma}
\begin{proof}
    We show (1); then (2) follows by definition.
    
    To show (1), first note that $Y$ trivially realizes $\operatorname{germ}(a/A)$, because $X$ does. The point is that we can also show $\dim(Y)=\dim(a/A)$. For this, by assumption, there is a neighborhood $U$ of $a$ such that $X$ and $Y$ agree on $U$. By the independent neighborhood property, we may assume $U$ is definable over $At$ where $\dim(a/At)=\dim(a/A)$. Now $Y\cap U=X\cap U$ is $At$-definable and contains $a$, so $\dim(Y\cap U)\geq\dim(a/A)$, and thus $\dim(Y)\geq\dim(a/A)$ as well. Equality then follows since $Y\subset X$ and $\dim(X)=\dim(a/A)$. 
\end{proof}

Lemma \ref{L: local dimension} implies a convenient reformulation of the independent neighborhood property in terms of d-approximations:

\begin{lemma}\label{L: GenOS}
    Let $a\in M^n$, and let $A$ be a parameter set. \begin{enumerate}
        \item Let $X$ be any d-approximation of $\tp(a/A)$ at $a$. Then there is a d-approximation $Y\subset X$ of $\tp(a/A)$ at $a$, such that $Y$ is definable over $At$ for some tuple $t$ with $\dim(a/At)=\dim(a/A)$.
        \item In particular, if $U$ is any neighborhood of $a$ in $M^n$, then there is a d-approximation $Y\subset U$ of $\tp(a/A)$ at $a$ which is definable over a tuple $t$ with $\dim(a/At)=\dim(a/A)$.
    \end{enumerate}
\end{lemma}
\begin{proof}
\begin{enumerate}
    \item Let $Z$ be an $A$-definable d-approximation of $\tp(a/A)$ at $a$ (for example, any minimal basic witness of $a$ over $A$). Then $X$ and $Z$ agree on some neighborhood of $a$, say $U$. By the independent neighborhood property, we may assume $U$ is $At$-definable, where $\dim(a/At)=\dim(a/A)$. Now take $Y=Z\cap U$. Then $Y$ is $At$-definable and realizes $\operatorname{germ}(a/A)$; while we have $Y=Z\cap U=X\cap U\subset X$, so $Y$ is a d-approximation by Lemma \ref{L: local dimension}.
    \item Let $X$ be any d-approximation of $\tp(a/A)$. Then $X\cap U$ is also a d-approximation of $\tp(a/A)$, by Lemma \ref{L: local dimension}. Now apply (1) to $X\cap U$.
\end{enumerate}
\end{proof}

Finally, we show that the independent neighborhood property implies a more natural characterization for d-approximations. Recall that under only t-minimality, a \textit{minimal} basic witness is always a d-approximation. Assuming independent neighborhoods, the same holds of \textit{all} basic witnesses. This will be very useful when characterizing independence in the next subsection.

\begin{lemma}\label{L: basic witness}
    Let $a\in M^n$, $A$ a parameter set, and $X$ a basic witness of $a$ over $A$. Then $X$ is a d-approximation of $\tp(a/A)$ at $a$.
\end{lemma}
\begin{proof}
    As in the proof of Corollary \ref{C: minimal witness is d-app}, we automatically have $\dim(X)=d$. To prove the lemma, we show that $\operatorname{germ}_a(X)$ is cofinal in $\tp(a/A)$. So let $Y\subset X$ be $A$-definable with $a\in Y$. We will show that $X$ and $Y$ agree in a neighborhood of $a$.
    
    Let $\pi:M^n\rightarrow M^d$ demonstrate that $X$ is a basic witness; so $a\in X$, $\pi$ is finite-to-one on $X$, and $b=\pi(a)$ is a basis of $a$ over $A$.
    
    Since $\pi$ is finite-to-one on $X$, clearly there is a neighborhood $U$ of $a$ such that $X\cap U$ and $Y\cap U$ agree on the fiber above $b$ (one chooses $U$ small enough that $|X\cap U|$ has only the single point $a$ lying above $b$). By the independent neighborhood property, we may assume $U$ is $At$-definable where $\dim(a/At)=\dim(a/A)$. It follows that $b$ is still a basis for $a$ over $At$. Thus $\dim(b/At)=d$, and so any $At$-definable property of $b$ holds on a neighborhood of $b$. In particular, this implies that $X\cap U$ and $Y\cap U$ agree on all fibers above points in some neighborhood of $b$; and as in Lemma \ref{L: germs exist}, this implies the desired statement.
\end{proof}

\subsection{Independence}

Due to the lack of exchange, there are two natural notions of independence in our setting. Namely, given tuples $a$ and $b$, one could only ask that $\dim(a/b)=\dim(a)$, or one could ask that the same also holds with $a$ and $b$ reversed. In this subsection, we show that the independent neighborhood property implies a geometric characterization of each of these notions of independence.

\begin{lemma}\label{L: weakly ind}
    Let $a\in M^n$, and $A\subset B$ parameter sets. The following are equivalent:
    \begin{enumerate}
        \item $\dim(a/B)=\dim(a/A)$.
        \item $\tp(a/A)$ and $\tp(a/B)$ have a common d-approximation at $a$.
        \item $\tp(a/A)$ and $\tp(a/B)$ have exactly the same d-approximations at $a$.
        \item $\operatorname{germ}(a/B)=\operatorname{germ}(a/A)$.
    \end{enumerate}
\end{lemma}
\begin{proof}
    (1)$\Rightarrow$(2): Let $X$ be a basic witness of $a$ over $A$. If $\dim(a/B)=\dim(a/A)$, then by definition $X$ is also a basic witness of $a$ over $B$\footnote{This is a subtle use of independent neighborhoods: if we only knew that minimal basic witnesses were d-approximations, the same argument would fail.}. By Lemma \ref{L: basic witness}, $X$ is a d-approximation of both $\tp(a/A)$ and $\tp(a/B)$ at $a$.

.?''//
'[
\\\\\\
];'    (2)$\Rightarrow$(3): By definition of d-approximations, (2) gives that $\tp(a/A)$ and $\tp(a/B)$ have the same dimension and the same germ at $a$, and thus they have the same d-approximations at $a$.

    (3)$\Rightarrow$(4): Clear, since a d-approximation contains the data of the relevant germ.

    (4)$\Rightarrow$(1): Assume $\operatorname{germ}(a/B)=\operatorname{germ}(a/A)$. Then By Lemma \ref{L: local dimension}, we have $$\dim(a/B)=\dim(\operatorname{germ}(a/B))=\dim(\operatorname{germ}(a/A))=\dim(a/A).$$
\end{proof}

\begin{remark}
    Lemma \ref{L: weakly ind} fails without the independent neighborhood property, again witnessed by $\mathbb R$ with the Sorgenfrey topology and the inverted diagonal. In that case, for generic $a\in\mathbb R$, we have $\operatorname{germ}(a,-a/a)=\operatorname{germ}(a,-a/\emptyset)$ (each is a single point), while $0=\dim(a,-a/a)<\dim(a,-a/\emptyset)=1$.
\end{remark}

\begin{definition}
    Let $a\in M^n$ and $A\subset B$ parameter sets. We say that $a$ is \textit{independent from $B$ over $A$} if the equivalent conditions of Lemma \ref{L: weakly ind} hold.
\end{definition}

We now turn to the obvious symmetric notion of independence.

\begin{lemma}\label{L: strongly ind}
    Let $a$ and $b$ be tuples, and $A$ a parameter set. The following are equivalent:
    \begin{enumerate}
        \item $\dim(a/Ab)=\dim(a/A)$ and $\dim(b/Aa)=\dim(b/A)$.
        \item $\dim(ab/A)=\dim(a/A)+\dim(b/A)$.
         \item There are sets $X$ and $Y$ such that $X$ is a d-approximation of $\tp(a/A)$ at $a$, $Y$ is a d-approximation of $\tp(b/A)$ at $b$, and $X\times Y$ is a d-approximation of $\tp(ab/A)$ at $(a,b)$.
         \item Whenever $X$ is a d-approximation of $\tp(a/A)$ at $a$, and $Y$ is a a-approximation of $\tp(b/A)$ at $b$, we have that $X\times Y$ is a d-approximation of $\tp(ab/A)$ at $(a,b)$.
        \item $\operatorname{germ}(ab/A)=\operatorname{germ(a/A)}\times\operatorname{germ}(b/A)$.
    \end{enumerate}
\end{lemma}
\begin{proof}
    (1)$\Rightarrow$(2): Assume (1). Let $c$ be a basis of $a$ over $A$, and let $d$ be a basis of $b$ over $A$. Then $ab\in\acl(Acd)$. If we can show that the concatenated tuple $cd$ is independent over $A$, it will follow that $cd$ is a basis of $ab$ over $A$, which implies (2). To show this, assume $cd$ is not independent over $A$. Then there is some coordinate from $cd$ that is algebraic over $A$ with all other coordinates of $cd$. Call this coordinate $e\in M$. If $e$ occurs in $c$, then $c$ is not independent over $Ad$, which implies that $\dim(a/Ab)<\dim(a/A)$. Similarly, if $e$ occurs in $d$, then $\dim(b/Aa)<\dim(b/A)$. In either case, we contradict (1).

    (2)$\Rightarrow$(3): Let $X$ and $Y$ be basic witnesses for $a$ and $b$ over $A$, respectively. Assuming (2) we have $$\dim(X\times Y)=\dim(X)+\dim(Y)=\dim(a/A)+\dim(b/A)=\dim(ab/A),$$ which implies easily that $X\times Y$ is a basic witness for $ab$ over $A$\footnote{As in Lemma \ref{L: weakly ind}, this argument fails with minimal basic witnesses, so we are really using independent neighborhoods.}. This implies (3) by Lemma \ref{L: basic witness}.
    
    (3)$\Rightarrow$(4): Assuming (3), we conclude that $\dim(ab/A)=\dim(a/A)+\dim(b/A)$ and $\operatorname{germ}(ab/A)=\operatorname{germ}(a/A)\times\operatorname{germ}(b/A)$. This implies (4) by definition.

    (4)$\Rightarrow$(5): By definition of d-approximations.

    (5)$\Rightarrow$(1): Assume $\operatorname{germ}(ab/A)=\operatorname{germ}(a/A)\times\operatorname{germ}(b/A)$. Then by Lemmas \ref{L: local dimension} and \ref{L: germ facts}(4), we have $$\dim(ab/A)=\dim(\operatorname{germ}(ab/A))=\dim(\operatorname{germ}(a/A)\times\operatorname{germ}(b/A))$$ $$=\dim(\operatorname{germ}(a/A))+\dim(\operatorname{germ}(b/A))=\dim(a/A)+\dim(b/A).$$ That is, we have shown that (5) implies (2). But (2) easily implies (1) by sub-additivity. %We show that $\dim(a/Ab)=\dim(a/A)$; the other statement is similar.
    
\end{proof}

\begin{definition}
    Let $a$ and $b$ be tuples, and $A$ a parameter set. We say that $a$ and $b$ are \textit{symmetrically independent over $A$} if the equivalent conditions of Lemma \ref{L: strongly ind} are satisfied.
\end{definition}

\subsection{Generic Continuity} Johnson \cite[Remark 1.16]{t-minimal-johnson} points out that generic continuity can fail in a general t-minimal structure. His example is yet again a real closed field with the Sorgenfrey topology and the inverted diagonal (namely $x\mapsto -x$ is nowhere continuous). In this subsection, we show that this is no coincidence: in fact, independent neighborhoods imply strong generic continuity results. In particular, a corollary will be that generic continuity holds for all definable functions with domain a power of $M$.

Most of our generic continuity results need some local instance of additivity. Thus, we define: 

\begin{definition}
    Let $a$ and $b$ be tuples, and $A$ a parameter set. We say that $(a,b)$ is \textit{additivite over $A$} if $\dim(ab/A)=\dim(b/A)+\dim(a/Ab)$.
\end{definition}

So if $\mathcal M$ satisfies exchange, then all pairs are additive.\\

Our main use of additivity, which we will exploit frequently, is the following:

\begin{lemma}\label{L: additive pairs}
    Let $(a,b)$ be additive over $A$, and let $B\supset A$ with $\dim(ab/B)=\dim(ab/A)$. Then:
    \begin{enumerate}
        \item $\dim(b/B)=\dim(b/A)$.
        \item $\dim(a/Bb)=\dim(a/Ab)$.
        \item $(a,b)$ is additive over $B$.
    \end{enumerate}
\end{lemma}
\begin{proof}
    Combining additivity over $A$ with sub-additivity over $B$, we have: 
    $$\dim(ab/B)\leq\dim(b/B)+\dim(a/Bb)\leq\dim(b/A)+\dim(a/Ab)=\dim(ab/A).$$ Since $\dim(ab/B)=\dim(ab/A)$, all of the above inequalities must be equalities. This implies (1), (2), and (3).
\end{proof}

Now the main technical input in generic continuity is the following:

\begin{lemma}\label{L: d-app images}
Let $a$ and $b$ be tuples so that $(a,b)$ is additive over $A$. Let $\pi$ be the projection sending $(a,b)$ to $b$. If $X$ is any d-approximation of $\tp(ab/A)$ at $(a,b)$, then $\pi(X)$ contains a d-approximation of $\tp(b/A)$ at $b$.
\end{lemma}

\begin{proof}
    By Lemma \ref{L: d-app locality}, we may assume $X$ is definable over $At$ where $\dim(ab/At)=\dim(ab/A)$. By Lemma \ref{L: additive pairs}, we have $\dim(b/At)=\dim(b/A)$. It follows by Lemma \ref{L: weakly ind} that $\operatorname{germ}(b/At)=\operatorname{germ}(b/A)$.

    On the other hand, note that $\pi(X)$ is an $At$-definable set containing $b$. Applying Lemma \ref{L: type germ facts}(3), we obtain $$\operatorname{germ}(b/A)=\operatorname{germ}(b/At)\leq\operatorname{germ}_b(\pi(X)).$$ Now let $Y$ be any d-approximation of $\tp(b/A)$ at $b$. Then $\operatorname{germ}_b(Y)\leq\operatorname{germ}_b(\pi(X))$. It follows that $Y\cap\pi(X)$ contains a neighborhood of $b$ in $Y$, so by Lemma \ref{L: d-app locality} is a d-approximation of $\tp(b/A)$ at $b$. Since $Y\cap\pi(X)$ is contained in $\pi(X)$, the lemma is now proved. 
\end{proof}

We conclude that projections are locally open at additive pairs:

\begin{lemma}\label{L: d-app open}
    Let $(a,b)$ be additive over $A$, and let $\pi$ be the projection sending $(a,b)$ to $b$. Then there are d-approximations $X$ of $\tp(ab/A)$ at $(a,b)$, and $Y$ of $\tp(b/A)$ at $b$, so that $\pi(X)=Y$ and $\pi:X\rightarrow Y$ is open.
\end{lemma}

\begin{proof}
    Fix $A$-definable d-approximations $X$ and $Y$ of $\tp(ab/A)$ at $(a,b)$ and $\tp(b/A)$ at $b$, respectively. Shrinking $X$ if necessary, we may assume $\pi(X)\subset Y$ (here we use that $X\cap\pi^{-1}(Y)$ is also $A$-definable, so is also a d-approximation). Recall that:
    \begin{enumerate}
        \item Every neighborhood of $(a,b)$ in $X$ is also a d-approximation of $\tp(ab/A)$ at $(a,b)$, by Lemma \ref{L: local dimension}; and
        \item Every d-approximation of $\tp(b/A)$ at $b$ contains a neighborhood of $b$ in $Y$, by definition.
        \end{enumerate}
        
        In particular, combining (1) and (2) with Lemma \ref{L: d-app images} gives that for every neighborhood $N$ of $(a,b)$ in $X$, $\pi(N)$ contains a neighborhood of $b$ in $Y$. This is an $A$-definable property of $(a,b)$, so by Lemma \ref{L: d-app locality} it holds on a neighborhood of $(a,b)$ in $X$; that is, there is a neighborhood $X'$ of $(a,b)$ in $X$ such that $\pi:X'\rightarrow Y$ is open. Let $Y'=\pi(X')$, so that $Y'$ is open in $Y$ and $\pi:X'\rightarrow Y'$ is open and surjective. Then by Lemma \ref{L: local dimension}, $X'$ and $Y'$ are d-approximations of $\tp(ab/A)$ at $(a,b)$ and $\tp(b/A)$ at $b$, respectively, and this proves the desired statement.
\end{proof}

In the case of algebraic pairs, we can improve `open' to `homeomorphism':

\begin{lemma}\label{L: d-app homeo}
    Let $a$ and $b$ be tuples, and $A$ a parameter set, so that $a\in\acl(Ab)$. Let $\pi$ be the projection sending $(a,b)$ to $b$. Then there are d-approximations $X$ of $\tp(ab/A)$ at $(a,b)$, and $Y$ of $\tp(b/A)$ at $b$, so that $\pi$ restricts to a homeomorphism $X\rightarrow Y$. 
\end{lemma}

\begin{proof}
    Since $a\in\acl(Ab)$, we have $\dim(ab/A)=\dim(b/A)$ and $\dim(a/Ab)=0$; thus $(a,b)$ is additive over $A$. By Lemma \ref{L: d-app open}, we can choose d-approximations $X$ and $Y$ so that $\pi:X\rightarrow Y$ is surjective and open. As a projection, $\pi$ is automatically continuous. The only missing ingredient is ensuring that $\pi$ is injective on $X$.

    Now separately, since $a\in\acl(Ab)$, there is an $A$-definable set $Z$ containing $(a,b)$ so that $\pi$ is finite-to-one on $Z$. By Lemma \ref{L: type germ facts}(3), $$\operatorname{germ}_{(a,b)}(Z)\geq\operatorname{germ}(ab/A)=\operatorname{germ}_{(a,b)}(X).$$ Thus $X\cap Z$ contains a neighborhood of $(a,b)$ in $X$, so is also a d-approximation of $\tp(ab/A)$ at $(a,b)
    $ by Lemma \ref{L: local dimension}. Now let $X'=X\cap Z$ and $Y'=\pi(X\cap Z)$; then $Y'$ is open in $Y$, so is a d-approximation of $\tp(b/A)$ at $b$ by Lemma \ref{L: local dimension}, and $\pi:X'\rightarrow Y'$ is still open and continuous. Finally since $X'\subset Z$, $\pi$ is moreover injective on $X'$, so that $\pi:X'\rightarrow Y'$ is a homeomorphism as desired.
\end{proof}

We conclude with generic continuity of functions on powers of $M$. In fact, we will be a bit more general. In his paper, Johnson (\cite[Definition 2.47]{t-minimal-johnson}) defines a \textit{very weak} $k$-cell to be a definable set $X$ admitting a finite-to-one projection $X\rightarrow M^k$ whose image has non-empty interior. One checks easily that any open subset of $M^k$ is a very weak $k$-cell, and that any very weak $k$-cell has dimension $k$. Now we show:

\begin{corollary}\label{C: generic continuity}
    Let $f:X\rightarrow Y$ be an $A$-definable function, where $X$ is a very weak $k$-cell. Let $a\in M^n$ with $\dim(a/A)=k$. Then $f$ is continuous in a neighborhood of $a$.
    \end{corollary}

    \begin{proof}
        Let $\pi:X\rightarrow M^k$ be a projection whose image has non-empty interior, as provided by the assumption that $X$ is a very weak $k$-cell. Let $e=\pi(a)$, and let $b=f(a)$. The assumptions imply that $\dim(e/A)=\dim(a/A)=\dim(ab/A)=k$, and thus $e$ is a basis for both $a$ and $ab$ over $A$. Letting $\Gamma\subset X\times Y$ be the graph of $f$, it now follows that $\Gamma$ and $X$ are basic witnesses of $ab$ and $a$ over $A$, respectively, and thus are d-approximations of $\tp(ab/A)$ and $\tp(a/A)$ at $(a,b)$ and $a$, respectively.
        
        Now by Lemma \ref{L: d-app homeo}, there are d-approximations $X'$ of $\tp(a/A)$ and $\Gamma'$ of $\tp(ab/A)$ such that the projection $(a,b)\mapsto a$ induces a homeomorphism $\Gamma'\rightarrow X'$. Then $\operatorname{germ}_a(X')=\operatorname{germ}_a(X)$, and $\operatorname{germ}_{(a,b)}(\Gamma')=\operatorname{germ}_{(a,b)}(\Gamma)$. So, letting $X''=X\cap X'$ and $\Gamma''=\Gamma\cap\Gamma'$, it follows that $X''$ is a neighborhood of $a$ in $X$ and $\Gamma''$ is a neighborhood of $(a,b)$ in $\Gamma$. Note that if $x\in X''$, then $(x,f(x))\in\Gamma''$ (and vice versa) -- thus $\Gamma''$ projects homeomorphically to $X''$. Finally, on $X''$, $f$ factors as $X''\rightarrow \Gamma''\rightarrow Y$, a composition of two continuous maps. 
    \end{proof}

        \subsection{Uniqueness of the Topology}

    We end this section with an interesting fact that we nevertheless will not need. Recall that a structure can be t-minimal with respect to more than one topology (for example, consider a real closed field with either the order topology or the Sorgenfrey topology). On the other hand, it turns out that there can only be \textit{one} such topology with the independent neighborhood property (at least up to perturbing a finite set) -- and if this topology exists, it is the coarsest t-minimal topology on $\mathcal M$ (again up to a finite set). Thus, if one t-minimal topology on a structure has independent neighborhoods, one might think of that topology as the `canonical' one.

    Recall that we have a fixed topology $\tau$ on $\mathcal M$ with the independent neighborhood property. Now we show:

    \begin{lemma}\label{L: unique top}
        Let $\tau'$ be another $\emptyset$-definable topology on $\mathcal M$ with respect to which $\mathcal M$ is t-minimal.
        \begin{enumerate}
            \item There is a cofinite set $M'\subset M$ so that $\tau'$ refines $\tau$ on $M'$.
            \item If $\mathcal M$ also has the independent neighborhood property with respect to $\tau'$, then there is a cofinite set $M'$ such that $\tau$ and $\tau'$ agree on $M'$.
        \end{enumerate}
    \end{lemma}
    \begin{proof}
        We show (1). Then (2) follows by applying (1) twice (the second time with $\tau$ and $\tau'$ reversed).

        Now to prove (1), we will show that for cofinitely many $a\in M$, every definable $\tau$-neighborhood of $a$ contains a definable $\tau'$-neighborhood of $a$. Since this is a definable property of $a$, it will suffice to prove it for any $a\in M-\acl(\emptyset)$. So fix such an $a$.

        Since $a\in M-\acl(\emptyset)$, we have $\dim(a)=1$. Now let $U$ be a definable $\tau$-neighborhood of $a$. It is harmless to shrink $U$, so by the independent neighborhood property, we may assume $U$ is definable over a tuple $t$ with $\dim(a/t)=1$. By t-minimality, the $\tau'$-boundary of $U$ is finite, and is thus contained in $\acl(t)$. So $a$ does not belong to the $\tau'$-boundary of $U$, and thus $a$ belongs to the $\tau'$-interior of $U$. This shows that $U'$ contains a definable $\tau'$-neighborhood of $a$, as desired.
    \end{proof}

    \begin{remark}
        If $\mathcal M$ has the independent neighborhood property with respect to $\tau$, then one can completely describe $\tau$ on $M-\acl(\emptyset)$ using only the model theory of $\mathcal M$ (with no topology). Let $a\in M-\acl(\emptyset)$. Then by the independent neighborhood property, $a$ has a neighborhood basis consisting of all $t$-definable sets containing $a$, where $t$ ranges over all tuples with $a\notin\acl(t)$.
    \end{remark}

\section{Infinitesimal Neighborhoods}

We now introduce the two-model setup that has become customary in treating infinitesimals in topological settings.

\begin{assumption}
    \textbf{From now until the end of Section 6, we fix $\mathcal N$, an elementary extension of $\mathcal M$ which is sufficiently saturated with respect to $\mathcal M$.}
\end{assumption}

By default (unless stated otherwise), we still perform our model-theoretic analysis in $\mathcal M$. In particular, parameter sets will continue to be small subsets of $\mathcal M^{eq}$.

Following previous works (see e.g. \cite{PetGconf}), we use $\mathcal N$ to discuss infinitesimal elements, and in particular infinitesimal neighborhoods of elements of $\mathcal M$. The idea is to prove statements about $\mathcal M$ by translating them to cleaner statements about infinitesimals in $\mathcal N$.

The present section is divided into four parts. The first subsection makes the framework of infinitesimals precise. There should be no surprises here, as similar frameworks have become well-known in topological settings. The next subsection then rewrites some of our main topological lemmas on d-approximations in terms of infinitesimal neighborhoods. Again, this is a straightforward translation procedure, and we will be brief. 

Finally, the next two subsections give two isolated but interesting results. First, we use infinitesimals to prove a very useful remnant of additivity that holds even if $\mathcal M$ does not satisfy exchange. Then we show that if $\mathcal M$ is saturated in its own cardinality, every infinitesimal neighborhood is the locus of a complete type. 

\subsection{Basics}

We begin with basic definitions:

\begin{definition}\label{D: infinitesimal}
    Let $a\in M^n$ and $b\in N^n$. We say that $b$ is \textit{$\mathcal M$-infinitesimally close to $a$} if $b$ belongs to every $\mathcal M$-definable neighborhood of $a$. The set of all elements $\mathcal M$-infinitesimally close to $a$ is called the \textit{infinitesimal neighborhood of $a$}, and denoted $\mu(a)$.
\end{definition}

Clearly, in the setup of Definition \ref{D: infinitesimal}, $b$ is infinitesimally close to $a$ if and only if each coordinate of $b$ is infinitesimally close to the corresponding coordinate of $a$. That is, we have $\mu(a,b)=\mu(a)\times\mu(b)$ for any tuples $a$ and $b$ from $\mathcal M$.

We also work with infinitesimal neighborhoods over parameter sets. In this case, instead of intersecting open sets, we intersect d-approximations:

\begin{definition}\label{D: inf neigh}
    Let $a\in M^n$ and $A$ a parameter set. The \textit{infinitesimal neighborhood of $a$ over $A$}, denoted $\mu(a/A)$, is the intersection of all $\mathcal M$-definable d-approximations of $\tp(a/A)$ at $a$, viewed as a type-definable subset of $N^n$.
    \end{definition}

    \begin{remark}
        Note that $\mu(a)$ is NOT the same as $\mu(a/\emptyset)$.
    \end{remark}

    Thus, in the above setting, $\mu(a/A)$ is type-definable over $M$. Recall that the dimension theory from $\mathcal M$ extends to type-definable sets, and the dimension of a type-definable set is the smallest dimension of any definable set containing it. It follows easily that $\dim(\mu(a/A))=\dim(a/A)$ for any $a$ and $A$ as above.

    We now give two equivalent definitions of $\mu(a/A)$, matching the corresponding notion in other settings (e.g. \cite{PetGconf}):

    \begin{lemma}\label{L: mu characterization}
        Let $a\in M^n$, $b\in N^n$, and $A$ a parameter set. The following are equivalent:
        \begin{enumerate}
            \item $b\in\mu(a/A)$.
            \item $b\in\mu(a)$ and $b$ belongs to some $\mathcal M$-definable d-approximation of $\tp(a/A)$ at $a$.
            \item $b\in\mu(a)$ and $b\models\tp(a/A)$.
            \end{enumerate}
    \end{lemma}
    \begin{proof}
        (1)$\Rightarrow$(2): Assume $b\in\mu(a/A)$. Let $X$ be any $\mathcal M$-definable d-approximation of $\tp(a/A)$ at $a$. By definition, $b\in X$. So to prove (2), it suffices to show that $b\in\mu(a)$. For this, let $U$ be any $\mathcal M$-definable neighborhood of $a$. By Lemma \ref{L: local dimension}, $X\cap U$ is a d-approximation of $\tp(a/A)$ at $a$, and is also $\mathcal M$-definable. Thus $b\in X\cap U\subset U$ as desired.

        (2)$\Rightarrow$(3): Assume (2). We need to show that $b\models\tp(a/A)$. For this, let $X$ be $A$-definable with $a\in X$. By (2), there is an $\mathcal M$-definable d-approximation $Y$ of $\tp(a/A)$ at $a$ with $b\in Y$. Then $$\operatorname{germ}_a(X)\geq\operatorname{germ}(a/A)=\operatorname{germ}_a(Y),$$ so there is a neighborhood $U$ of $a$ with $Y\cap U\subset X\cap U$. Since $X$ and $Y$ are $\mathcal M$-definable, we can choose $U$ to also be $\mathcal M$-definable. By (2), $b\in\mu(a)$, and thus $b\in U$. So $b\in Y\cap U\subset X\cap U\subset X$, as desired.

        (3)$\Rightarrow$(1): Assume (3). Fix an $A$-definable d-approximation $Y$ of $\tp(a/A)$ at $a$. To show (1), let $X$ be any $\mathcal M$-definable d-approximation of $\tp(a/A)$ at $a$. Then $\operatorname{germ}_a(X)=\operatorname{germ}_a(Y)$, so there is a neighborhood $U$ of $a$ so that $X\cap U=Y\cap U$. Since $X$ and $Y$ are $\mathcal M$-definable, we can choose $U$ to also be $\mathcal M$-definable. By (3), $b\in Y\cap U$, and thus $b\in X\cap U\subset X$ as desired.
    \end{proof}

    We also point out that the term `neighborhood' is justified: that is, $\mu(a/A)$ is actually a neighborhood of $a$ in an appropriate sense: 

    \begin{lemma}\label{L: mu open}
        Let $a\in M^n$ and $A$ a parameter set. 
        \begin{enumerate}
            \item If $X\subset M^n$ is an $\mathcal M$-definable d-approximation of $\tp(a/A)$ at $a$, then $\mu(a/A)$ is open in $X$.
            \item In particular, $\operatorname{germ}_a(\mu(a/A))=\operatorname{germ}(a/A)$.
            \end{enumerate}
    \end{lemma}
    \begin{proof} It is clear that (1) implies (2), because any d-approximation of $\tp(a/A)$ at $a$ realizes $\operatorname{germ}(a/A)$. We show only (1). So let $X$ be an $\mathcal M$-definable d-approximation of $\tp(a/A)$ at $a$. Then Lemma \ref{L: mu characterization}(2) gives that $\mu(a/A)$ is the intersection of a small number of open subsets of $X$ -- namely those of the form $X\cap U$ where $U$ is an $\mathcal M$-definable open neighborhood of $a$ (here `small' is in the sense of $\mathcal N$, not $\mathcal M$). (1) now follows by applying Fact \ref{F: int of opens} in the structure $\mathcal N$.
    \end{proof}

        \subsection{Independence and Generic Continuity Revisited}

    We now rewrite key lemmas from Section 4 using infinitesimal neighborhoods. We will be terse, as most of the statements we give are straightforward exercises.

    We begin with independence:

    \begin{lemma}\label{L: germ fiber equality}
        Let $a\in M^n$, and $A\subset B\subset\mathcal M^{eq}$.
        \begin{enumerate}
            \item $\mu(a/B)\subset\mu(a/A)$.
            \item $\mu(a/B)=\mu(a/A)$ if and only if $a$ is independent from $B$ over $A$.
        \end{enumerate}
    \end{lemma}

    \begin{proof}
        \begin{enumerate}
            \item This follows by Lemma \ref{L: mu characterization}(3), because every realization of $\tp(a/B)$ also realizes $\tp(a/A)$.
            \item First suppose $a$ is independent from $B$ over $A$. Then by Lemma \ref{L: weakly ind}(3), $\tp(a/A)$ and $\tp(a/B)$ have the same d-approximations at $a$, which by definition gives that $\mu(a/A)=\mu(a/B)$. Conversely, if $\mu(a/A)=\mu(a/B)$ then $$\dim(a/A)=\dim(\mu(a/A))=\dim(\mu(a/B))=\dim(a/B),$$ so $a$ is independent from $B$ over $A$.
        \end{enumerate}
    \end{proof}

    \begin{lemma}\label{L: mu strong ind} Let $a\in M^m$, $b\in M^n$, and $a\subset\mathcal M^{eq}$.
    \begin{enumerate}
    \item $\mu(ab/A)\subset\mu(a/A)\times\mu(b/A)$; thus the natural projections map $\mu(ab/A)$ into $\mu(a/A)$ and $\mu(b/A)$, respectively. 
    \item $\mu(ab/A)=\mu(a/A)\times\mu(b/A)$ if and only if $a$ and $b$ are symmetrically independent over $A$.
\end{enumerate}
\end{lemma}

\begin{proof}
    \begin{enumerate}
        \item By Lemma \ref{L: mu characterization}(3). Namely, suppose $(x,y)\in\mu(ab/A)$. Then $(x,y)\in\mu(a,b)$, so $x\in\mu(a)$ and $y\in\mu(b)$. Moreover, we have $(x,y)\models\tp(ab/A)$, and thus $x\models\tp(a/A)$ and $y\models\tp(b/A)$. Thus, by Lemma \ref{L: mu characterization}(3), $x\in\mu(a/A)$ and $y\in\mu(b/A)$.
        \item First suppose $a$ and $b$ are symetrically independent over $A$. By (1), we only need to show that if $x\in\mu(a/A)$ and $y\in\mu(b/A)$, then $(x,y)\in\mu(ab/A)$. It is automatic that $(x,y)\in\mu(a,b)$, since $x\in\mu(a)$ and $y\in\mu(b)$. So by Lemma \ref{L: mu characterization}(2), it suffices to show that $(x,y)$ belongs to some $\mathcal M$-definable d-approximation of $\tp(ab/A)$ at $(a,b)$. Now by Lemma \ref{L: strongly ind}, there are $\mathcal M$-definable d-approximations $X$ of $\tp(a/A)$ at $a$ and $Y$ of $\tp(b/A)$ at $b$, so that $X\times Y$ is a d-approximation of $\tp(ab/A)$ at $(a,b)$. Since $x\in\mu(a/A)$ and $y\in \mu(b/A)$, we have $x\in X$ and $y\in Y$, thus $(x,y)\in X\times Y$ as desired.

        Now assume $\mu(ab/A)=\mu(a/A)\times\mu(b/A)$. Then $$\dim(\mu(ab/A))=\dim(\mu(a/A))+\dim(\mu(b/A)),$$ so $\dim(ab/A)=\dim(a/A)+\dim(b/A)$, and thus $a$ and $b$ are symmetrically independent over $A$.
    \end{enumerate}
\end{proof}

Now we turn to generic continuity:

\begin{lemma}\label{L: mu generic cont}
    Let $a\in M^m$, $b\in M^n$, and $A$ a parameter set so that $(a,b)$ is additive over $A$. Then: 
    \begin{enumerate}
        \item The projection $\mu(ab/A)\rightarrow\mu(b/A)$ is open and surjective.
        \item If $a\in\acl(Ab)$, then the projection $\mu(ab/A)\rightarrow\mu(b/A)$ is a homeomorphism.
        \item In particular, if $a$ and $b$ are interalgebraic over $A$, then $\mu(ab/A)$ is the graph of a homeomorphism $\mu(a/A)\rightarrow\mu(b/A)$.
    \end{enumerate}
\end{lemma}

\begin{proof}
    \begin{enumerate}
        \item Let $\pi$ be the projection sending $(a,b)$ to $b$. By Lemma \ref{L: d-app open}, there are $\mathcal M$-definable d-approximations $X$ of $\tp(ab/A)$ at $(a,b)$, and $Y$ of $\tp(b/A)$ at $b$, so that $\pi$ induces an open surjective map $X\rightarrow Y$. By Lemma \ref{L: mu open}, $\mu(ab/A)$ and $\mu(b/A)$ are open in $X$ and $Y$, respectively, which implies that the projection $\mu(ab/A)\rightarrow\mu(b/A)$ is also open. 
        
        To show $\mu(ab/A)\rightarrow\mu(b/A)$ is surjective, let $y\in\mu(b/A)$. By compactness, it suffices to show that for any $\mathcal M$-definable d-approximation $Z$ of $\tp(ab/A)$ at $(a,b)$, we have $y\in\pi(Z)$. But by Lemma \ref{L: d-app images}, $\pi(Z)$ contains an $\mathcal M$-definable d-approximation of $\tp(b/A)$ at $b$; thus $\pi(Z)$ contains $\mu(b/A)$, and thus $y\in\pi(Z)$.

        \item Repeat the proof of (1), but starting with d-approximations $X$ and $Y$ so that $X\rightarrow Y$ is moreover injective (as provided by Lemma \ref{L: d-app homeo}).
        \item Apply (2) to both $\mu(ab/A)\rightarrow\mu(a/A)$ and $\mu(ab/A)\rightarrow\mu(b/A)$.
    \end{enumerate}
\end{proof}

    \subsection{The Fiber Germ Equation}

We now give a rather surprising consequence of independent neighborhoods. Namely, we prove a fairly strong additivity-like statement that holds even if full additivity fails. This will be extremely useful throughout the rest of the paper.

Recall that our main uses of additive pairs will ultimately stem from Lemma \ref{L: additive pairs}. Roughly, that lemma said that independence is transitive for additive pairs: if $(a,b)$ is additive over $A$ and independent from $B$ over $A$, then all dimension computations in the additive decomposition of $\dim(ab/A)$ remain unchanged over $B$.

The result of this subsection, then, is that one of the clauses of Lemma \ref{L: additive pairs} (namely Lemma \ref{L: additive pairs}(2)) does not require additivity at all. This seems to be a unique and pleasant feature of the topological setting (compared to other sub-additivite dimension theories).

    \begin{lemma}\label{L: fiber germ equation}
        Let $a$ and $b$ be real tuples, and $A$ a parameter set.
        \begin{enumerate}
            \item $\mu(ab/A)_b=\mu(a/Ab)$. That is, the fiber above $b$ of the projection $\mu(ab/A)\rightarrow\mu(b/A)$ is precisely $\mu(a/Ab)$.
            \item $\operatorname{germ}(ab/A)_b=\operatorname{germ}(a/Ab)$ -- that is, the fiber germ of $\operatorname{germ}(ab/A)$ above $b$ is $\operatorname{germ}(a/Ab)$.
            \item If $B\supset A$ and $\dim(ab/B)=\dim(ab/A)$, then $\dim(a/Bb)=\dim(a/Ab)$.
            \end{enumerate}
    \end{lemma}

    \begin{proof} (1) follows from Lemma \ref{L: mu characterization}(3): we have $x\in\mu(a/Ab)$ if and only if $x\in\mu(a)$ and $x\models\tp(a/Ab)$; and this is equivalent to the assertion that $(x,b)\in\mu(a,b)$ and $(x,b)\models\tp(ab/A)$. 
    
    (2) now follows from (1) by applying Lemma \ref{L: mu open}(2) to both sides of the equation. That is, by Lemma \ref{L: mu open}(2), we can rewrite the desired statement as $\operatorname{germ}_a(\mu(ab/A)_b)=\operatorname{germ}_a(\mu(a/Ab))$. But (1) gives that $\mu(ab/A)_b=\mu(a/Ab)$, so of course the two sets have the same germ at $a$.
    
     Now for (3), assume $B\supset A$ and $\dim(ab/B)=\dim(ab/A)$. By Lemma \ref{L: weakly ind}, $\operatorname{germ}(ab/B)=\operatorname{germ}(ab/A)$. Now applying (2) gives $$\operatorname{germ}(a/Bb)=(\operatorname{germ}(ab/B))_b=(\operatorname{germ}(ab/A))_b=\operatorname{germ}(a/Ab).$$ Then by Lemma \ref{L: weakly ind} again, $\dim(a/Bb)=\dim(a/Ab)$.
     \end{proof}

     In fact, Lemma \ref{L: fiber germ equation} implies:

     \begin{corollary}\label{C: fiber germ}
         Let $a$ and $b$ be real tuples, and $A$ a parameter set. If $(x,y)\in\mu(ab/A)$, then the fiber $\mu(ab/A)_y$ realizes $\operatorname{germ}(x/Ay)$.
     \end{corollary}
     \begin{proof}
         Let $X$ be an $A$-definable d-approxiation of $\tp(ab/A)$ at $(a,b)$. By Lemma \ref{L: fiber germ equation}(2), the fiber $X_b$ realizes $\operatorname{germ}(a/Ab)$. Now if $(x,y)\in\mu(ab/A)$, then $\tp(xy/A)=\tp(ab/A)$. So $X_y$ must realize $\operatorname{germ}(x/Ay)$. Since $\mu(ab/A)$ is open in $X$, the same germ is also realizeed by $\mu(ab/A)_y$.
     \end{proof}

\subsection{Infinitesimal Neighborhoods as Complete Types}

    We end this section with another interesting fact about infinitesimal neighborhoods. Namely, in the case that $\mathcal M$ is saturated in its own cardinality, Lemma \ref{L: mu characterization} (in combination with Lemma \ref{L: d-app locality}) implies that every infinitesimal neighborhood is, in fact, the locus of a complete type. We will not explicitly use this fact, as we might not have access to saturated models. Nevertheless, this seems interesting in its own right, so we include a proof sketch. Note that if we had a saturated model, some of our arguments could be expressed a bit more elegantly using this fact.
    
    \begin{lemma}
        Assume $\mathcal M$ is saturated. Let $a\in M^n$ and $A$ a parameter set. Then there is a set $B\subset M^{eq}$ such that $A\subset B$, $\dim(a/B)=\dim(a/A)$, and $\mu(a/A)$ is precisely the set of realizations of $\tp(a/B)$ (note that $B$ will be small in the sense of $\mathcal N$, but not in the sense of $\mathcal M$).
    \end{lemma}
    \begin{proof}
        Let $\{X_\alpha:\alpha<\kappa\}$  enumerate  all $\mathcal M$-definable d-approximations of $\tp(a/A)$ at $a$. We may assume that $\kappa$ is a cardinal and $\kappa\leq|M|$. We will inductively build $b_{\alpha}$ and $Y_{\alpha}$, for $\alpha<\kappa$, such that:
        \begin{enumerate}
            \item $b_\alpha$ is a finite tuple from $\mathcal M$.
            \item $Y_{\alpha}\subset X_{\alpha}$ is an $Ab_{\alpha}$-definable d-approximation of $\tp(a/A)$ at $a$.
            \item $\dim(a/A_{\alpha}b_{\alpha})=\dim(a/A_{\alpha})$, where $A_{\alpha}=A\cup\{b_{\beta}:\beta<\alpha\}$. \end{enumerate}

            Indeed, to build $b_{\alpha}$ and $Y_{\alpha}$, one just needs to apply Lemma \ref{L: d-app locality} to $X_\alpha$ with parameter set $A_{\alpha}$. The key point making this work is that $|A_{\alpha}|<|M|$ for all $\alpha$, hence (as $\mathcal M$ is saturated) such a $b_{\alpha}$ can be found inside $\mathcal M$.
            %$|A_{\alpha}|$ is smaller than the level of saturation of $\mathcal M$, 
            
            Now suppose we have built $b_{\alpha}$ and $Y_{\alpha}$ for al $\alpha$. Let $B=A_{\kappa}$ be $A$ together with all $b_{\alpha}$ for $\alpha<\kappa$ -- equivalently, the union of all $A_{\alpha}$ for $\alpha<\kappa$. It follows by induction, that $\dim(a/A_{\alpha})=\dim(a/A)$ for all $\alpha\leq\kappa$ (for the limit step, use that any drop in dimension is witnessed by a formula with finitely many parameters). In particular, $\dim(a/B)=\dim(a/A)$. 
            
            Now we show that $\mu(a/A)$ is the set of realizations of $\tp(a/B)$. First assume $a'\models\tp(a/B)$. Then for each $\alpha<\kappa$ we have $\tp(a'/Ab_{\alpha})=\tp(a/Ab_{\alpha})$, thus $a'\in Y_{\alpha}\subset X_{\alpha}$. So $a'$ belongs to all $\mathcal M$-definable d-approximations of $\tp(a/A)$ at $a$, and thus $a'\in\mu(a/A)$.
            
            Conversely, suppose $a'\in\mu(a/A)$. If $\tp(a'/B)\neq\tp(a/B)$, then there is $\alpha<\kappa$ so that $\tp(a'/A_{\alpha})\neq\tp(a/A_{\alpha})$. But $A_{\alpha}$ is a small set in $\mathcal M^{eq}$, and $\dim(a/A_{\alpha})=\dim(a/A)$; so we have $\mu(a/A_{\alpha})=\mu(a/A)$. Thus $a'\in\mu(a/A_{\alpha})$, and thus $\tp(a'/A_{\alpha})=\tp(a/A_{\alpha})$, a contradiction.
    \end{proof}

\section{Topological 1-basedness}

Throughout this section, we continue to fix a sufficiently saturated t-minimal structure $\mathcal M$ with the independent neighborhood property, as well as a highly saturated elementary extension $\mathcal N$ where we treat infinitesimal neighborhoods of elements of $\mathcal M$.

The purpose of this section is to introduce our notion of topological 1-basedness and study its basic properties. The first two subsections give some motivation and equivalent definitions. The next three subsections show that topological 1-basedness is preserved under topological reducts and expansions by constants, and that (similarly to strongly minimal theories) if exchange holds then topological 1-basedness can be checked at the level of 2-types. Finally, the last three subsections give examples of topologically 1-based and non-1-based structures.

\subsection{The Definition}

We begin by defining topological 1-basedness.

\begin{definition}\label{D: 1-based type}
    Let $a\in M^m$, $b\in M^n$, and $A$ a parameter set. We say that $\tp(a/Ab)$ is \textit{topologically 1-based over $A$} if $\dim(b/Aa)=\dim(b/A\operatorname{germ}(a/Ab))$. If $A=\emptyset$, we simply say that $\tp(a/b)$ is topologically 1-based.
\end{definition}

\begin{definition}\label{D: 1-base}
    We say that $\mathcal M$ is \textit{topologically 1-based} if $\tp(a/b)$ is 1-based (over $\emptyset$) for all real tuples $a$ and $b$.
\end{definition}

\begin{remark}
    Letting $g:=\operatorname{germ}(a/Ab)$, an equivalent statement in Definition \ref{D: 1-based type} would be that $\dim(b/Aag)=\dim(b/Aa)$, i.e. $b$ is independent from $g$ over $Aa$. This is equivalent because $a\in\dcl(g)$, so $\dim(b/Aag)=\dim(b/Ag)$ is automatic.
\end{remark}

\begin{remark} If $\mathcal M$ satisfies exchange, Definition \ref{D: 1-based type} makes sense even if $b$ is an imaginary tuple (that is, in that case the values $\dim(b/Ag)$ and $\dim(b/Aa)$ are well-defined). Without exchange, dimension theory only works for real tuples, and thus we must assume $b$ is a real tuple. Note that if $\mathcal M$ satisfies exchange, one could also allow imaginary $b$ in Definition \ref{D: 1-base}, and this would result in the same notion of topological 1-basedness (this is essentially Lemma \ref{L: 1-based iff dim 0}). 
\end{remark}

Let us motivate the above definitions. As we stated in the introduction, we intend to treat germs as the topological analog of canonical bases. In stability, one would say $\tp(a/b)$ is 1-based if $\operatorname{Cb}(\tp(a/\acl(b)))\in\acl(a)$. In our case, there is no need to pass from $b$ to $\acl(b)$, as Lemma \ref{L: weakly ind} gives that $\tp(a/b)$ and $\tp(a/\acl(b))$ have the same germ at $a$. Thus, the most obvious translation of the definition from stability would simply be that $\operatorname{germ}(a/b)\in\acl(a)$.

Our chosen formulation is a bit more general. Indeed, let $g=\operatorname{germ}(a/b)$. Notice, then, that if $g\in\acl(a)$, then $\dim(b/g)=\dim(b/a)$, which is exactly our definition. On the other hand, the latter property is all we will need, so we have chosen to stick with it. In fact, we do not know whether the two potential definitions are equivalent.

\subsection{Equivalent Formulations}

Recall that if $\mathcal M$ satisfies exchange, one can extend the dimension theory to imaginary sorts (see \cite[\S 3]{Gagelman}) -- with the caveat that there can be infinite zero-dimensional interpretable sets. For example, if $\mathcal M\models \mathrm{ACVF}$, then the set of all balls of a given radius is 0-dimensional (the intuition is that the data of a ball is `locally constant', so the set of balls is topologically `small'). In this light, the following should help clarify our definition:

\begin{lemma}\label{L: 1-based iff dim 0}
    Assume $\mathcal M$ satisfies exchange.
    \begin{enumerate}
        \item Let $a$ and $b$ be real tuples, $A$ a parameter set, and $g=\operatorname{germ}(a/Ab)$. Then $\tp(a/Ab)$ is topologically 1-based over $A$ if and only if $\dim(g/Aa)=0$.
        \item $\mathcal M$ is topologically 1-based if and only if for all real tuples $a$ and parameter sets $A$, we have $\dim(\operatorname{germ}(a/A)/a)=0$.
    \end{enumerate}
\end{lemma}
\begin{proof}
    \begin{enumerate}
        \item For ease of notation, assume $A=\emptyset$. Since $a\in\dcl(g)$, we have $\dim(b/g)=\dim(b/ga)$. Moreover, since $g\in\dcl(ab)$, we have $\dim(ab)=\dim(abg)$. Now combining with additivity, we compute: $$\dim(a)+\dim(b/a)=\dim(ab)=\dim(abg)$$ $$=\dim(a)+\dim(g/a)+\dim(b/ga)=\dim(a)+\dim(g/a)+\dim(b/g).$$ Rearranging now gives $$\dim(b/a)=\dim(g/a)+\dim(b/g).$$ In particular, $\dim(g/a)=0$ holds if and only if $\dim(b/a)=\dim(b/g)$, as desired.
        \item The right to left direction is trivial by (1); we show left to right. So assume $\mathcal M$ is topologically 1-based. Let $a$ be a real tuple, $A$ be a parameter set, and $g=\operatorname{germ}(a/A)$. Let $c$ be a finite tuple from $A$ so that $\dim(a/c)=\dim(a/A)$, and thus $\operatorname{germ}(a/c)=g$. Then let $b$ be a real tuple so that $c\in\dcl(b)$. We may choose $b$ to be independent from $a$ over $c$; thus $\operatorname{germ}(a/b)=g$ as well. Then since $\mathcal M$ is topologically 1-based, and by (1), we get that $$\dim(\operatorname{germ}(a/b)/a)=\dim(g/a)=0,$$ which proves (2).
        \end{enumerate}
\end{proof}

%Lemma \ref{L: 1-based iff dim 0} allows, as mentioned above, to extend the definition of topological 1-basedness to types where the parameter $b$ over arbitrary parameter sets (still assuming exchange). 

Returning to the general case, one should think of topologically 1-based types as `locally linear' -- namely, $\tp(a/b)$ is topologically 1-based if in some neighborhood of $b$, any conjugate $b'\equiv_ab$ satisfies $\operatorname{germ}(a/b)=\operatorname{germ}(a/b')$. In some sense, this is the most natural way to adapt the stability-theoretic definition using canonical bases to the topological setting. Namely, in a stable theory, the asserion that $\operatorname{Cb}(\operatorname{stp}(a/b))\in\acl(a)$ really says that $\operatorname{Cb}(\operatorname{stp}(a/b'))$ is constant on a `big' set of realizations $b'\models\tp(b/a)$ (here `big' means non-forking over $a$). In the topological setting, we make exactly the same statement but interpret `big' as `on a neighborhood of $b$ in $\tp(b/a)$' (see Lemma \ref{L: 1-based iff constant germ} below). In general, passing from `$\operatorname{acl}^{eq}$' to `constant on a neighborhood' seems to be the correct philosophy in finding topological analogs of stability-theoretic phenomena.

Lemma \ref{L: 1-based iff constant germ} makes the above discussion precise:

\begin{lemma}\label{L: 1-based iff constant germ}
    Let $a$ and $b$ be real tuples, and $A$ a parameter set. The following are equivalent:
    \begin{enumerate}
        \item $\tp(a/Ab)$ is topologically 1-based over $A$.
        \item The map $y\mapsto\operatorname{germ}(a/Ay)$ is constant for $y\in\mu(b/Aa)$.
    \end{enumerate}
\end{lemma}
\begin{proof} We may assume $A=\emptyset$.
    Let $g=\operatorname{germ}(a/b)$. First assume $\tp(a/b)$ is topologically 1-based, and let $b'\in\mu(b/a)$. We show that $g=\operatorname{germ}(a/b')$. Indeed, by 1-basedness, $\dim(b/a)=\dim(b/ga)$, thus $\mu(b/a)=\mu(b/ga)$. So $b'\in\mu(b/ga)$, and thus $b'\models\tp(b/ga)$. So $\tp(abg)=\tp(ab'g)$, and since $g=\operatorname{germ}(a/b)$, we conclude that also $g=\operatorname{germ}(a/b')$.

    Conversely, suppose $y\mapsto\operatorname{germ}(a/y)$ is constant on $\mu(b/a)$. Then for any $b'\in\mu(b/a)$ we have $\tp(abg)=\tp(ab'g)$. Now since $\dim(\mu(b/a))=\dim(b/a)$, there is an element $b'\in\mu(b/a)$ with $\dim(b'/abg)=\dim(b/a)=\dim(b'/a)$. It follows that $\dim(b'/ag)=\dim(b'/g)=\dim(b'/a)$ as well, and since $\tp(abg)=\tp(ab'g)$, this implies $\dim(b/g)=\dim(b/a)$.  
\end{proof}

\subsection{Expansion by Constants}

Recall that the class of locally modular strongly minimal theories is preserved under naming constants and taking reducts (as is the analogous class of `linear' o-minimal theories). In the next two subsections, we prove the same for topological 1-basedness.

\begin{lemma}\label{L: constants}
Let $C$ be a paraneter set, and let $\mathcal M_C$ be $\mathcal M$ expanded by naming the elements of $C$. Then $\mathcal M_C$ is topologically 1-based if and only if $\mathcal M$ is topologically 1-based.
\end{lemma}
\begin{proof}

Before starting the proof, we should note that $\mathcal M_C$ is also t-minimal with the same topology. Moreover, $\mathcal M_C$ also has independent neighborhoods. Indeed, given any $a$ and $A$, the independent neighborhood property for $\mathcal M$ gives us arbitrarily small neighborhoods $U$ of $a$ defined over parameters $t$ with $\dim(a/CAt)=\dim(a/CA)$; and this is equivalent to the independent neighborhood property for $\mathcal M_C$. Thus, for the rest of the proof, we can use all of the results collected to this point on independent neighborhoods.

We now proceed with the proof. First suppose $\mathcal M_C$ is topologically 1-based. To show $\mathcal M$ is topologically 1-based, let $a$ and $b$ be real tuples. We wish to show that $\tp(a/b)$ is topologically 1-based. As this only depends on $\tp(ab)$, we may replace $(a,b)$ with any other realization of the same type. In particular, we may assume that $(a,b)$ is independent from $C$, i.e. $\dim(ab/C)=\dim(ab)$. By Lemma \ref{L: fiber germ equation}, this implies that $\dim(a/Cb)=\dim(a/b)$ and $\dim(b/Ca)=\dim(b/a)$.

Now let $g=\operatorname{germ}(a/b)$. Since $\dim(a/Cb)=\dim(a/b)$, we also have $g=\operatorname{germ}(a/Cb)$. Then the fact that $\mathcal M_C$ is topologically 1-based gives $\dim(b/Cg)=\dim(b/Ca)$. But now we compute:

$$\dim(b/g)\geq\dim(b/Cg)=\dim(b/Ca)=\dim(b/a),$$ and this implies that $\tp(a/b)$ is topologically 1-based.

Now assume $\mathcal M$ is topologically 1-based; we show that $\mathcal M_C$ is topologically 1-based. To do this, fix real tuples $a$ and $b$, and set $g=\operatorname{germ}(a/Cb)$. We want to show that $\tp(a/Cb)$ is topologically 1-based over $C$ -- i.e. that $\dim(b/Cg)=\dim(b/Ca)$.

Let $c$ be a finite tuple from $C$ such that $\dim(a/Cb)=\dim(a/cb)$, and $\dim(b/Cg)=\dim(b/cg)$. Then let $d$ be a real tuple such that $c\in\dcl(d)$. Replacing $d$ with an independent realization of $\tp(d/c)$, we may assume that $ab$ and $d$ are symmetrically independent over $c$ (for example, this is, essentially, \cite[Corollary 2.16]{t-minimal-johnson}). 

Let $g=\operatorname{germ}(a/Cb)$. By the choice of $c$ we also have $g=\operatorname{germ}(a/cb)$. Now by the choice of $d$ we have $\dim(ab/d)=\dim(ab/c)$, and thus Lemma \ref{L: fiber germ equation}(2) gives $\operatorname{germ}(a/bd)=\operatorname{germ}(a/bc)=g$. Since $\mathcal M$ is topologically 1-based, it follows that $\dim(bd/g)=\dim(bd/a)$. By Lemma \ref{L: fiber germ equation}(3), this implies $\dim(b/dg)=\dim(b/da)$.

Next, recall that $\dim(ab/d)=\dim(ab/c)$. By Lemma \ref{L: fiber germ equation}(3) again, this implies $\dim(b/da)=\dim(b/ca)$. Then, combining with the previous paragraph, we have:

$$\dim(b/Cg)=\dim(b/cg)\geq\dim(b/dg)$$ $$=\dim(b/da)=\dim(b/ca)\geq\dim(b/Ca),$$ which implies that $\tp(a/Cb)$ is topologiclly 1-based over $C$.
\end{proof}

\subsection{Reducts}\label{ss: reducts} Now we show that topological 1-basedness is preserved under reducts in an appropriate sense. There is a subtlety here in that one needs to preserve the topology and independent neighborhoods for this to even make sense. Thus, we define:

\begin{definition}
    An \textit{admissible reduct of $\mathcal M$} is a structure $\mathcal M^-$ such that:
    \begin{enumerate}
        \item The universe of $\mathcal M^-$ is $\mathcal M$.
        \item Every (parametrically) definable set in $\mathcal M^-$ is also parametrically definable in $\mathcal M$.
        \item With respect to the same topology $\tau$, $\mathcal M^-$ is also t-minimal with the independent neighborhood property. 
    \end{enumerate}
\end{definition}

In all known examples of t-minimal structures with independent neighborhoods, the independent neighborhoods are witnessed by a fixed definable basis. This means there is a $\emptyset$-definable basis $\{U_t\}$ such that for all $a\in M^n$ and $A$, there are arbitrarily small $U_{t_1}\times\dots\times U_{t_n}$ containing $a$ and satisfying $\dim(a/At_1\dots t_n)=\dim(a/A)$. In this case, any reduct of $\mathcal M$ containing the same family $\{U_t\}$ is admissible. In contrast, we do not know whether any reduct defining the same topology is admissible.

Now we show:

\begin{lemma}\label{L: reducts}
    Let $\mathcal M^-$ be an admissible reduct of $\mathcal M$. If $\mathcal M$ is topologically 1-based, then $\mathcal M^-$ is topologically 1-based. 
\end{lemma}
\begin{proof}
    By Lemma \ref{L: constants}, we may assume all basic relations of $\mathcal M^-$ are definable without parameters in $\mathcal M$. Now we begin with:
    \begin{claim}
        If $X\subset M^n$ is $\mathcal M^-$-definable, then $\dim(X)$ is the same in $\mathcal M^-$ and $\mathcal M$.
    \end{claim}
    \begin{proof} Working first in $\mathcal M^-$, let $\{X_i\}$ be a weak cell-decomposition of $X$ (in the sense of \cite[Proposition 2.50]{t-minimal-johnson}). So each $X_i$ is a weak $k_i$-cell for some $k_i$. Then $X_i$ is still a weak $k_i$-cell in the sense of $\mathcal M$, and it follows that in both structures $\dim(X)$ is equal to the largest $k_i$.
    \end{proof}

    The claim also implies the analogous statement for germs: the dimension of a definable germ in $\mathcal M^-$ remains the same when viewed as a definable germ in $\mathcal M$  (on the other hand, the data of which germ is $\operatorname{germ}(a/A)$ for fixed $a$ and $A$ may change).
    
    Now toward a proof of the lemma, fix real tuples $a$ and $b$. We want to show that $\tp_{\mathcal M^-}(a/b)$ is topologically 1-based. We can freely replace $(a,b)$ by any other realization of $\tp_{\mathcal M^-}(ab)$. In particular, by a compactness argument (using the above claim), one finds a realization such that $\dim_{\mathcal M}(ab)=\dim_{\mathcal M^-}(ab)$. Now let $c$ be a basis of $(a,b)$ in the sense of $\mathcal M^-$; then one still has $(a,b)\in\acl_{\mathcal M}(c)$, and it follows that $c$ is still a basis in the sense of $\mathcal M$. In particular, if $X$ is an $(a,b)$-witness for $c$ in the sense of $\mathcal M^-$, then $X$ is also an $(a,b)$-witness for $c$ in the sense of $\mathcal M$. It follows that $\operatorname{germ}_{\mathcal M^-}(a/b)$ and $\operatorname{germ}_{\mathcal M}(a/b)$ have a common realization, and thus are equal. Call this common germ $g$. By Lemma \ref{L: fiber germ equation}(1), $\operatorname{germ}_{\mathcal M^-}(b/a)$ and $\operatorname{germ}_{\mathcal M}(b/a)$ are both the fiber germ $g_a$, and thus $$\dim_{\mathcal M^-}(b/a)=\dim(\operatorname{germ}_{\mathcal M^-}(b/a))=\dim(\operatorname{germ}_{\mathcal M}(b/a))=\dim_{\mathcal M}(b/a).$$

    Now since $\mathcal M$ is topologically 1-based, we have $\dim_{\mathcal M}(b/g)=\dim_{\mathcal M}(b/a)$. If $\tp_{\mathcal M^-}(a/b)$ were not topologically 1-based, then we would have $$\dim_{\mathcal M^-}(b/g)<\dim_{\mathcal M^-}(b/a)=\dim_{\mathcal M}(b/a)=\dim_{\mathcal M}(b/g).$$ So there is some $g$-definable set in $\mathcal M^-$ which contains $b$ and has dimension less than $\dim_{\mathcal M}(b/g)$. But this assertion is still valid in $\mathcal M$, so one obtains $\dim_{\mathcal M}(b/g)<\dim_{\mathcal M}(b/g)$, a contradiction.
\end{proof}

\subsection{Reduction to Plane Curves} A fundamental result in the theory of strongly minimal sets says that if such a set is locally modular, it is also 1-based; equivalently, non-1-basedness can always be witnessed in the plane. A similar result was proven for weakly 1-based geometric theories in \cite[Proposition 2.15]{BerVas}. We do not know whether the analogous statement holds of topological 1-basedness in complete generality -- however it does hold assuming exchange (essentially because of the result of \cite{BerVas}). Here we give a short topological proof.

\begin{lemma} Assume $\mathcal M$ has exchange and is not topologically 1-based. Then there are $a,b\in M$ and a real tuple $t$ so that $\dim(a/t)=\dim(b/t)=\dim(ab/t)=1$ and $\tp(ab/t)$ is not topologically 1-based over $\emptyset$.
\end{lemma}
\begin{proof}
    Let $a\in M^n$ and $t\in M^m$ so that $\tp(a/t)$ is not topologically 1-based over $\emptyset$, and so that the value of $n$ is minimized among all such pairs. Note that $n\geq 2$ (as one can easily check that all 1-types are topologically 1-based).
    
    Let $d=\dim(a/t)$, and write $a=(e_1,\dots,e_d,f_1,\dots,f_{n-d})$ where $e:=(e_1,\dots,e_d)$ is a basis for $a$ over $t$. Clearly $n>d$, since otherwise $\tp(a/t)$ would be topologically 1-based.
    
    By Lemma \ref{L: mu generic cont}, $g=\operatorname{germ}(a/t)$ is realized by the graph of a function $\mu(e)\rightarrow\mu(f_1)\times\dots\times\mu(f_{n-d})$. This function is determined by its coordinate components $\mu(e)\rightarrow\mu(f_i)$, and each of these components realizes $g_i:=\operatorname{germ}(ef_i/t)$. In other words, $g\in\dcl(g_1, \dots, g_{n-d})$. Now by assumption (and Lemma \ref{L: 1-based iff dim 0}) we have $\dim(g/a)>0$, and thus some $\dim(g_i/a)>0$. Then $\tp(e_1,\dots,e_d,f_i/t)$ is not topologically 1-based over $\emptyset$. By the minimality of $n$, we must have $a=(e_1,\dots,e_d,f_i)$, so that $d=n-1$.

    Now assume $n\geq 3$. Let $t'\in\mu(t/a)$ be independent from $t$ over $a$. A short computation (using that $d=n-1$ and $n\geq 3$) yields that $a\notin\acl(tt')$. So we can write $a=(b,c)$ where $b\in M$, $c\in M^{n-1}$, and $\dim(b/tt')=1$.

    Let $X$ be a $\emptyset$-definable d-approximation of $\tp(cbt)$ at $cbt$. By Lemma \ref{L: fiber germ equation}, the fiber $X_{bt}$ realizes $\operatorname{germ}(c/bt)$. Since $\tp(cbt)=\tp(cbt')$, $X_{bt'}$ realizes $\operatorname{germ}(c/bt')$. 
    
    Now by the minimality of $n$, $\tp(c/bt)$ is topologically 1-based, so by Lemma \ref{L: 1-based iff constant germ}, $X_{bt}$ and $X_{bt'}$ agree on some neighborhood $U$ of $c$. By the independent neighborhood property, we may assume $U$ is definable over a tuple $u$ with $\dim(cbtt'/u)=\dim(cbtt')$. By additivity of $(c,btt')$ over $\emptyset$ (i.e. Lemma \ref{L: additive pairs}), $\dim(btt'/u)=\dim(btt')$, and thus $\dim(b/tt'u)=\dim(b/tt')=1$. Thus, there is a neighborhood $V$ of $b$ such that the fibers $X_{vt}$ and $X_{vt'}$ agree on $U$ for all $v\in V$. This means the fibers $X_t$ and $X_t'$ agree on $U\times V$, so that $\operatorname{germ}_a(X_t)=\operatorname{germ}_a(X_{t'})$.

    Now similarly to above, since $X$ is a $\emptyset$-definable d-approximation of $\tp(at)$ at $at$, we have that $X_t$ realizes $\operatorname{germ}(a/t)$, and thus $X_{t'}$ realizes $\operatorname{germ}(a/t')$. It follows that $\operatorname{germ}(a/t)=\operatorname{germ}(a/t')=g$, say. But then $g\in\dcl(at')$, which gives $$\dim(t/g)\geq\dim(t/at')=\dim(t/a).$$ So $\tp(a/t)$ is topologically 1-based afterall, a contradiction.

    So we have $n=2$ and $d=n-1=1$. Write $a=(a_1,a_2)$. If either $\dim(a_i/t)=0$ then $\tp(a/t)$ is easily seen to be topologically 1-based. Thus we have $\dim(a_1/t)=\dim(a_1a_2/t)=\dim(a_2/t)=1$, as desired.
\end{proof}
\begin{remark}\label{R: reduction to n=2}
    The proof above actually shows a bit more, in the sense that the statement can be localized. Namely, let $X\subset M$ be type-definable over $A$, and suppose exchange holds for any $X$-tuples over $A$. Moreover, suppose there are $X$-tuples $a$ and $t$ with $\tp(a/t)$ not topologically 1-based over $A$. Then one can find such tuples $a$ and $t$ so that $a\in X^2$ (and so that the other dimension requirements of the lemma are satisfied). The reason this works is that the extra parameters we had to add to $t$ during the proof came from the original tuple $a$ (so we never introduced anything outside $X$; similarly, $t'\in X$ as well, though this is not needed).
\end{remark}
\subsection{Locally Linear Groups}\label{ss: local linearity}

In the remainder of this section, we give examples of topologically 1-based and non-topologically-1-based structures, with the hope of convincing the reader that we have isolated the correct dividing line. 

Recall that a 1-based group in stability theory has only group-like definable sets: namely, every definable set is either a subgroup, a coset of a subgroup, or a finite Boolean combination of those. In the topological setting, the analogous notion is \textit{local linearity}:
\begin{definition}
    Assume $(M,\cdot,\dots)$ is an expansion of a topological group (with the given topology $\tau$). We say $\mathcal M$ is \textit{locally linear} if for all $a\in M^m$ and $b\in M^n$, there is a definable subgroup $H\leq M^m$ such that $\operatorname{germ}(a/b)=\operatorname{germ}_a(a\cdot H)$.
\end{definition}

The following are all examples of locally linear t-minimal topological groups with independent neighborhoods:

\begin{itemize}
    \item Ordered vector spaces over any ordered division ring.
    \item The reduct $(K,+,v)$ whenever $(K,+,\cdot,v)\models ACVF$. One can make this more general by taking valued vector spaces (in a certain language) over any fixed valued field, but we omit the details as they will not be needed.
    \item If $K$ is a model of ACVF or RCVF, and $O$ is the valuation ring, take $K/O$ with its full induced structure (see \cite[Corollary 3.12, \S 6.3]{HaHaPeVF}).
\end{itemize}

Now we show:

\begin{lemma}
    Assume (in addition to our standing assumption of t-minimality and independent neighborhods) that $\mathcal M=(M,\cdot,\dots)$ is an expansion of a topological group, and is locally linear. Then $\mathcal M$ is topologically 1-based.
\end{lemma}

\begin{proof} We use Lemma \ref{L: 1-based iff constant germ}. Namely, let $a\in M^m$ and $b\in M^n$, and let $b'\in\mu(b/a)$ -- so by Lemma \ref{L: fiber germ equation}, $(a,b')\in\mu(ab/\emptyset)$. We will show that $\operatorname{germ}(a/b)=\operatorname{germ}(a/b')$.

By local linearity, there is an $\mathcal M$-definable subgroup $H\leq M^{m+n}$ so that $\operatorname{germ}(ab/\emptyset)=\operatorname{germ}_{(a,b)}((a,b)\cdot H)$. By Lemmas \ref{L: local dimension} and \ref{L: d-app locality}, there is an $\mathcal M$-definable d-approximation $X$ of $\tp(ab)$ at $(a,b)$ such that:
\begin{enumerate}
    \item $X$ is an open subset of $(a,b)\cdot H$, and
   \item $X$ is $t$-definable for some $t$ with $\dim(ab/t)=\dim(ab/\emptyset)$.
\end{enumerate}

Note that (2) implies $\mu(ab/t)=\mu(ab/\emptyset)$ -- thus $(a,b')\in\mu(ab/t)$, and thus $\tp(abt)=\tp(ab't)$. It follows that $X$ is also a d-approximation of $\tp(ab')$ at $(a,b')$.

Now let $X_b$ and $X_{b'}$ be the fibers in $X$ above $b$ and $b'$ -- so $a\in X_b\cap X_{b'}$. Since $X$ is open in $(a,b)\cdot H$, $X_b$ is open in the corresponding fiber $((a,b)\cdot H)_b$ -- and likewise for $b'$. But since $H$ is a subgroup, the two fibers $((a,b)\cdot H)_b$ and $((a,b)\cdot H)_{b'})$ are equal -- namely, since they both contain $a$, they are both equal to $a\cdot N$ where $N\leq M^m$ is the kernel of the projection $H\rightarrow M^n$. It follows that $\operatorname{germ}_a(X_b)=\operatorname{germ}_a(X_{b'})$.

Finally, by Lemma \ref{L: germ fiber equality} we now have $$\operatorname{germ}(a/b)=\operatorname{germ}(ab/\emptyset)_b=(\operatorname{germ}_{(a,b)}X)_b=\operatorname{germ}_a(X_b)$$ $$=\operatorname{germ}_a(X_{b'})=(\operatorname{germ}_{(a,b')}X)_{b'}=\operatorname{germ}(ab'/\emptyset)_{b'}=\operatorname{germ}(a/b'),$$ which completes the proof.
\end{proof}

In fact, later on we will show that local linearity is equivalent to topological 1-basedness in the case of topological groups (see Theorem \ref{T: loc lin}). This is an analog of the Hrushovsk-Pillay classification of 1-based groups in stability theory (\cite{HrPi87}). We postpone the proof of the full result as it requires more machinery than we have available at this point. 

\subsection{1-basedness and Fields}

On the opposite end of the previous lemma, let us now check that no infinite field is topologically 1-based:

\begin{lemma}\label{L: no field}
    Assume $\mathcal M$ defines an infinite field -- that is, assume there is an infinite definable $F\subset M^n$ for some $n$ with a definable field structure. Then $\mathcal M$ is not topologically 1-based.
\end{lemma}
\begin{proof} Suppose $F\subset M^n$ is an infinite definable field. By Lemma \ref{L: constants}, we may assume $F$ (with its operations) is $\emptyset$-definable. Let $d=\dim(F)>0$. Let $P=L=F^2$, interpreted as the sets of points and (non-vertical) lines in the plane, and let $I\subset X\times T$ be the incidence relation. Then one easily checks the following:

\begin{enumerate}
    \item $\dim(P)=\dim(L)=2d$ and $\dim(I)=3d$. 
    \item If $(p,l)\in I$ with $\dim(pl)=3d$ then $\dim(p)=\dim(l)=2d$ and $\dim(p/l)=\dim(l/p)=d$.
\end{enumerate}

Now fix $(p,l)\in I$ as above, so $\dim(pl)=3d$. Let $g=\operatorname{germ}(p/l)$. If $\mathcal M$ were topologically 1-based, we would have $$\dim(l/gp)=\dim(l/p)=d>0.$$ So there is $l'\neq l$ with $\tp(plg)=\tp(pl'g)$, and thus $g=\operatorname{germ}(p/l')$. It follows that $\tp(p/l)$ and $\tp(p/l')$ have the same d-approximations at $x$, which gives that $\mu(p/l)=\mu(p/l')$ and $\dim(\mu(p/l))=d>0$. So there is $p'\neq p$ with $p'\in\mu(p/l)\cap\mu(p/l')$. Thus $\tp(pl)=\tp(p'l)$ and $\tp(pl')=\tp(p'l')$. In particular, $p$ and $p'$ are distinct intersection points of the lines $l$ and $l'$, which is a contradiction. 
\end{proof}

\begin{remark} The above lemma fails if we replace `definable field' with `interpretable field'. Indeed, for $K\models ACVF$, we already saw that $K/O$ is locally linear, thus topologically 1-based; but it interprets a copy of the residue field $k$.
\end{remark}

\subsection{The o-minimal Case}

We conclude this section by pointing out that topological 1-basedness successfully captures the dividing line of the trichotomy for o-minimal structures.

Recall that every point in an o-minimal structure is either trivial, an element of a group interval with linear structure, or an element of a real closed field interval (see the main result of \cite{PeStTricho}). Call an o-minimal structure \textit{linear} if only the first two cases occur. In \cite{BerVas}, Berenstein and Vassiliev generalize linearity in o-minimal structures to a general notion of \textit{weakly 1-based} for geometric theories (including all t-minimal theories with exchange): a geometric theory is weakly 1-based if for all real tuples $a$ and sets $B$, there is $a'\models\tp(a/B)$ so that $\dim(a/a')=\dim(a/Ba')=\dim(a/B)$.

Now return to our fixed t-minimal $\mathcal M$ (and elementary extension $\mathcal N$) with independent neighborhoods. If $\mathcal M$ has exchange, then $\mathcal M$ is also geometric. As it turns out, in this case $\mathcal M$ is topologically 1-based in our sense if and only if it is weakly 1-based as a geometric structure. We will not fully prove this until the next section; for now, we give one implication in general, and the other implication only in the o-minimal case.

\begin{lemma}\label{L: w1b to 1b}
    Suppose $\mathcal M$ has exchange. If $\mathcal M$ is weakly 1-based, then $\mathcal M$ is topologically 1-based.
\end{lemma}
\begin{proof}
    We freely use the extension of dimension theory to imaginaries, given by (\cite{Gagelman}). Fix real tuples $a$ and $b$, let $d=\dim(a/b)$, and let $g=\operatorname{germ}(a/b)$. By Lemma \ref{L: 1-based iff dim 0}, it suffices to show that $\dim(g/a)=0$.
    
    Since $\mathcal M$ is weakly 1-based, there is $a'\models\tp(a/b)$ with $\dim(a/a')=\dim(a/ba')=d$. Let $a''$ be an independent realization of $\tp(a'/ab)$ -- i.e. $\tp(aba')=\tp(aba'')$ and $\dim(a''/aba')=d$. It follows easily that $\dim(a''/a)=\dim(a''/aa')=d$, so that $a'$ and $a''$ are independent over $a$. 

    Now since $\dim(a/ba')=\dim(a/b)=d$, we have $\operatorname{germ}(a/ba')=g$. Then since $\dim(a/a')=\dim(a/ba')$, we also have $\operatorname{germ}(a/a')=g$. Similarly, $\operatorname{germ}(a/a'')=g$. It follows that $g\in\dcl(aa')\cap\dcl(aa'')$; and since $a'$ and $a''$ are independent over $a$, it follows from additivity that $\dim(g/a)=0$ as desired. 
\end{proof}

\begin{corollary}\label{C: omin}
    Assume $\mathcal M$ is o-minimal. Then $\mathcal M$ is topologically 1-based if and only if it is linear.
\end{corollary}
\begin{proof} If $\mathcal M$ is topologically 1-based, Lemma \ref{L: no field} gives that it does not define an infinite field, and thus is linear. If $\mathcal M$ is linear, then by \cite[\S 3.2, Theorem 2.20]{BerVas} it is weakly 1-based, so by Lemma \ref{L: w1b to 1b} it is topologically 1-based.
\end{proof}

\section{The Structure Theorem for Topologically 1-based Types}

In this section, we prove our main topological structure theorem for topologically 1-based types. Roughly speaking, we show that such types look like actual partitions of an ambient space, at least after zooming into an infinitesimal locus. This result and proof are inspired by arguments from Peterzil's group construction in linear o-minimal theories \cite{Pe}.

Before presenting the main result, we give two technical lemmas in the next two subsections.

\subsection{Additive Extensions}

Our first lemma concerns taking independent extensions of types. Roughly, we want to know that we can always find independent realizations of a type without introducing new failures of additivity:

\begin{lemma}\label{L: additive extension}
    Let $a$ and $b$ be real tuples, and $A$ a parameter set. 
    \begin{enumerate}
        \item There is $c\models\tp(b/a)$ such that $\dim(c/Aab)=\dim(b/Aa)$ and $(c,ab)$ is additive over $A$.
        \item For any $c$ as in (1), and any $B\supset A$ with $\dim(abc/B)=\dim(abc/A)$, we have that $b$ and $c$ are symmetrically independent over $Ba$.
        \item If $\tp(a/Ab)$ is topologically 1-based over $A$, then we can choose $c$ as in (1) so that in addition $\operatorname{germ}(a/Ab)=\operatorname{germ}(a/Ac)$. 
    \end{enumerate}
\end{lemma}
\begin{proof}
\begin{enumerate}
    \item Let $e$ be a basis for $b$ over $Aa$. Then $e$ is independent over $Aa$, so also over $A$, and thus $\dim(e/A)=\dim(e/Aa)=\dim(b/Aa)$. Note that this implies $\mu(e/Aa)=\mu(e/A)$, which we use below.

Next, note that $\dim(\mu(ae/A))=\dim(ae/A)=\dim(ab/A)$, and $\dim(\mu(e/A))=\dim(e/A)=\dim(b/Aa)$. So $X=\mu(ae/A)\times\mu(e/A)$ is a type-definable set of dimension $\dim(ab/A)+\dim(b/Aa)$. By compactness, we can find elements of $X$ of maximal dimension over $A$. Our main goal is to show that we can take the first coordinate of such an element to be $(a,e)$ itself. 

\begin{claim} There is $f'\in\mu(e/A)$ such that $\dim(aef'/A)\geq\dim(X)=\dim(ab/A)+\dim(b/Aa)$.
\end{claim}
\begin{proof}
    Let $Y$ be any d-approximation of $\tp(e/A)$ at $e$. By compactness, it will suffice to find $y\in Y$ so that $\dim(aey/A)\geqq\dim(X)$ (noting that this last inequality is a type-definable condition in $y$).

    By Lemma \ref{L: GenOS}, we may assume $Y$ is definable over $At$ where $\dim(ae/At)=\dim(ae/A)$, and thus $\mu(ae/At)=\mu(ae/A)$. Now since $Y$ is a d-approximation, we have $\dim(Y)=\dim(e/A)=\dim(\mu(e/A))$. So $\dim(\mu(ae/A)\times Y)=\dim(X)$, and thus there is $(a',e',y')\in\mu(ae/A)\times\mu(e/A)$ with $\dim(a'e'y'/A)\geq\dim(X)$. So $(a',e')\in\mu(ae/A)=\mu(ae/At)$, and thus $\tp(a'e't/A)=\tp(aet/A)$. So there is $y$ with $\tp(a'e'y't/A)=\tp(aeyt/A)$. Then $\tp(y/At)=\tp(y'/At)$, implying $y\in Y$; and moreover $\dim(aey/A)=\dim(a'e'y'/A)\geq\dim(X)$.    
\end{proof}

Fix $f'$ satisfying the conclusion of the claim: i.e., $f'\in\mu(e/A)$ and $\dim(aef'/A)\geq\dim(X)$. Since $\mathcal M$ is sufficiently saturated, there is $f$ from $\mathcal M$ so that $\tp(aef/A)=\tp(aef'/A)$.

    Now since $f'\in\mu(e/A)=\mu(e/Aa)$, we have $\tp(ae/A)=\tp(af'/A)=\tp(af/A)$. So there is $c$ from $\mathcal M$ so that $\tp(abe/A)=\tp(acf/A)$. Thus $c\models\tp(b/Aa)$. Moreover, $f$ is also a basis for $c$ over $Aa$; it follows that $bc$ is interalgebraic with $ef$ over $Aa$, so that $$\dim(abc/A)=\dim(aef/A)=\dim(aef'/A)\geq\dim(ab/A)+\dim(b/Aa)$$ $$=\dim(ab/A)+\dim(c/Aa)\geq\dim(ab/A)+\dim(c/Aab).$$ The reverse inequality is automatic by sub-additivity; thus, all inequalities above are equalities. We conclude that $(c,ab)$ is additive over $A$, and that $\dim(c/Aab)=\dim(c/Aa)=\dim(b/Aa)$. Thus, we have now established (1). 
    \item Let $c$ be as in (1), and let $B\supset A$ so that $\dim(abc/B)=\dim(abc/A)$. By Lemma \ref{L: fiber germ equation}(3), we get $$\dim(c/Bab)=\dim(c/Aab)=\dim(b/Aa)=\dim(c/Aa)\geq\dim(c/Ba).$$ So $c$ is independent from $b$ over $Ba$.

    On the other hand, we compute: $$\dim(abc/B)=\dim(abc/A)=\dim(ab/A)+\dim(c/Aab)$$ $$=\dim(ac/A)+\dim(b/Aa)\geq\dim(ac/B)+\dim(b/Bac)$$ $$\geq\dim(abc/B).$$ So all inequalities above must be equalities, and this implies $$\dim(b/Bac)=\dim(b/Aa)\geq\dim(b/Ba).$$ Thus $b$ is also independent from $c$ over $Ba$, which proves (2).
    
     \item

    Now assume in addition that $\tp(a/Ab)$ is topologically 1-based over $A$. We revisit the construction in (1) slightly to prove (3). First, choose $e$, $B$, and $f'$ as in the original construction. Now let $g=\operatorname{germ}(a/Ab)$, so by assumption $\dim(b/Aa)=\dim(b/Aga)$. It follows that $e$ is also a basis for $b$ over $Aga$, so that $\dim(e/Aa)=\dim(e/Aga)$, and thus $\mu(e/Aa)=\mu(e/Aga)$. Since $f'\in\mu(e/Aa)$, we conclude that $f'\in\mu(e/Aga)$, and thus $\tp(aeg/A)=\tp(af'g/A)$.
    
    Now we make two quick edits to the construction. Namely, we can first choose $f$ to satisfy the stronger property that $\tp(aefg/A)=\tp(aef'g/A)$, and thus $\tp(afg/A)=\tp(af'g/A)=\tp(aeg/A)$. Then we can in turn choose $c$ to satisfy the stronger property that $\tp(abeg/A)=\tp(acfg/A)$ (so in particular $\tp(abe/A)=\tp(acf/A)$ as in (1)). Then the conclusions of (1) still follow; while we in addition have $\tp(abg/A)=\tp(acg/A)$, thus $g=\operatorname{germ}(a/Ab)=\operatorname{germ}(a/Ac)$. This proves (2).
    \end{enumerate}
    \end{proof}

\subsection{Topological Characterization of Zero-Dimensionality}

We also need one more lemma before the main argument. Recall that topologically 1-based types require a `low complexity germ', in the sense that knowing the germ doesn't lower the dimension of the parameter (and assuming exchange, this is equivalent to the germ having dimension zero). The following lemma gives a topological characterization of this phenomenon for general imaginaries. This result is likely well known, but we include a proof for completeness.

\begin{lemma}\label{L: 0-dim}
    Let $a\in M^n$, $b\in\mathcal M^{eq}$, and $A$ a parameter set. Assume that $b\in\dcl(Aa)$. Then the following are equivalent:
    \begin{enumerate}
    \item $\dim(a/Ab)=\dim(a/A)$.
    \item There is an $A$-definable function $f:X\rightarrow Y$, where $X\subset M^n$ and $Y\subset\mathcal M^{eq}$, which takes constant value $b$ on a neighborhood of $a$.
    \end{enumerate}
\end{lemma}

\begin{proof}
    First, assume (1). Since $b\in\dcl(Aa)$, there are an $A$-definable function $f:X\rightarrow Y$, where $X\subset M^n$ and $Y\subset\mathcal M^{eq}$, such that $f(a)=b$. Shrinking $X$ if necessary, we may assume it is a d-approximation of $\tp(a/A)$ at $a$. Now (1) implies that $\operatorname{germ}(a/Ab)=\operatorname{germ}(a/A)=\operatorname{germ}_a(X)$. In particular, the $Ab$-definable condition `$f(x)=b$' must hold on a neighborhood of $a$ in $X$. This proves (2).

    Now assume (2), and let $f:X\rightarrow Y$ be the given function. Let $U$ be an open set containing $a$ such that $f$ takes constant value $b$ on $X\cap U$. By the independent neighborhood property, we may assume $U$ is $t$-definable where $\dim(a/At)=\dim(a/A)$. But now $b$ is $At$-definable, since $b$ is the unique element of $f(X\cap U)$. Thus also $\dim(a/Ab)=\dim(a/A)$, which proves (1).
\end{proof}

\subsection{The Main Structure Result}

Finally, we give the main result of this section:

\begin{theorem}\label{T: partition} Let $a$ and $b$ be real tuples, and $A$ a parameter set. Then the following are equivalent: 
\begin{enumerate}
    \item $\tp(a/Ab)$ is topologically 1-based over $A$.
    \item Any two fibers in the projection $\mu(ab/A)\rightarrow\mu(b/A)$ are either equal or disjoint. 
\end{enumerate}
\end{theorem}

\begin{proof}
    First assume $\tp(a/b)$ is topologically 1-based. To prove (2), we chose tuples $(a',b'),(a'',b'),(a',b'')\in\mu(ab/A)$. We want to show that also $(a'',b'')\in\mu(ab/A)$. To show this, let $X$ be an $\mathcal M$-definable d-approximation of $\tp(ab/A)$ at $(a,b)$; we show that $(a'',b'')\in X$. By Lemma \ref{L: d-app locality}, we may assume $X$ is $t$-definable where $\dim(ab/At)=\dim(ab/A)$.

    Let $g=\operatorname{germ}(a/Ab)$. Since $\dim(ab/At)=\dim(ab/A)$, Lemma \ref{L: fiber germ equation}(3) gives that also $g=\operatorname{germ}(a/Abt)$. Now by Lemma \ref{L: additive extension}(3), there is $c\models\tp(b/Aat)$ such that $\dim(c/Aabt)=\dim(b/Aat)$, $(c,ab)$ is additive over $At$, and $g=\operatorname{germ}(a/Act)$. Then $\tp(abgt/A)=\tp(acgt/A)$, since $g$ is definable over $Aabt$ and $Aact$ by the same formula. 

    Since $\dim(ab/At)=\dim(ab/A)$, $X$ is also a d-approximation of $\tp(ab/At)$ at $(a,b)$. Then by Lemma \ref{L: fiber germ equation}(2), the fiber $X_b$ realizes $g=\operatorname{germ}(a/Abt)$. Since $\tp(abgt/A)=\tp(acgt/A)$, $X_c$ also realizes $g$. So $X_b$ and $X_c$ agree on some neighborhood of $a$, say $U$. By Lemma \ref{L: inp for bigger tuple}, we may assume $U$ is $u$-definable where $\dim(abc/Atu)=\dim(abc/At)=\dim(abc/A)$. Note that by Lemma \ref{L: additive extension}(2) we know that $b$ and $c$ are symmetrically independent over $Aatu$; thus $\dim(b/Aactu)=\dim(b/Aatu)$.

    Let $e$ be a code of the set $X_b\cap U=X_c\cap U$. So $e\in\dcl(Aabtu)\cap\dcl(Aactu)$. Thus $$\dim(b/Aatu)\ge \dim(b/Aaetu)\geq\dim(b/Aactu)=\dim(b/Aatu).$$ So equality holds throughout, and by Lemma \ref{L: 0-dim}, there is an $Aatu$-definable function $f:Y\rightarrow Z$ such that $b\in Y$ and $f$ takes constant value $e$ in a neighborhood of $b$, say $V$. By Lemma \ref{L: inp for bigger tuple}, we may assume $V$ is $v$-definable where $\dim(abc/Atuv)=\dim(abc/Atu)=\dim(abc/At)$. Note, then, that $e\in\dcl(Aatuv)$, since it is the unique image of $f$ on $Y\cap V$.

    Now since $(c,ab)$ is additive over $At$ and $\dim(abc/Atuv)=\dim(abc/At)$, it follows by Lemma \ref{L: additive pairs} that $\dim(ab/Atuv)=\dim(ab/At)=\dim(ab/A)$. Thus $\mu(ab/Atuv)=\mu(ab/A)$. In particular, we have $$(a',b'),(a'',b'),(a',b'')\in\mu(ab/Atuv),$$ and thus each of these three pairs realizes $\tp(ab/Atuv)$. Let $e',e''$ be such that $$\tp(abetuv/A)=\tp(a'b'e'tuv/A)=\tp(a'b''e''tuv/A).$$ Since $e\in\dcl(Aatuv)$ (i.e. the $b$-coordinate is not needed to define $e$), it follows that $e'=e''$. Thus $e'$ codes both $X_{b'}\cap U$ and $X_{b''}\cap U$, and thus $X_{b'}\cap U=X_{b''}\cap U$. But since $$\tp(abtuv/A)=\tp(a''b'tuv/A),$$ and $a\in X_b\cap U$, it follows that $a''\in X_{b'}\cap U$ (recall here that $X$ and $U$ are $Atu$-definable). Thus we also have $a''\in X_{b''}\cap U$, which shows that $(a'',b'')\in X$, as desired. We have now proved that (1) implies (2).

    Now for the converse, assume that any two fibers of $\mu(ab/A)\rightarrow\mu(b/A)$ are equal or disjoint. We show that $\tp(a/Ab)$ is topologically 1-based. Since $\dim(\mu(b/Aa))=\dim(b/Aa)$, a compactness argument gives an element $c\in\mu(b/Aa)$ with $\dim(c/Aab)\geq\dim(b/Aa)$. By Lemma \ref{L: fiber germ equation}, $(a,c)\in\mu(ab/A)$. So by assumption, the fibers above $b$ and $c$ in $\mu(ab/A)$ are equal.

    Let $X$ be a $\emptyset$-definable d-approximation of $\tp(ab/A)$ at $(a,b)$. By Lemma \ref{L: fiber germ equation}, the fiber $X_b$ realizes $\operatorname{germ}(a/Ab)$. Since $(a,c)\in\mu(ab/A)$, we have $\tp(ab/A)=\tp(ac/A)$, and thus the fiber $X_c$ realizes $\operatorname{germ}(a/Ac)$. But by Lemma \ref{L: mu open}, $\mu(ab/A)$ is open in $X$; thus the common fiber above $b$ and $c$ in $\mu(ab/A)$ realizes both $\operatorname{germ}_a(X_b)$ and $\operatorname{germ}_a(X_c)$. In particular, $\operatorname{germ}(a/Ab)=\operatorname{germ}(a/Ac)=g$, say. Then $\tp(abg/A)=\tp(acg/A)$, since $g$ is definable over $Aab$ and $Aac$ by the same formula. Now since $g\in\dcl(Aab)$, we have $$\dim(b/Aag)=\dim(c/Aag)\geq\dim(c/Aab)\geq\dim(b/Aa).$$  This implies that $\tp(a/Ab)$ is topologically 1-based.
\end{proof}

We conclude, as promised earlier, that topological 1-basedness restricts exactly to the prior notion of weak 1-basedness for geometric theories:

\begin{corollary}\label{C: weak 1b} Assume $\mathcal M$ satisfies exchange. Then $\mathcal M$ is topologically 1-based if and only if it is weakly 1-based.
\end{corollary}
\begin{proof}
    We showed in Lemma \ref{L: w1b to 1b} that weak 1-basedness implies topological 1-basedness. We now show the reverse implication. So assume $\mathcal M$ is topologically 1-based.

    Recall that to show $\mathcal M$ is weakly 1-based, we are given a real tuple $a$ and a small set $B\subset M$. We need to find $a'\models\tp(a/B)$ so that $\dim(a/a')=\dim(a/Ba')=\dim(a/B)$. By compactness (since $\mathcal M$ is sufficiently saturated), it is enough to prove the result in the case where  $B:=b$ is a finite tuple.

    First, since $\dim(\mu(a/b))=\dim(a/b)$, there is $c\in\mu(a/b)$ with $\dim(c/ab)\geq\dim(a/b)$. In particular, $\tp(c/b)=\tp(a/b)$, and $\dim(c/ab)=\dim(a/b)$.

     Similarly, there is $d\in\mu(b/a)$ with $\dim(d/abc)=\dim(b/a)$. Then the fibers above $b$ and $d$ in $\mu(ab/\emptyset)\rightarrow\mu(b/\emptyset)$ both contain $a$, and thus they are equal by Theorem \ref{T: partition}. In particular, since $(c,b)\in\mu(ab/\emptyset)$, it follows that $(c,d)\in\mu(ab/\emptyset)$ as well. Thus $\tp(cd)=\tp(ab)$, and so $\dim(d/c)=\dim(b/a)=\dim(d/abc)$. It follows that $d$ and $a$ are independent over $c$, so that $\dim(a/c)=\dim(a/cd)$.

    Now, since $(a,d)\in\mu(ab/\emptyset)$, we also have $\tp(ad)=\tp(ab)$, and thus $$\dim(d/a)=\dim(b/a)=\dim(d/abc).$$ It follows that also $\dim(d/ab)=\dim(abc)$, so $d$ and $c$ are independent over $ab$. Thus also $$\dim(c/abd)=\dim(c/ab)=\dim(a/b)=\dim(c/d).$$ It follows that $c$ and $a$ are independent over $d$. Thus we also have $\dim(a/cd)=\dim(a/d)$.
    
    To recap, we have now shown that $\dim(a/c)=\dim(a/cd)=\dim(a/d)$. Now, since $(a,d)\in\mu(ab/\emptyset)$, we have $\tp(ab)=\tp(ad)$, and thus there is $a'$ so that $\tp(aba')=\tp(adc)$. Since $\mathcal M$ is sufficiently saturated, we may assume $a'$ is a tuple from $\mathcal M$. Then $\tp(a'b)=\tp(cd)=\tp(ab)$, so that $a'\models\tp(a/b)$ -- and moreover, since $\dim(a/c)=\dim(a/cd)=\dim(a/d)$, it follows that $\dim(a/a')=\dim(a/a'b)=\dim(a/b)$. This shows that $\mathcal M$ is weakly 1-based.
\end{proof}

\section{Groupoid Spines}

Later on we will show that under a suitable non-triviality assumption, any topologically 1-based t-minimal structure with the independent neighborhood property admits an infinite type-definable group. This amounts to a topological version of the group configuration theorem. Before giving the argument, we first separate its purely combinatorial content. Thus, the current section is independent from model theory and concerns only the recovery of a group (in fact, a groupoid) from a well-behaved family of bijections. The general idea is similar to \cite{ChePetSta}, but with details more suited to our needs. Then in the next section, we add type-definability requirements (but working in an arbitrary theory), before finally returning to our topologically 1-based setting in Sections 9-11. 

\subsection{Basic Notions}

The central object we will work with for now is:

\begin{definition}\label{D: groupoid spine}
    A \textit{groupoid spine} consists of:
    \begin{itemize}
        \item A linear order $(I,<)$,
        \item An $I$-indexed family of non-empty sets. $\mathcal X=\{X_i:i\in I\}$,
        \item A non-empty relation $\emptyset\neq R\subset I^2$ containing all pairs $(i,j)$ with $i<j$, and
        \item For all $(i,j)\in R$, a non-empty collection $\operatorname{Mor}(i,j)$ of bijections $X_i\rightarrow X_j$ (called \textit{morphisms}),
    \end{itemize}
    such that the morphisms are relatively closed under the groupoid operations. That is:
        \begin{enumerate}
        \item If $(i,i)\in R$ for some $i$ then $\operatorname{id}:X_i\rightarrow X_i\in\operatorname{Mor}(i,i)$.
        \item If $(i,j),(j,i)\in R$ and $f\in\operatorname{Mor}(i,j)$ then $f^{-1}\in\operatorname{Mor}(j,i)$.
        \item If $(i,j)(j,k),(i,k)\in R$, $f\in\operatorname{Mor}(i,j)$, and $g\in\operatorname{Mor}(j,k)$, then $g\circ f\in\operatorname{Mor}(i,k)$.
        \end{enumerate}
    We denote a groupoid spine by the data $(I,<,\mathcal X,R,\mathrm{Mor})$, where $\mathrm{Mor}$ is the family $\{\operatorname{Mor}(i,j):(i,j)\in R\}$.
\end{definition}

Note that Definition \ref{D: groupoid spine} implies that for any $i,j\in I$, there is a composition of morphisms and their inverses sending $X_i$ to $X_j$ (which may be empty if $i=j$). This ensures that the groupoid determined by the data (if it exists) is connected. 

\begin{definition}
    Let $(I,<,\mathcal X,R,\mathrm{Mor})$ be a groupoid spine. 
    \begin{enumerate}
        \item Suppose $R\subset S\subset I^2$. We say that $(I,<,\mathcal X,R,\mathrm{Mor})$ \textit{extends to $S$} if one can define $\operatorname{Mor}(i,j)$ for $(i,j)\in S-R$ in such a way that $(I,<,\mathcal X,S,\mathrm{Mor})$ is a groupoid spine.
        \item We say that $(I,<,\mathcal X,R,\mathrm{Mor})$ \textit{extends to a groupoid}, if it extends to $S=I^2$.
    \end{enumerate}
\end{definition}

In other words, $(I,<,\mathcal X,R,\mathrm{Mor})$ extends to $S$ if it can be grown to a groupoid spine on $S$ without adding any new morphisms to pairs in $R$.

The following is then clear.

\begin{lemma} Suppose $(I,\mathcal X,R,\mathrm{Mor})$ is a groupoid spine that extends to a groupoid. Then the extension to a groupoid is uniquely determined. More precisely, the resulting groupoid is precisely the groupoid on $I$ generated by all morphisms between pairs in $R$.
\end{lemma}

\subsection{Combinatorial Group Configuration Theorem}

We now give a purely combinatorial analog of the group configuration theorem. The statement will say that a symmetric groupoid spine on at least three objects always extends to a groupoid. First, we define:

\begin{definition}
    A groupoid spine $(I,<,\mathcal X,R,Mor)$ is \textit{symmetric} if $R$ is symmetric -- that is, $(i,j)\in R$ implies $(j,i)\in R$ for all $i,j\in I$.
\end{definition}

Equivalently, one easily checks that $(I,<,\mathcal X,R,Mor)$ is symmetric if and only if it contains all pairs $(i,j)$ with $i\neq j$. We now show:

\begin{theorem}\label{T: gp con} Let $(I,<\mathcal X,R,\mathrm{Mor})$ be a symmetric groupoid spine with $|I|\geq 3$.
Then $(I,<,\mathcal X,R,\mathrm{Mor})$ extends to a groupoid.
\end{theorem}
\begin{proof}
    Let $\mathcal G$ be the groupoid on $I$ generated by all maps in the sets $\operatorname{Mor}(i,j)$ for $(i,j)\in R$. It suffices to show that for $(i,j)\in R$, every $\mathcal G$-morphism $i\rightarrow j$ already belongs to $\operatorname{Mor}(i,j)$.

    So let $(i,j)\in R$, and let $f:X_i\rightarrow X_j$ be any $\mathcal G$-morphism. So $f$ is a composition of maps from Mor and their inverses. On the other hand, by symmetry, inverses are not needed. So in fact, $f$ is directly a composition of maps from Mor.
    
    To be precise, we now obtain $i=i_0,\dots,i_n=j\in I$ such that each $(i_l,i_{l+1})\in R$ and $f$ factors as a composition $$X_{i_0}\xrightarrow{f_1}\dots\xrightarrow{f_n}X_{i_n},$$ where each $f_{i_l}\in\operatorname{Mor}(i_{l-1},i_l)$. We may assume that $n$ is minimal among all such decompositions of $f$. Our goal is to show that $n=0$ or $n=1$ (where we interpret $n=0$ as $i=j$ and $f=\operatorname{id}:X_i\rightarrow X_i$). So, assume $n\geq 2$.

    Note that if $(i_l,i_{l+2})\in R$ for any $l$, then by closure under composition, the chain $$X_{i_l}\xrightarrow{f_l}X_{i_{l+1}}\xrightarrow{f_{l+1}}X_{i_{l+2}}$$ can be shortened to a single morphism $X_{i_l}\rightarrow X_{i_{l+2}}$, contradicting the minimality of $n$. So each $(i_l,i_{l+2})\notin R$. Since $R$ contains all distinct pairs, this implies $i_{l+2}=i_l$ whenever $0\leq l\leq n-2$. So the sequence $i_0,\dots,i_n$ alternates between the two values $i_0=i$ and $i_1$, and $(i,i)\notin R$. Since $(i,j)\in R$ by assumption, this implies $i\neq j$. This is only possible if $j=i_1$. That is, we deduce that $i$ and $j$ are distinct and the sequence $i_0,\dots,i_n$ alternates between $i$ and $j$. Since $i_n=j$, it follows that $n$ is odd, so in particular $n\geq 3$. So our decomposition starts with the chain $$X_i\xrightarrow{f_1}X_j\xrightarrow{f_2}X_i\xrightarrow{f_3}X_j.$$

    We now show that this chain of three maps can be shortened to a single morphism $X_i\rightarrow X_j$, which will contradict the minimality of $n$ and thus prove the theorem. To do this, recall the assumption that $|I|\geq 3$. Thus (as $i\neq j$) there is $k\in I$ so that $i,j,k$ are all distinct. Thus all non-diagonal pairs from $\{i,j,k\}$ belong to $R$.

    Now fix any morphism $g:j\rightarrow k$. By closure under inverses and composition, there is a morphism $h:k\rightarrow i$ so that $h\circ g=f_2$. So the first three maps in our decomposition can be expanded to the chain $$X_i\xrightarrow{f_1}X_j\xrightarrow gX_k\xrightarrow hX_i\xrightarrow{f_3}X_j.$$ By closure under composition applied to each of $X_i\xrightarrow{f_1}X_j\xrightarrow gX_k$ and $X_k\xrightarrow hX_i\xrightarrow{f_3}X_j$, this chain simplifies to a composition of two morphisms $X_i\rightarrow X_k\rightarrow X_j$ -- which then simplifies for the same reason to a single morphism $X_i\rightarrow X_j$, as desired. 
    \end{proof}

\subsection{Regularity}

Typically, one should not expect a groupoid spine to extend to a groupoid (at least, not without having many pairs $(i,j)\in R$ with $i>j$) -- and this is related to group configurations. For example, let $I=\{1,2,3\}$ with the usual order $<$, fix a non-empty set $X$, and fix any non-empty family $F$ of bijections $X\rightarrow X$. Now define:
\begin{itemize}
    \item $X_i=X$ for each $i$.
    \item $\operatorname{Mor}(1,2)=\operatorname{Mor}(2,3)=F$.
    \item $\operatorname{Mor}(1,3)=F\circ F$, the set of all compositions of two maps from $F$.
\end{itemize}
This data defines a groupoid spine (where $R=\{(1,2),(2,3),(1,3)\}$). One can check that if this groupoid spine extends to a groupoid, then $F$ is a (left and right) coset of a group of permutations of $X$. In particular, given $f,g\in F$, the composition $f\circ g$ is `independent' from both $f$ and $g$  (meaning that for any $f'$, there is $g'$ with $f\circ g=f'\circ g'$, and the analogous statement with $f$ and $g$ reversed). This independence of a composition with its factors is also the main idea of group configurations. So one can think of a groupoid spine with this property as an analog to a group configuration.

Typically, the independence property above is not expected. However, in stable 1-based theories, group configurations arise `automatically': for example, in a locally modular strongly minimal theory, any composition of one-dimensional families of plane curves gives rise to a group configuration. This is roughly because the composite family cannot be too large, in the sense that a composition $f\circ g$ is generically determined by one point. In the abstract language of groupoid spines, the analogous property is \textit{regularity}:

\begin{definition}
    Let $(I,<,\mathcal X,R,Mor)$ be a groupoid spine. We say that $(I,<,\mathcal X,R,Mor)$ is \textit{regular} if for all $(i,j)\in R$, and all $x\in X_i$ and $y\in X_j$, there is exactly one morphism $f\in\operatorname{Mor}(i,j)$ with $f(x)=y$.
\end{definition}

Note the analogy of the above definition with regular group actions. In particular, a group of permutations of a set induces a regular groupoid spine in the obvious way if and only if the corresponding group action is regular.

\subsection{The Main Result on Regular Groupoid Spines}

We now show an analog to the fact that group configurations are automatic in the 1-based case. Namely, we show that for groupoid spines, regularity guarantees extendability.

\begin{theorem}\label{T: regular gp con}
  Let $(I,<,\mathcal X,R,Mor)$ be a regular groupoid spine. Then:
  \begin{enumerate}
      \item There is a symmetric set $S\supset R$ such that $(I,\mathcal X,R,Mor)$ extends to $S$.
      \item In particular, if $|I|\geq 3$ then $(I,\mathcal X,R,Mor)$ extends to a groupoid.
  \end{enumerate}
\end{theorem}
\begin{proof}
    That (1) implies (2) is immediate by Theorem \ref{T: gp con}. We only show (1).

    Let $S$ be the symmetric closure of $R$ -- so $S$ adds to $R$ all pairs $(i,j)\notin R$ with $i>j$, and in particular $S$ is symmetric. We will use multiple times the fact that $S$ only adds strictly decreasing pairs to $R$ -- so if $(i,j)\in S$ and $i\leq j$ then $(i,j)\in R$.
    
    For any pair $(i,j)\in S-R$, note that $(j,i)\in R$, so $\operatorname{Mor}(j,i)$ is well-defined. Now set $\operatorname{Mor}(i,j)$ to be the collection of all inverses of maps in $\operatorname{Mor}(j,i)$. All of the required properties in Definition \ref{D: groupoid spine} follow immediately, except closure under composition, which will occupy the rest of the proof.

    So, let $i,j,k\in I$ with $(i,j),(j,k),(i,k)\in S$, and consider a composition of morphisms $i\rightarrow j\rightarrow k$. We want to show that the composite map $X_i\rightarrow X_k$ is a morphism $i\rightarrow k$. To do this, our first main goal is to convert everything into a problem in the original groupoid spine (by replacing certain maps with their inverses to reduce to a statement about only maps between $R$-pairs). We then apply regularity in the original spine to prove the resulting statement.

    Let us proceed. First, note that if $i\leq j\leq k$, then $(i,j),(j,k),(i,k)\in R$, and the desired statement follows from closure under composition for $R$-pairs. The case $i\geq j\geq k$ then follows similarly by replacing each map with its inverse.

    Otherwise, $j$ is either the maximum or minimum of $\{i,j,k\}$. We will prove the case that $j=\max\{i,j,k\}$, as the other case is similar.  Assuming $j=\max\{i,j,k\}$, the problem is now symmetric in $i$ and $k$ -- so without loss of generality, we assume $i\leq k$, and thus $(i,k)\in R$ as well.
    
    By replacing the given morphism $j\rightarrow k$ with its inverse, we can now restate the problem in terms of $R$-pairs: we know that $(i,j),(k,j),(i,k)\in R$, and we have morphisms $f:i\rightarrow j$ and $g:k\rightarrow j$. We want to find a morphism $h:i\rightarrow k$ with $g\circ h=f$.

    We now find such a morphism $h$ using regularity. First, pick any element $y\in X_k$. Let $z=g(y)\in X_j$, and let $x\in X_i$ with $f(x)=z$. By regularity, there is $h\in\operatorname{Mor}(i,k)$ with $h(x)=y$. Thus $g\circ h$ is a morphism $i\rightarrow j$ sending $x$ to $z$. By regularity, there is only one such morphism. So since $f(x)=z$ as well, it follows that $g\circ h=f$, as desired.
\end{proof}

Before moving on, we also note that regularity is preserved under groupoid extensions:

\begin{lemma}\label{L: regularity of extension}
    Let $(I,\mathcal X,R,Mor)$ be a regular groupoid spine. If $(I,\mathcal X,R,Mor)$ extends to a groupoid, then the resulting extension is also regular.
\end{lemma}
\begin{proof}
    Let $(i,j)\in I^2$, and let $x\in X_i$ and $y\in X_j$. We show that there is exactly one morphism $i\rightarrow j$ (in the groupoid extension) sending $x$ to $y$.

    If $i<j$, then $(i,j)\in R$, and the desired statement follows from regularity for $R$-pairs. The case $i>j$ then follows similarly, by taking inverses and finding a unique morphism $j\rightarrow i$ sending $y$ to $x$. Combining these two cases, the desired statement holds whenever $i\neq j$.

    Now assume $i=j$, and from now on write $(i,j)$ as $(i,i)$. If $(i,i)\in R$ then we are done (by regularity for $R$-pairs), so we may further assume $(i,i)\notin R$. Recall that $R\neq\emptyset$ (by the definition of groupoid spines); so there must be at least one other element of $I$, say $k$. Fix any morphism $g:X_i\rightarrow X_k$, and let $z=g(y)\in X_k$. By the groupoid axioms, the map $f\mapsto g\circ f$ gives a bijective correspondence between $\operatorname{Mor}(i,i)$ and $\operatorname{Mor}(i,k)$. Moreover, for $f:i\rightarrow i$, the condition $f(x)=y$ is equivalent to $g\circ f(x)=z$. Since $i\neq k$ (and this case was covered above), there is a unique morphism $i\rightarrow k$ sending $x$ to $z$. Passing back through our bijective correspondence, this implies there is a unique morphism $i\rightarrow i$ sending $x$ to $y$, as desired.
\end{proof}

\section{Type-Definability Requirements}

We now revisit the constructions from the previous section in the context of a first-order theory, adding type-definability hypotheses in each step. The results of this section are valid in any first-order theory. In the following section, we then show that they apply in the setting of topological 1-basedness considered earlier. In this section, we work in a highly saturated first-order structure $\mathcal M$ (with no topology assumptions).

\subsection{Type-Definability and Faithful Type-Definability}

\begin{definition}\label{D: weak type def}
    Let $A$ be a parameter set, and let $(I,<,\mathcal X,R,Mor)$ be a groupoid spine. We say that $(I,<,\mathcal X,R,Mor)$ is \textit{$A$-type-definable} if $I$ is small, each $X_i$ is $A$-type-definable, and the morphisms can be represented by $A$-type-definable families. More precisely, this means that for each $(i,j)\in R$, one can choose $A$-type-definable sets $T_{ij}$ and $F_{ij}\subset X_i\times X_j\times T_{ij}$, so that:
    \begin{enumerate}
        \item Each fiber $(F_{ij})_t\subset X_i\times X_j$ for $t\in T_{ij}$ is the graph of a (necessarily unique) morphism $i\rightarrow j$, which we denote $f^{ij}_t:X_i\rightarrow X_j$. 
        \item For each morphism $f:i\rightarrow j$, there is $t\in T_{ij}$ with $f^{ij}_t=f$.
    \end{enumerate}
    In this case, we call $\{T_{ij}\}$ and $\{F_{ij}\}$ a \textit{type-definition of $(I,<,\mathcal X,R,Mor)$ over $A$}.
    \end{definition}

    \begin{remark}\label{R: relatively definable} Assume that $(I,<,\mathcal X,R,Mor)$ is an $A$-type-definable groupoid spine, with a type-definition $(\{T_{ij}\},\{F_{ij}\})$ over $A$. Then by compactness, each $F_{ij}$ is in fact relatively $A$-definable in $X_i\times X_j\times T_{ij}$ (this is a general property of type-definable families of functions).
    \end{remark}
    
    \begin{definition}
        Let $(I,<,\mathcal X,R,Mor)$ be a groupoid spine, and let $A$ be a parameter set. We say $(I,<,\mathcal X,R,Mor)$ is \textit{faithfully $A$-type-definable} if it is $A$-type-definable and in addition equality and composition of maps are $A$-type-definable. More precisely, this means one can find a type-definition $(\{T_{ij}\},\{F_{ij}\})$ over $A$ with the following extra properties:
        \begin{enumerate}
            \item For all $(i,j)\in R$, the set $$\{(t,u)\in T_{ij}:f_t^{ij}=f_u^{ij}\}$$ is $A$-type-definable.
            \item Whenever $(i,j)(j,k),(i,k)\in R$, the set $$\{(t,u,v)\in T_{ij}\times T_{jk}\times T_{ik}:f_u^{jk}\circ f_t^{ij}=f_v^{ij}\}$$ is $A$-type-definable.
        \end{enumerate}
    \end{definition}

    Later, we will see that faithfully type-definable groupoid spines can be definably quotiented by equality of maps (see Propositions \ref{P: elimination} and \ref{P: group from groupoind}). This is the reason for the word \textit{faithful}.  

In general, it is not clear that a type-definable groupoid spine should be faithfully type-definable: for example, the obvious definition of the formula $f^{ij}_t=f^{ij}_u$ (for given $(i,j)\in R$) would require quantifying over the type-definable set $X_i$ -- and this is only allowed if $X_i$ is in fact definable.

\subsection{Preservation Under Extensions}

The main advantage of ordinary type-definability (versus faithful type-definability) is that it allows for a straightforward extension result:

\begin{lemma}\label{L: extension definable} Let $A$ be a parameter set, and let $(I,<,\mathcal X,R,Mor)$ be an $A$-type-definable groupoid spine which extends to a groupoid. Then the resulting extension is also $A$-type-definable.
\end{lemma}
\begin{proof} Fix $(\{T_{ij}\},\{F_{ij}\})$, a type-definition of $(I,<,\mathcal X,R,Mor)$ over $A$. We show how to define $T_{ij}$ and $F_{ij}$ when $(i,j)\notin R$.

So let $(i,j)\notin R$. First suppose $i\neq j
$. Then necessarily $i>j$, so that $(j,i)\in R$ and $\operatorname{Mor}(i,j)$ is the collection of inverses of maps in $\operatorname{Mor}(i,j)$. Now define $T_{ij}=T_{ji}$, and set $(x,y,t)\in F_{ij}$ if and only if $(y,x,t)\in F_{ji}$; clearly this gives an $A$-type-definable family representing all maps in $\operatorname{Mor}(i,j).$

Now suppose $i=j$. By definition the relation $R\subset I^2$ is non-empty, which implies there must be some $k\in I$ with $k\neq i$. By the previous paragraph, we may assume $(i,k),(k,i)\in R$. Note, then, that $\operatorname{Mor}(i,i)$ is the collection of compositions $g\circ f$ for $f\in\operatorname{Mor}(i,k)$ and $g\in\operatorname{Mor}(k,i)$. So, we let $T_{ii}=T_{ik}\times T_{ki}$. And for $(x,y,(t,u))\in X_i\times X_i\times T_{ii}$, we set $(x,y,(t,u))\in F_{ii}$ if and only if there is $z$ with $(x,z,t)\in F_{ik}$ and $(z,y,u)\in F_{ki}$. Then $F_{ii}$ is a projection of an $A$-type-definable set, so is $A$-type-definable. 
\end{proof}

\subsection{An Elimination Result}

Our main goal is to interpret, under appropriate assumptions,  a group from a type-definable groupoid spine. The group we hope to construct is the group of morphisms from an object to itself in an extension to a groupoid. Notice, though, that our definitions of type-definability and faithful type-definability have a flaw in this regard: a type-definition is allowed to repeat one map several times (i.e. the natural map $T_{ij}\rightarrow\operatorname{Mor}(i,j)$ is surjective but might not be injective). Thus, one cannot hope to build a group structure on the parameter space $T_{ii}$ directly -- instead, one needs to build a quotient of it. Forming the necessary quotient object (in a type-definable manner) is the main use of faithful type-definability -- in this case, it essentially amounts to an elimination result for certain hyperimaginaries, which we treat separately below.

\begin{proposition}\label{P: elimination} Let $A$ be a parameter set. Let $X$ and $T$ be $A$-type-definable sets, and $Y\subset X\times T$ a relatively $A$-definable relation. For $t\in T$, denote $Y_t:=\{x:(x,t)\in Y\}$, and let $\sim$ be the equivalence relation $Y_t=Y_u$ on $T$. If $\sim$ is $A$-type-definable in $T^2$, then the quotient $T/\sim$ is $A$-type-definable in $\mathcal M^{eq}$: that is, there are an $A$-type-definable set $Z$ in $\mathcal M^{eq}$, and a relatively $A$-definable function $f:T\rightarrow Z$, such that for $t,u\in T$ we have $f(t)=f(u)$ if and only if $t\sim u$.
\end{proposition}

\begin{proof}
Write $X=\bigcap_{i\in I}X_i$ as a small intersection of $A$-definable sets. We may assume  $\{X_i\}_{i\in I}$ is closed under finite intersections. Also write $Y=Y_0\cap(X\times T)$ for some $A$-definable set $Y_0$.

By compactness, for any $t,u\in T$, we have $t\sim u$ if and only if $(Y_0)_t\cap X_i=(Y_0)_u\cap X_i$ for some $i\in I$. By compactness again (and since the $X_i$ form a directed family),  there is a single $i\in I$ so that for all $t,u\in T$, $t\sim u$ if and only if $(Y_0)_t\cap X_i=(Y_0)_u\cap X_i$.

Now let $T_0\supset T$ be an $A$-definable set, and let $E$ be the $A$-definable equivalence relation given by $(Y_0)_t\cap X_i=(Y_0)_u\cap X_i$ on $T_0$. Finally, let $Z=T_0/E$, an $A$-definable set in $\mathcal M^{eq}$. Then we have a relatively $A$-definable function $f:T\rightarrow Z$ (factoring as $T\rightarrow T_0\rightarrow Z$). And for $t,u\in T$, we have $f(t)=f(u)$ if and only if $t$ and $u$ are $E$-equivalent, if and only if $(Y_0)_t\cap X_i=(Y_0)_u\cap X_i$, if and only if $t\sim u$. 
\end{proof}

\subsection{Interpreting a Group in the faithful Case}

We now use Proposition \ref{P: elimination} to interpret a group from a faithfully type-definable groupoid:

\begin{proposition}\label{P: group from groupoind} Let $A$ be a parameter set, and let $(I,<,\mathcal X,R,Mor)$ be a faithfully $A$-type-definable groupoid spine, with $R=I^2$. \footnote{In other words, a groupoid spine with $R=I^2$ is an actual groupoid -- or more precisely a connected groupoid of bijections equipped with a linear ordering of its objects.} Then for any $i$, the group $\operatorname{Mor}(i,i)$, equipped with its action on $X_i$, is $A$-type-definable in $\mathcal M^{eq}$. More precisely, in $\mathcal M^{eq}$ there are an $A$-type-definable group $G_i$ acting $A$-type-definably on $X_i$, and an isomorphism of group actions $(G_i,X_i)\leftrightarrow(\operatorname{Mor}(i,i),X_i)$ fixing $X_i$ pointwise. 
\end{proposition}

\begin{proof}
    Let $\{T_{ij}\}$ and $\{F_{ij}\}$ be a faithful type-definition over $A$, and fix $i\in I$. By Remark \ref{R: relatively definable}, $F_{ii}$ is relatively $A$-definable in $X_i\times X_j\times T_{ij}$, so that Proposition \ref{P: elimination} can apply to $F_{ii}$. As in the setup of Proposition \ref{P: elimination}, let $\sim$ be the equivalence relation $f_t^{ii}=f_u^{ii}$ on $T_{ii}$. By faithful $A$-type-definability, $\sim$ is automatically $A$-type-definable, so that Proposition \ref{P: elimination} applies.
    
    Now by Proposition \ref{P: elimination}, after passing to $\mathcal M^{eq}$ we obtain an $A$-type-definable set $G$ and a relatively $A$-definable function $h:T_{ii}\rightarrow G$ such that for $t,u\in T_{ii}$ we have $h(t)=h(u)$ if and only if $t\sim u$. So $G$ is in natural bijective correspondence with $\operatorname{Mor}(i,i)$, and one can now endow $G$ $A$-type-definably with the action of $\operatorname{Mor}(i,i)$ on $X_i$ (so that for $g\in G$ and $x\in X_i$, $gx$ is now defined). To finish the proof, we just need to show that the group operation and action (of $G$ and $X_i$) are $A$-type-definable. 
    
    Now for $g\in G$ and $x,y\in X_i$, we have $gx=y$ if and only if there is $t\in T_{ii}$ with $h(t)=g$ and $f_t^{ii}(x)=y$. This shows that the action (i.e. the set defined by $gx=y$) is an image of an $A$-type-definable set, so is $A$-type-definable. Similarly, for $g_1,g_2,g_3\in G$, we have $g_1g_2=g_3$ if and only if there are $t_1,t_2,t_3\in T_{ii}$ with each $h(t_j)=g_j$ and $f_{t_1}^{ii}\circ f_{t_2}^{ii}=f_{t_3}^{ii}$. And again, this shows that the group operation is an image of an $A$-type-definable set, so is $A$-type-definable (note that we are using faithful $A$-type-definability here to get that $f_{t_1}^{ii}\circ f_{t_2}^{ii}=f_{t_3}^{ii}$ is $A$-type-definable).
\end{proof}

\subsection{Regularity revisited}

So far we have an extension result for type-definable groupoid spines, and a group interpretation result for faithfully type-definable groupoid spines. Our goal now is to combine these into a single theorem. The trick is Lemma \ref{L: regular weak is strong} below, which shows that in the regular case, there is no difference between ordinary and faithful type-definability:

\begin{lemma}\label{L: regular weak is strong}
    For every parameter set $A$, every $A$-type-definable regular groupoid spine is faithfully $A$-type-definable.
\end{lemma}
\begin{proof} Let $(\{T_{ij}\},\{F_{ij}\})$ be a type-definition of the regular groupoid spine $(I,<,\mathcal X,R,Mor)$ over $A$. We show that equality and composition of maps are $A$-type-definable:
    \begin{itemize}
        \item Let $(i,j)\in R$. We show that equality of maps is $A$-type-definable in $T_{ij}^2$. Now for $(t,u)\in T_{ij}^2$, regularity gives that $f_t^{ij}=f_u^{ij}$ if and only if there are $x\in X_i$ and $y\in X_j$ such that $f_t^{ij}(x)=f_u^{ij}(x)=y$. This shows that the equality relation in $T_{ij}^2$ is a projection of an $A$-type-definable set, so is $A$-type-definable.
        \item Let $(i,j),(j,k),(i,k)\in R$. We show that composition is $A$-type-definable in $T_{ij}\times T_{jk}\times T_{ik}$. Now for $(t,u,v)\in T_{ij}\times T_{jk}\times T_{ik}$, regularity gives that $f_u^{jk}\circ f_t^{ij}=f_v^{ik}$ if and only if there are $x\in X_i$, $y\in X_j$, and $z\in X_k$ such that $f_t^{ij}(x)=y$, $f_u^{jk}(y)=z$, and $f_v^{ik}(x)=z$. This shows that the graph of composition is a projection of an $A$-type-definable set, so is $A$-type-definable. 
    \end{itemize}
\end{proof}

\subsection{Main Result}
Finally, we conclude with the main theorem of this section. The reader should view this as an analog of the group configuration for 1-based stable theories:

\begin{theorem}\label{T: combinatorial group conf}
    Let $A$ be a parameter set, and let $(I,<,\mathcal X,R,Mor)$ be an $A$-type-definable regular groupoid spine with $|I|\geq 3$. Then for any $i$, there is an $A$-type-definable group in $\mathcal M^{eq}$ acting regularly and type-definably on $X_i$. In particular, for any $e\in X_i$, there is an $Ae$-type-definable group structure on the set $X_i$, with $e$ as the identity element.
\end{theorem}
\begin{proof}
    By Theorem \ref{T: regular gp con}, $(I,<,\mathcal X,R,Mor)$ extends to a groupoid. By Lemma \ref{L: regularity of extension}, the resulting groupoid is regular; by Lemma \ref{L: extension definable}, the resulting groupoid is $A$-type-definable; thus by Lemma \ref{L: regular weak is strong}, the resulting groupoid is faithfully $A$-type-definable. Now apply Proposition \ref{P: group from groupoind} to obtain an $A$-type-definable group $G$ acting $A$-type-definably and regularly on $X_i$.

    To build a group structure on $X_i$ itself, fix any $e\in X_i$. There is now an $Ae$-type-definable bijection $G\rightarrow X_i$, given by $g\mapsto ge$ (by regularity). This allows us to transfer the group structure from $G$ to $X_i$ $Ae$-type-definably. Note that $e$ is identified with the identity, because the identity sends $e$ to $e$.
    \end{proof}

\section{Construction of a Regular Groupoid Spine}

\textbf{We now return to the topological setting considered earlier in the paper. Thus, we fix $\mathcal M$, a sufficiently saturated t-minimal structure with the independent neighborhood property; and we fix an elementary extension $\mathcal N$ of $\mathcal M$, sufficiently saturated with respect to $\mathcal M$.} As before, we consider infinitesimal neighborhoods of $\mathcal M$-tuples as $\mathcal M$-type-definable sets in $\mathcal N$.

Our goal now is to show that if $\mathcal M$ is non-trivial (see Definition \ref{D: nontrivial}) and topologically 1-based, then we can construct a type-definable regular groupoid spine in $\mathcal M$. Our main tool is the structure theorem for topologically 1-based types, Theorem \ref{T: partition}.

\subsection{Families of Bijections}

First, we show that from a certain configuration of interalgebraicities, we can construct a regular familiy of homeomrphisms between infinitesimal neighborhoods: 

\begin{lemma}\label{L: family of bijections} Let $(a,b,t)$ be real tuples. Assume:
\begin{enumerate}
\item Any two of $a,b,t$ are symmetrically independent.
\item $a$ and $b$ are interalgebraic over $t$.
\item $(t,ab)$ is additive. 
\item $\tp(ab/t)$ is topologially 1-based.
\end{enumerate}
Then $\mu(abt/\emptyset)$ is the graph of a regular family of homeomorphisms $\mu(a/\emptyset)\rightarrow\mu(b/\emptyset)$. More precisely, there is a family $F$ of homeomorphisms $f:\mu(a/\emptyset)\rightarrow\mu(b/\emptyset)$ such that:
\begin{enumerate}
    \item[(i)] For each $t'\in\mu(t)$, the fiber $\mu(abt/\emptyset)_{t'}\subset\mu(ab/\emptyset)$ is the graph of a homeomorphism $f_{t'}\in F$.
    \item[(ii)] For each $f\in F$, there is $t'\in\mu(t/\emptyset)$ with $f_{t'}=f$.
    \item[(iii)] For each $a'\in\mu(a/\emptyset)$ and $b'\in\mu(b/\emptyset)$, there is exactly one $f\in F$ with $f(a')=b'$.
\end{enumerate}
\end{lemma}
\begin{proof}
    By symmetric independence, Lemma \ref{L: mu strong ind} gives: $\mu(at/\emptyset)=\mu(a/\emptyset)\times\mu(t/\emptyset)$ and $\mu(bt/\emptyset)=\mu(b/\emptyset)\times\mu(t/\emptyset)$. By Lemma \ref{L: mu generic cont}, since $a\in\acl(bt)$, projecting gives a homeomorphism $$\mu(abt/\emptyset)\rightarrow\mu(bt/\emptyset)=\mu(b/\emptyset)\times\mu(t/\emptyset)$$ (technically this requires that $(a,bt)$ is additive, but additivity follows automatically from $a\in\acl(bt)$). Similarly, the projection $\mu(abt/\emptyset)\rightarrow\mu(a/\emptyset)\times\mu(t/\emptyset)$ is a homeomorphism. It follows that $\mu(abt/\emptyset)$ is the graph of a family of homeomorphisms $\mu(a/\emptyset)\rightarrow\mu(b/\emptyset)$. Let $F$ be the set of distinct homeomorphisms $\mu(a/\emptyset)\rightarrow\mu(b/\emptyset)$ occurring in this way.

    It remains to show that given $a'\in\mu(a/\emptyset)$ and $b'\in\mu(b/\emptyset)$, there is exactly one $f\in F$ with $f(a')=b'$. First, the fact that there is at least one such $f$ follows by the additivity of $(t,ab)$: indeed, since $a$ and $b$ are symmetrically independent, we have $\mu(ab/\emptyset)=\mu(a/\emptyset)\times\mu(b/\emptyset)$; then by additivity and Lemma \ref{L: mu generic cont}, the projection $$\mu(abt/\emptyset)\rightarrow\mu(ab/\emptyset)=\mu(a/\emptyset)\times\mu(b/\emptyset)$$ is surjective.
    
    It remains to show that there is at most one $f\in F$ sending $a'$ to $b'$. But this is exactly the content of Theorem \ref{T: partition}: if $f,g\in F$ both send $a'$ to $b'$, then there are $t',t''\in\mu(t/\emptyset)$ so that $\mu(abt/\emptyset)_{t'}$ is the graph of $f$ and $\mu(abt/\emptyset)_{t''}$ is the graph of $g$. Then both $\mu(abt/\emptyset)_{t'}$ and $\mu(abt/\emptyset)_{t''}$ contain $(a',b')$ -- and since $\tp(ab/t)$ is topologically 1-based, Theorem \ref{T: partition} gives that these two fibers are equal. That is, $f$ and $g$ have the same graph, and so they are the same function.
    \end{proof}

    \subsection{Pre-Group Configurations}

    We now aim toward a topological analog of a group configuration theorem. Toward that end, we define:
    
    \begin{definition}
        A \textit{one-dimensional topological pre-group configuration} in $\mathcal M$ is a 5-tuple $(a,b,c,t,u)$ such that all of the following hold:
        \begin{enumerate}
            \item Each of $a,b,c,t,u$ has dimension 1.
            \item $\dim(abctu)=3$.
            \item $a$ and $b$ are interalgebraic over $t$, and $b$ and $c$ are interalgebraic over $u$.
        \end{enumerate}
    \end{definition}

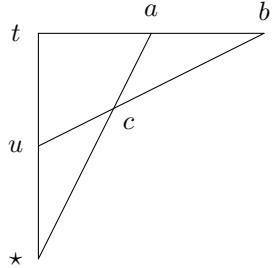
 \[   \groupconf{t}{u}{\star}{a}{b}{c} \]

    Thus, a one-dimensional topological pre-group configuration consists of all but one point (denoted $\star$ in the above diagram) in a conventional one-dimensional group configuration. One can easily check that all dimension computations in a one-dimensional pre-group configuration match those in a usual group configuration. Precisely, we leave the following to the reader:

    \begin{lemma}
        Let $(a,b,c,t,u)$ be a one-dimensional topological pre-group configuration in $\mathcal M$. Separately, let $a',t',u'$ be $\mathbb Q$-linearly independent real numbers, let $b'=a'+t'$, and let $c'=b'+u'$. Then the map $f:\{a,b,c,t,u\}\mapsto\{a',b',c',t',u'\}$ -- sending each element to its primed counterpart -- preserves all dimensions. More precisely, if $A,B\subset\{a,b,c,d,e\}$, then $\dim(A/B)$ (in the sense of $\mathcal M$) is the same as the pregeometry dimension of $f(A)$ over $f(B)$ in the sense of the strongly minimal group $(\mathbb R,+)$. 
    \end{lemma}

    In particular, all dimension computations in a one-dimensional topological pre-group configuration have the additivity property:

    \begin{corollary}\label{C: gconf is additive}
        Let $(a,b,c,t,u)$ be a one-dimensional pre-group configuration. If $A,B\subset\{a,b,c,t,u\}$, then $(A,B)$ is additive.
    \end{corollary}
    \begin{proof} The analogous property holds in $(\mathbb R,+)$, by exchange.
    \end{proof}

    \subsection{Group Configurations}

    We now define what it means for a one-dimensional topological pre-group configuration $(a,b,c,t,u)$ to be an actual topological group configuration. In a stable theory, the requirement is the existence of a non-algebraic tuple $v\in\acl(ac)\cap\acl(tu)$ (i.e. the sixth point of a conventional group configuration). In the topological setting, this is (a priori) too much to ask for. On the other hand, the same condition can be expressed using 1-basedness: if such a $v$ as above exists in a stable context, then $v$ is interalgebraic with $\operatorname{Cb}(\operatorname{stp}(ac/tu))$, which implies that $\operatorname{stp}(ac/tu)$ is 1-based (and conversely, if $\operatorname{stp}(ac/tu)$ is 1-based, one can take its canonical base for $v$). Thus, in the topological setting, we simply adapt the version of the definition using 1-basedness:

    \begin{definition}
        Let $(a,b,c,t,u)$ be a one-dimensional topological pre-group configuration. We say that $(a,b,c,t,u)$ is a \textit{one-dimensional topological group configuration} if $\tp(ac/tu)$ is topologically 1-based.
    \end{definition}

 We now show:

    \begin{theorem}\label{T: gconf}
        Let $(a,b,c,t,u)$ be a one-dimensional topological group configuration in $\mathcal M$. Then there is an $\mathcal M$-type-definable group structure on the set $\mu(a/\emptyset)$ with identity element $a$ (and similarly for $b$ and $c$).
    \end{theorem}

    \begin{proof}
        We will build an $\mathcal M$-type-definable regular groupoid spine $(I,<,\mathcal X,R,\operatorname{Mor})$ in the structure $\mathcal N$, where:
        \begin{itemize}
            \item $I=\{1,2,3\}$ with the usual order.
            \item $X_1=\mu(a/\emptyset)$, $X_2=\mu(b/\emptyset)$, and $X_3=\mu(c/\emptyset)$.
            \item $R=\{(i,j):i<j\}$.
        \end{itemize}
        The Theorem will then follow by Theorem \ref{T: combinatorial group conf}. Precisely, Theorem \ref{T: combinatorial group conf} will provide an $\mathcal Ma$-type-definable group structure on the set $X_1=\mu(a/\emptyset)$, with identity element $a$ (and similarly for $b$ and $c$). But since $a$, $b$, and $c$ are tuples from $\mathcal M$, these groups are actually $\mathcal M$-type-definable. 

        Our task, then, is to define the sets $\operatorname{Mor}(i,j)$ for $i<j$. We will do this by applying Lemma \ref{L: family of bijections} to each of $(a,b,t)$, $(b,c,u)$, and $(a,c,tu)$. We leave it to the reader to check the various independence and algebraicity requirements needed to apply Lemma \ref{L: family of bijections} in each of these cases; note that there is no need to check the additivity requirements, as Corollary \ref{C: gconf is additive}  gives them for free.

        So, in order to apply Lemma \ref{L: family of bijections}, the main thing to check is that each of $\tp(ab/t)$, $\tp(bc/u)$, and $\tp(ac/tu)$ is topologically 1-based. But indeed:
        
        \begin{itemize}
            \item By assumption $t\in\acl(ab)$, so $\dim(t/ab)=0$, and thus $\dim(t/\operatorname{germ}(ab/t))$ must also be 0. Thus $\tp(ab/t)$ is topologically 1-based.
            \item By the same reason as above, $\tp(bc/u)$ is also topologically 1-based.
            \item Finally, $\tp(ac/tu)$ is topologically 1-based because of the assumption that $(a,b,c,t,u)$ is a one-dimensional topological group configuration.
        \end{itemize}
        
        So, we may legally apply Lemma \ref{L: family of bijections} in each of the above instances. Denote the three resulting families of bijections by $\operatorname{Mor}(i,j)$ (in the obvious way, so that $\operatorname{Mor}(i,j)$ is the family of maps from $X_i$ to $X_j$).
        
        If we can show that the data defined so far gives a groupoid spine, then it will automatically be $\mathcal M$-type-definable and regular. That is:
        \begin{itemize}
            \item Each $X_i$ is an infinitesimal neighborhood, so is $\mathcal M$-type-definable.
            \item Similarly, each $\operatorname{Mor}(i,j)$ is represented as in Definition \ref{D: weak type def} by an $\mathcal M$-type-definable family -- precisely  by $\mu(abt/\emptyset)$, $\mu(bcu/\emptyset)$, and $\mu(actu/\emptyset)$, respectively.
            \item By Lemma \ref{L: family of bijections}, each $\operatorname{Mor}(i,j)$ has the required property for the definition of regularity.
        \end{itemize}
        Thus, our only remaining task is showing that $(I,<,\mathcal X,R,\operatorname{Mor})$ indeed gives a groupoid spine. And for this, the only thing to check is closure under composition. So let $f\in\operatorname{Mor}(1,2)$, and let $g\in\operatorname{Mor}(2,3)$. We want to show that $g\circ f\in\operatorname{Mor}(1,3)$.

        Since $f\in\operatorname{Mor}(1,2)$, there is $t'\in\mu(t/\emptyset)$ so that the fiber $\mu(abt/\emptyset)_{t'}$ is the graph of $f$. Similarly, since $g\in\operatorname{Mor}(2,3)$, there is $u'\in\mu(u/\emptyset)$ so that the fiber $\mu(bcu)_{u'}$ is the graph of $g$. We will show that the fiber $\mu(actu/\emptyset)_{t'u'}$ is the graph of $g\circ f$. Let $h:\mu(a/\emptyset)\rightarrow\mu(c/\emptyset)$ be the bijection whose graph is the fiber above $(t',u')$. So we want to show that $g\circ f=h$.

        To show this, let $a'\in\mu(a/\emptyset)$. Note that $abctu\in\acl(atu)$ -- thus $\dim(atu)=3$, and thus $\mu(atu/\emptyset)=\mu(a/\emptyset)\times\mu(t/\emptyset)\times\mu(u/\emptyset)$. Moreover, by Lemma \ref{L: mu generic cont}, the projection gives a homeomorphism $$\mu(abctu/\emptyset)\rightarrow\mu(atu/\emptyset)=\mu(a/\emptyset)\times\mu(t/\emptyset)\times\mu(u/\emptyset).$$ Since $a'\in\mu(a/\emptyset)$, $t'\in\mu(t/\emptyset)$, and $u'\in\mu(u/\emptyset)$, we conclude that $(a',t',u')\in\mu(atu/\emptyset)$, and thus there are $b',c'$ so that $(a',b',c',t',u')\in\mu(abctu/\emptyset)$. Thus:
        \begin{itemize}
            \item $(a',b',t')\in\mu(abt/\emptyset)$, and so $b'=f(a')$;
            \item Similarly, $(b',c',u')\in\mu(bcu/\emptyset)$, and so $c'=g(b')$;
            \item And finally, $(a',c',t',u')\in\mu(actu/\emptyset)$, and so $c'=h(a')$.
        \end{itemize}
        
        It now follows that $g\circ f(a')=h(a')$. Since $a'\in\mu(a/\emptyset)$ was arbitrary, it follows that $g\circ f=h$, and this proves the theorem.
        \end{proof}

\subsection{Building a Topological Group Configuration}

We now show that, combined with a non-triviality assumption, topological 1-basedness implies the type-definability of an infinite group.

\begin{definition}\label{D: nontrivial}
Say that $\mathcal M$ is \textit{non-trivial} if for some parameter set $A$, there are $a,b,c\in M$ with $\dim(abc/A)=2$ and each of $a,b,c$ is algebraic over $A$ with the other two.
\end{definition}

    \begin{corollary}\label{C: 1-based gp}
        Assume $\mathcal M$ is non-trivial and topologically 1-based. Then there are $a\in M$ with $\dim(a)=1$, and an infinite $\mathcal M$-type-definable group structure on the set $\mu(a)$, with identity element $a$.  
    \end{corollary}
    \begin{proof}
        Let $A$ be a parameter set and $(a,b,t)\in M^3$ with $\dim(abt/A)=2$ and each of $a,b,t$ algebraic over the other two. By Lemma \ref{L: constants}, we may assume $A=\emptyset$. By Lemma \ref{L: additive extension}, there is a pair $(c,u)\models\tp(at/b)$ so that $$\dim(abctu)=\dim(abt)+\dim(at/b)=3.$$ It follows easily that $(a,b,c,t,u)$ is a one-dimensional pre-group configuration in $\mathcal M$. Since $\mathcal M$ is topologically 1-based, $(a,b,c,t,u)$ is then automatically a one-dimensional topological group configuration in $\mathcal M$. Now apply Theorem \ref{T: gconf} to get an $\mathcal M$-type-definable group structure on $\mu(a/\emptyset)$ with identity element $a$, and note that since $\dim(a)=1$, $\mu(a/\emptyset)=\mu(a)$.
    \end{proof}

    We can also give some more local analysis:

\begin{definition}
    Let $p=\tp(a/A)$ where $\dim(a/A)=1$.
    \begin{enumerate}
        \item Say that $p$ is \textit{trivial} if there are $(a,b,c)\models p$ with $\dim(abc/A)=3$ and each of $a,b,c$ algebraic over $A$ with the other two.
        \item Say that $p$ is \textit{topologically locally modular} if for any $a,b\models p$ and any real tuple $c$, if $\dim(a/Ac)=\dim(b/Ac)=1$ and $a$ and $b$ are interalgebraic over $Ac$, then $\tp(ab/Ac)$ is topologically 1-based over $A$.
    \end{enumerate}
\end{definition}

So $p$ is topologically locally modular if, at the level of germs and topologically 1-basedness, there are no large families of plane curves inside $p^2$. Then the exact same argument gives:

\begin{corollary}\label{C: loc mod gp}
    Let $p=\tp(a/A)$ where $\dim(a/A)=1$. Assume that $p$ is non-trivial and topologically locally modular. Then there is an $\mathcal M$-type-definable group structure on the set $\mu(a/A)$, with identity element $a$. 
\end{corollary}

\section{Topologizing the Group}

\textbf{Throughout this section, assume $e\in M$ with $\dim(e)=1$. Moreover, assume $(G_0,\cdot)$ is an $\mathcal M$-type-definable group with universe $\mu(e)$ and identity element $e$.} 

Our goal will be to replace the group structure on $\mu(e)$ with a suitable type-definable group over a small set in $\mathcal M$. A key feature will be that the group we build is a topological group with the topology inherited from the given t-minimal topology $\tau$.

\subsection{Step 1: A Type-Definable Group in $\mathcal M$}

First we build a type-definable supergroup of $G_0$, defined over a small set in $\mathcal M$:

\begin{lemma}\label{L: open supergroup}
    Let $A$ be any parameter set such that $\dim(e/A)=1$. Then there are a countable set $B$, and a $B$-type-definable group $(G,\cdot)$ with $G_0\leq G$ so that the following hold:
    \begin{enumerate}
        \item $G$ is an open subset of $M$.
        \item Every element of $G$ realizes $\tp(e/A)$.
    \end{enumerate}
\end{lemma}

\begin{proof}
    Since $\dim(e/A)=1$, every formula in $\tp(e/A)$ holds on an open neighborhood of $e$. So by Fact \ref{F: int of opens}, there is an open neighborhood $U_0$ of $e$ all of whose elements realize $\tp(e/A)$. Shrinking $U_0$ if necessary, we may assume it is $\mathcal M$-definable.

    Since $G$ is type-definable over $\mathcal M$, its operation and inverse are given by (restrictions of) $\mathcal M$-definable functions, say $(x,y)\models m(x,y)$ and $x\mapsto i(x)$.

    Now recall that $G_0=\mu(e)$ is the intersection of all $\mathcal M$-definable open neighborhoods of $e$. By compactness, there is such a neighborhood $U_1\subset U_0$ satisfying the universal group axioms. That is:
    \begin{itemize}
        \item If $a,b,c\in U_1$ then $m(m(a,b),c)$ and $m(a,m(b,c))$ are both defined and equal.
        \item If $a\in U_1$ then $m(a,e)=m(e,a)=a$.
        \item If $a\in U_1$ then $i(a)$ is defined and $m(a,i(a))=m(i(a),a)=e$.
    \end{itemize}
    We now build a decreasing sequence of definable open neighborhoods $U_i\subset U_1$ (for $i\geq 2$) whose intersection is closed under $m$ and $i$, and so is a group. To do this, it suffices to check:
    \begin{claim}
        Let $W$ be any $\mathcal M$-definable open neighborhood of $e$. Then there is an $\mathcal M$-definable open neighborhood $V\subset U_1$ of $e$ such that $m(V^2),i(V)\subset W$.
    \end{claim}
    \begin{proof}
        Since $W$ is $\mathcal M$-definable, we have $\mu(e)\subset W$, and thus all products and inverses of elements of $\mu(e)$ belong to $W$. By compactness, the same holds after replacing $\mu(e)$ with some $\mathcal M$-definable neighborhood of $e$.
    \end{proof}
    Now if we have constructed all $U_i$, then clearly $G=\bigcap U_i$ is a group with the operations $m$ and $i$. Thus $G_0$ is a subgroup. Since $G$ is an intersection of countably many definable open sets, it is open (by Fact \ref{F: int of opens}). Finally, $G$ is defined over a countable set in $\mathcal M$ (namely, the parameters used to define each $U_i$).
    \end{proof}

    \subsection{Step 2: A Locally Topological Group}

    Recall that a version of generic continuity holds in our setting (see Lemma \ref{L: mu generic cont}). It follows that the group $G$ built above is a `generically topological group': the operation and inverse are almost everywhere continuous. In \cite{MarikovaGps}, Marikova initiated a general method for turning a `generically topological' group into a genuinely topological group. We recall one version of Marikova's method, following \cite{CasHasTconv}:

    \begin{definition}\label{D: G_d}
        Let $G$ be a group, and $d\in G$. By $G_d$, we mean the group on underlying set $G$ with operation $(x,y)\mapsto xd^{-1}y$ and inverse $x\mapsto dx^{-1}d$. 
    \end{definition}

    The group $G_d$ is obtained from $G$ by a relabelling, translating every element by $d$. Thus $d$ is the identity of $G_d$.

    \begin{definition}
        Let $G$ be a group whose underlying set is also a topological space.
        \begin{enumerate}
            \item We say $G$ is \textit{generically topological} if there are dense open sets $U\subset G^2$ and $V\subset G$ such that $(x,y)\mapsto xy$ and $x\mapsto x^{1}$ are continuous on $U$ and $V$, respectively.
            \item We say $G$ is \textit{locally topological} if there is an open neighborhood $U\subset G$ of the identity such that $(x,y)\mapsto xy$ and $x\mapsto x^{-1}$ are continuous on $U^2$ and $U$, respectively.
        \end{enumerate}
    \end{definition}
The following is from \cite[\S 7]{CasHasTconv}:
    \begin{fact}\label{F: marikova}
        Let $G$ be a generically topological group. Then there is a dense open set $W$ such that $G_d$ is locally topological for all $d\in W$. 
    \end{fact}

We also need the following reinterpretation of generic continuity:

\begin{lemma}\label{L: open dense continuity}
    Let $U\subset M^n$ be definable and open, and let $f:U\rightarrow M$ be definable. Then $f$ is continuous on a dense open set. 
    \end{lemma}
    \begin{proof}
        Let $U'$ be the set of $u\in U$ so that $f$ is continuous on a neighborhood of $u$. We show that $U'$ is dense. If not, there is a definable open $V\subset U$ disjoint from $U'$. Let $A$ be a set so that $f$ and $V$ are $A$-definable. Since $V$ is open, there is $v\in V$ with $\dim(v/A)=n$. Then by generic continuity, Corollary \ref{C: generic continuity}, $f$ is continuous on a neighborhood of $v$, a contradiction.
    \end{proof}

    We conclude:

    \begin{lemma}\label{L: loc top gp}
        Let $G$ be a type-definable group over a small set $A$, whose universe is an open subset of $M$. Then for any $d\in G-\acl(A)$, the group $G_d$ is locally topological with the topology inherited from $M$.
    \end{lemma}
    \begin{proof}
        The operation and inverse of $G$ are given by $A$-definable functions, say $m(x,y)$ and $i(x)$. By compactness, there is an $A$-definable $X\supset G$ so that $m$ and $i$ are defined on elements of $X$. By Lemma \ref{L: open dense continuity}, $m$ and $i$ are continuous on dense open subsets of $X^2$ and $X$, respectively. But $G$ is open in $X$, so it follows that $m$ and $i$ are also continuous on dense open subsets of $G^2$ and $G$, respectively. Thus $G$ is generically topological. By Fact \ref{F: marikova}, there is a dense open set $W\subset G$ so that $G_d$ is locally topological for all $d\in W$. Now let $Y$ be the set of all $d\in X$ so that the maps $(x,y)\mapsto xd^{-1}y$ and $x\mapsto dx^{-1}d$ defined and continuous on neighborhoods of $(d,d)$ and $d$, respectively. Then $Y$ contains $W$ and is $A$-definable.

        Now let $d\in G-\acl(A)$. If $G_d$ is not locally topological, then $d\notin Y$, and thus by t-minimality, $X-Y$ contains a neighborhood of $d$. Since $G$ is open, $X-Y$ contains a neighborhood of $d$ in $G$. This contradicts that $W\subset Y$ and $W$ is dense.
    \end{proof}

    \subsection{Step 3: A Topological Group}

    Finally, we show that from a locally topological group, one can always extract a genuinely topological group:

    \begin{lemma}\label{L: top gp}
        Let $G$ be a locally topological type-definable group whose underlying set is an open subset of $G$. Then there is a type-definable topological open subgroup of $G$ defined over a countable set.
    \end{lemma}
\begin{proof}
    By Fact \ref{F: int of opens}, there is a definable neighborhood $U_0$ of the identity contained in $G$. Since $G$ is locally topological, there is a definable open neighborhood $U_1\subset U_0$ of the identity such that the operation and inverse are continuous on elements of $U_1$. By continuity, there is a definable open neighborhood $U_2\subset U_1$ over the identity such that $U_2\cdot U_2$ and $U_2^{-1}$ are contained in $U_1$. Continue in this manner, constructing definable open neighborhoods $U_2\supset U_3\supset U_4\dots$ of the identity such that $U_{i+1}\cdot U_{i+t}\subset U_i$ and $U_{i+1}^{-1}\subset U_i$. Then $\bigcap U_i$ satisfies the requirements of the lemma.
\end{proof}

We conclude the main result of this subsection. Recall that we have a fixed $\mathcal M$-type-definable group structure $G_0$ on $\mu(e)$, with $e$ as the identity. Now we show:

\begin{proposition}\label{P: top gp}
    Let $A$ be any parameter set so that $\dim(e/A)=1$. Then there are a countable set $B$, and a $B$-type-definable group $G$, such that:
    \begin{enumerate}
        \item $e$ is the identity of $G$.
        \item $G$ is open in $M$.
        \item $G$ is a topological group with the topology inherited from $M$.
        \item Every element of $G$ realizes $\tp(e/A)$.
    \end{enumerate}
\end{proposition}
\begin{proof}
  By Lemma \ref{L: open supergroup}, we can find $B$ and $G$ satisfying all of the given requirements except being topological. Let $d\in G-\acl(B)$. Then the group $G_d$ is $Bd$-type-definable and topological. So the pair $(G_d,Bd)$ satisfies all requirements of the proposition except that its identity element is $d$, not $e$. But $d$ and $e$ have the same type over $A$, so there is an automorphism $\sigma$ of $\mathcal M$ fixing $A$ and satisfying $\sigma(d)=e$. Now the pair $(\sigma(G),\sigma(Bd))$ satisfies the requirements of the proposition.   
\end{proof}

\subsection{Concluding}

We now give our main results at the level we have reached thus far. So, we drop the assumptions from the beginning of this section; but we retain our fixed pair $\mathcal M$ and $\mathcal N$ of t-minimal structures with independent neighborhoods.

\begin{theorem}\label{T: top gp 1-based}
    Assume $\mathcal M$ is non-trivial and 1-based. Then there are a countable parameter set $B$, and a $B$-type-definable group $(G,\cdot)$ such that:
    \begin{enumerate}
        \item $G$ is a non-empty open subset of $M$.
        \item $(G,\cdot)$ is a topological group with the subspace topology inherited from $M$.
    \end{enumerate}
\end{theorem}
\begin{proof}
    By Corollary \ref{C: 1-based gp}, there are $a\in M$ with $\dim(a)=1$, and an $\mathcal M$-type-definable group structure on $\mu(a)$. Now apply Proposition \ref{P: top gp}.
\end{proof}

\begin{theorem}\label{T: top gp loc mod}
    Let $a\in M$ and $A$ a parameter set so that $\dim(a/A)=1$, and $p=\tp(a/A)$ is non-trivial and topologically locally modular. Then there are a countable set $B$, and a $B$-type-definable group $(G,\cdot)$ such that:
    \begin{enumerate}
        \item The identity element of $G$ is $a$.
        \item $G$ is a non-empty open subset of $M$, and is contained in the realizations of $p$.
        \item $(G,\cdot)$ is a topological group with the subspace topology inherited from $M$.
    \end{enumerate}
\end{theorem}

\begin{proof}
    By Corollary \ref{C: loc mod gp}, there is an $\mathcal M$-definable group structure on the set $\mu(a/A)$, with $a$ as the identity. Now apply Proposition \ref{P: top gp}.
\end{proof}

\subsection{Viscerality of T-minimal Groups}

Before moving on, we point out an isolated fact that falls out of our work in this section and may be of interest. Roughly, we show that \textit{every} t-minimal expansion of a group with independent neighborhoods is \textit{already} both topological and visceral, at least up to editing the topology at finitely many points. 

Recall that $\mathcal M$ is a fixed sufficiently saturated t-minimal structure with the independent neighborhood property.

\begin{theorem}\label{T: the group is visceral}
    Assume that $\mathcal M=(M,\cdot,e,...)$ is an expansion of a group. Then there is a definable topology $\tau'$ on $\mathcal M$ such that:
    \begin{enumerate}
        \item With respect to $\tau'$, $\mathcal M$ is a visceral topological group.
        \item $\tau'$ agrees with $\tau$ on a cofinite subset of $M$.
    \end{enumerate}
\end{theorem}
\begin{proof}
    It is enough to find $\tau'$ satisfying (1). Then by Johnson's appendix (Theorem \ref{wj-indep-nbhd}), $\mathcal M$ has the independent neighborhood property with respect to $\tau'$; therefore, by Lemma \ref{L: unique top}, $\tau'$ agrees with $\tau$ on a cofinite set.

    The rest of the proof finds $\tau'$ satisfying (1). First, let $d\in M$ be generic, and consider the group $M_d$ as in Definition \ref{D: G_d}. By Lemma \ref{L: loc top gp}, $M_d$ is locally topological with respect to $\tau$. Moreover, if we construct $\tau'$ satisfying (1) for $M_d$, the same $\tau'$ will work for $M$. Thus, we assume moving forward that $M$ is locally topological with respect to $\tau$.
    
    By Lemma \ref{L: top gp}, there is thus a type-definable open topological subgroup $H\leq M$, defined over a countable set. Absorbing parameters, we assume $H$ is $\emptyset$-definable. We then obtain a definable uniform structure on $M$ by taking translations of basic neighborhoods of $e$. Call the resulting topology $\tau'$. So $\tau'$ is translation invariant (both left and right), and agrees with $\tau$ on $H$.
    
    We will show that $(M,\tau')$ is a topological group, and $\mathcal M$ is visceral with respect to $\tau$'. We start with the first of these goals, i.e. showing $(M,\tau')$ is topological. By \cite[Lemma 2.11]{HaHaPeGps}, it is enough to show that for each $g\in M$, conjugation by $g$ is a \textit{local homeomorphism}: this means there are $\tau'$-open neighborhoods $e\in U,V\subset M$ so that $u\mapsto gug^{-1}$ induces a homeomorphism $U\rightarrow V$. Since $\tau$ and $\tau'$ agree on $H$, it is equivalent to find such $U,V$ which are $\tau$-open.
    
    Before proceeding, we collect some easy claims. Throughout, by a \textit{generic} element of $M$ over $A$, we mean an element $g$ with $\dim(g/A)=1$. 
    
    \begin{claim}\label{C: product of two generics} Every element of $M$ is a product of two generic elements.
    \end{claim}
    \begin{proof}
        Given $g\in M$, let $h$ be generic over $g$, and write $g=h(h^{-1}g)$.
        \end{proof}
    \begin{claim}\label{C: translation of mus}
        Let $g,h\in G$ with $\dim(g,h)=2$. Then the maps $x\mapsto gx$ and $x\mapsto xg$ define homeomorphisms $\mu(h)\rightarrow\mu(gh)$ and $\mu(h)\rightarrow\mu(hg)$. \end{claim}
        \begin{proof} This follows by Lemma \ref{L: mu generic cont}, noting that $\dim(h/g)=\dim(h)=1$, so $\mu(h/g)=\mu(h/\emptyset)=\mu(h)$ (and similarly for $gh$).
        \end{proof}
    \begin{claim}\label{C: local translation of mus} Let $h\in H$ be generic. Then the maps $x\mapsto hx$ and $x\mapsto xh$ give homeomorphisms $\mu(e)\rightarrow\mu(h)$.
    \end{claim}
            \begin{proof}
                This follows since $H$ is topological, so these maps are homeomorphisms, and thus they preserve all $\mathcal M$-defianble open sets (see also Lemma \ref{L: mu translation}).
            \end{proof}
    \begin{claim}\label{C: local generic mu trans}
        Let $g\in M$ be generic. Then $x\mapsto gx$ and $x\mapsto xg$ each give homeomorphisms $\mu(e)\rightarrow\mu(g)$.
    \end{claim}
    \begin{proof}
        Let $h\in H$ with $\dim(g,h)=2$, and write $g=ah$ where $a=gh^{-1}$. So also $\dim(ah)=2$. By Claims \ref{C: local translation of mus} and \ref{C: translation of mus}, $x\mapsto gx=ahx$ is now a composition of translation homeomorphisms $\mu(e)\rightarrow\mu(h)\rightarrow\mu(g)$ (translating by $h$ then by $a$). The argument for $x\mapsto xg$ is similar, using $b$ so that $g=hb$.
    \end{proof}
            
    Finally, we conclude:
    \begin{claim}
        Let $g\in M$. Then conjugation by $g$ induces a homeomorphism between neighborhoods of $e$. Thus, $(M,\tau')$ is a topological group. 
    \end{claim}
    \begin{proof}
        By Claim \ref{C: product of two generics}, we may assume $g$ is generic. Now using Claim \ref{C: local generic mu trans}, write conjugation by $g$ as a composition of translation homeomorphisms $\mu(e)\rightarrow\mu(g)\rightarrow\mu(e)$, sending $x\mapsto gx\mapsto gxg^{-1}$. So conjugation by $g$ gives a homeomorphism $\mu(e)\rightarrow\mu(e)$, and thus (by compactness) also a homeomorphism between neighborhoods of $e$.
    \end{proof}

    We now turn to our second main goal mentioned above:
    \begin{claim}
        $\mathcal M$ is t-minimal, and thus visceral, with respect to $\tau'$.
    \end{claim}
    \begin{proof}
        Viscerality follows from t-minimality by definition, since $(M,\tau')$ is uniform. We show t-minimality. Recall this means showing that for definable $X\subset M$, $X$ has $\tau$'-interior if and only if it is infinite.
        
        First suppose $X$ has $\tau'$-interior. Then by translation invariance, some translation $Y$ of $X$ contains a $\tau'$-neighborhood of $e$. Since $\tau$ and $\tau'$ agree on $H$ (i.e. near $e$), $Y$ also contains a $\tau$-neighborhood of $e$. By t-minimality of $(\mathcal M,\tau)$, $Y$ is infinite -- thus so is $X$.
        
        Now suppose $X$ is infinite. Say $X$ is $A$-definable, and pick $a\in X$ with $\dim(a/A)=1$. By t-minimality of $(\mathcal M,\tau)$, $X$ contains a neighborhood of $a$, and thus contains $\mu(a)$. By Claim \ref{C: local generic mu trans}, $Y=a^{-1}X$ contains $\mu(e)$, so in particular contains a $\tau$-neighborhood of $e$. Since $\tau$ and $\tau'$ agree near $e$, $Y$ contains a $\tau'$-neighborhood of $e$. Then by translation invariance, $X$ contains a $\tau'$-neighborhood of $a$.
    \end{proof}

    We have now shown that $(\mathcal M,\tau')$ is a visceral topological group, which completes the proof of (1), and thus of the theorem.
    \end{proof}

\begin{remark} One could also prove Theorem \ref{T: the group is visceral} by quoting upcoming results of Johnson. According to private communication, he shows that any visceral expansion of a group admits a definable visceral group topology. This allows us to skip the proof that $(M,\tau')$ is a topological group above, concentrating only on the proof that $(\mathcal M,\tau')$ is t-minimal. Indeed, once we know $(\mathcal M,\tau')$ is t-minimal (and thus visceral), we obtain from Johnson a second definable visceral topology $\tau''$ with respect to which $\mathcal M$ is a topological group. (In the end, one can show $\tau''=\tau'$ using independent neighborhoods and translation invariance, so Johnson's argument really gives an alternate proof that $\tau'$ is a group topology).
\end{remark}

\section{The Structure of the Group}

\textbf{In this section, we assume that $G$ is a type-definable group on a non-empty open subset of $M$, and is a topological group with the subspace topology inherited from $M$. Absorbing parameters, we assume $G$ is $\emptyset$-type-definable. We denote the identity of $G$ by $e$.}

Our goal is to prove results analogous to the Hrushovski-Pillay classification of 1-based stable groups (\cite{HrPi87}). In a very coarse language, we will show:

\begin{enumerate}
    \item $G$ is locally linear.
    \item $G$ is locally abelian.
    \item There are no large type-definable families of germs of subgroups of $G$.
\end{enumerate}

\subsection{The Infinitesimal Subgroup}

Let $\mu=\mu(e)$ be the intersection of all $\mathcal M$-definable open neighborhoods of $e$. So $\mu$ is open and type-definable over $\mathcal M$. One easily checks:

\begin{lemma}\label{L: mu subgroup}
    $\mu$ is a subgroup of $G$.
\end{lemma}
\begin{proof}
    Let $a,b\in\mu$. We show $ab^{-1}\in\mu$. Indeed, let $V$ be any $\mathcal M$-definable open neighborhood of $e$. Since $G$ is topological, the set $\{(x,y):xy^{-1}\in V\}$ is an $\mathcal M$-definable open neighborhood of $e^2$, so contains a set $U^2$ where $U$ is an $\mathcal M$-definable open neighborhood of $e$. By definition of $\mu$ we get $a,b\in U$, and so $ab^{-1}\in V$. Since $V$ was arbitrary, $ab^{-1}\in\mu$. 
\end{proof}

\begin{lemma}\label{L: mu translation}
    Let $a\in G$. Then $\mu(a)=a\cdot\mu=\mu\cdot a$.
\end{lemma}
\begin{proof}
    The maps $x\mapsto ax$ and $x\mapsto xa$ are $\mathcal M$-definable homeomorphisms, so preserve $\mathcal M$-definable open sets.
\end{proof}

\begin{lemma}\label{L: big subgroups}
    Let $A$ be a parameter set, and $H\leq G^n$ an $A$-type-definable subgroup. Then the following are equivalent:
    \begin{enumerate}
        \item $H$ has non-empty interior.
        \item $H$ is open.
        \item $\mu^n\subset H$.
        \item $\dim(H)=n$.
    \end{enumerate}
\end{lemma}
\begin{proof}
$(1)\implies(2)$: Let $a$ be an interior point of $H$. For any $b\in H$, the map $x\mapsto ba^{-1}x$ sends $a$ to $b$ and is a homeomorphism of both $H$ and $G$. So $b$ is also an interior point of $H$.

$(2)\implies(3)$: If $H$ is $\mathcal M$-definable and open, it contains the intersection of all $\mathcal M$-definable open neighborhoods of $e$, which is precisely $\mu^n$.

$(3)\implies(4)$: If $H$ contains $\mu^n$, then $$n=\dim(\mu^n)\leq\dim(H)\leq\dim(G^n)=n,$$ so $\dim(H)=n$.

$(4)\implies(1)$: If $\dim(H)=n$, then there is $h\in H$ with $\dim(h/A)=n$. Then every $A$-definable superset of $h$ contains a neighborhood of $h$, and by Fact \ref{F: int of opens}, so does $H$. 
\end{proof}

\subsection{Characterization of Topological 1-basedness}

In this subsection we prove the main geometric result concerning topological 1-basedness in $G$. This should be seen as analogous to the Hrushovski-Pillay characterization of 1-based stable groups.

We will need the following, which is an elementary group-theoretic exercise:

\begin{fact}\label{F: easy group facts}
    Let $H$ be a group, and $X\subset H$ a non-empty subset. Then the following are equivalent:
    \begin{enumerate}
        \item Any two left translates of $X$ are equal or disjoint.
        \item Any two right translates of $X$ are equal or disjoint.
        \item $X$ is a left coset of a subgroup of $G^n$.
        \item $X$ is a right coset of a subgroup of $G^n$.
        \item $X$ is closed under the map $(x,y,z)\mapsto xy^{-1}z$.
    \end{enumerate}
\end{fact}

Recall that in Section \ref{ss: local linearity} we introduced \textit{local linearity} for t-minimal topological groups. We now give this notion a more general definition in the type-definable setting.

\begin{lemma}\label{L: loc lin char}
    Let $a\in G^n$ and $A$ a parameter set. The following are equivalent:
    \begin{enumerate}
        \item There are a parameter set $A$, and an $A$-type-definable subgroup $H\leq G^n$, so that $\operatorname{germ}(a/A)=\operatorname{germ}_a(a\cdot H)$.
        \item There are a parameter set $A$, and an $A$-type-definable subgroup $H\leq G^n$, so that $\operatorname{germ}(a/A)=\operatorname{germ}_a(H\cdot a)$.
        \item $\mu(a/A)$ is a left coset of a subgroup of $\mu^n$.
        \item $\mu(a/A)$ is a right coset of a subgroup of $\mu^n$.
    \end{enumerate}
    \begin{proof}
        (1) and (2) are equivalent by Fact \ref{F: easy group facts}. Similarly, (3) and (4) are equivalent by Fact \ref{F: easy group facts}.

        $(1)\rightarrow(3)$: Let $A$ be a parameter set, and $H\leq G^n$ an $A$-type-definable subgroup, so that $\operatorname{germ}(a/A)=\operatorname{germ}_a(a\cdot H)$. Then $\mu(a/A)=\mu(a)\cap(a\cdot H)$. By Lemma \ref{L: mu translation}, this is the same as $a\cdot(\mu\cap H)$. By Lemma \ref{L: mu subgroup}, $\mu\cap H$ is a subgroup of $\mu^n$. This implies (3).
        
        $(3)\rightarrow(1)$. Assume $\mu(a/A)$ is a left coset of a subgroup. By Fact \ref{F: easy group facts}, $\mu(a/A)$ is closed under $(x,y,z)\mapsto xy^{-1}z$. Now using compactness, inductively build $\mathcal M$-definable d-approximations $X_1\supset X_2\dots$ of $\tp(a/A)$ at $a$ so that $X_{i+1}X_{i+1}^{-1}X_{i+1}\subset X_i$ for all $i$. Then $X=\bigcap X_i$ is closed under $(x,y,z)\mapsto xy^{-1}z$, so is a left coset of the subgroup $a^{-1}X$. Since each $X_i$ realizes $\operatorname{germ}(a/A)$, Fact \ref{F: int of opens} gives that $X$ also realizes $\operatorname{germ}(a/A)$. Finally, $a^{-1}X$ is defined over countably many parameters from $\mathcal M$, so satisfies the requirements of (1).
    \end{proof}
    \end{lemma}

    \begin{definition}
        Let $a\in G^n$ and $A$ a parameter set. We say that $\tp(a/A)$ is \textit{locally linear} if the equivalent conditions of Lemma \ref{L: loc lin char} hold. 
    \end{definition}

    \begin{definition}
        We say that $G$ is \textit{locally linear} if $\tp(a/A)$ is locally linear for all $a\in G^n$ and all parameter sets $A$.
    \end{definition}

    We also give a general definition of topological 1-basedness for $G$:

    \begin{definition}\label{D: 1-based group}
    We say that $G$ is \textit{topologically 1-based} if for every $a\in G^m$, every $b\in G^n$, and every parameter set $A$, $\tp(a/Ab)$ is topologically 1-based over $A$.
\end{definition}

\begin{remark}\label{R: loc mod gp} By Remark \ref{R: reduction to n=2}, if exchange holds then Definition \ref{D: 1-based group} is equivalent to the analogous assertion restricted to the case when $n=2$.
\end{remark}

In this language, we now prove our main analog of the Hrushovski-Pillay theorem on 1-based stable groups:

\begin{theorem}\label{T: loc lin}
     The following are equivalent:
    \begin{enumerate}
        \item $G$ is topologically 1-based.
        \item $G$ is locally linear.
    \end{enumerate}
    In particular, both conditions hold if $\mathcal M$ is topologically 1-based.
\end{theorem}
\begin{proof}
 $(1)\implies(2)$: Let $a\in G^n$ and let $A$ be a parameter set. Let $X=a^{-1}\cdot\mu(a/A)\subset\mu^n$, so that $\mu(a/A)=a\cdot X$. Then let $b\in G^n$ so that $\dim(a,b/A)=\dim(a/A)+n$. So $\dim(b/A)=n$, which implies that $\mu(b/
    A)=\mu(b)=\mu^n\cdot b$. Also, $a$ and $b$ are symmetrically independent over $A$, and thus $$\mu(a,b/A)=\mu(a/A)\times\mu(b/A)=(a\cdot X)\times(\mu^n\cdot b).$$ So for all $u\in\mu^n$, the fiber above $u\cdot b$ in $\mu(a,b/A)$ is $\mu(a/A)=a\cdot X$. 
    
    Now $(a,b)$ and $(a\cdot b,b)$ are interdefinable over $A$ by the map $(x,y)\mapsto(x\cdot y,y)$. It follows that the same map gives a homeomorphism $\mu(a,b/A)\rightarrow\mu(a\cdot b,b/A)$. Thus, for $u\in\mu^n$, the fiber above $u\cdot b$ in $\mu(a\cdot b,b/A)$ is the right translate by $u\cdot b$ of the corresponding fiber in $\mu(a,b/A)$, i.e. $a\cdot X\cdot u\cdot b$. Now, assuming (1), the type $\tp(a\cdot b/Ab)$ is topologically 1-based over $A$. Thus, by Theorem \ref{T: partition}, any two of the sets $a\cdot X\cdot u\cdot b$ for $u\in\mu^n$ are equal or disjoint. The same holds after cancelling $a$ on the left and $b$ on the right. Thus, any two of the sets $X\cdot u$, for $u\in\mu^n$, are equal or disjoint. By Fact \ref{F: easy group facts}, it follows that $X$ is a left coset of a subgroup of $\mu^n$, and thus so is $a\cdot X=\mu(a/A)$.

$(2)\implies(1)$: Let $a\in G^m$, $b\in G^n$, and $A$ a parameter set. Assuming (2), the set $\mu(a,b/A)$ is a left coset of $\mu^{m+n}$. It follows that there is a subgroup $H\leq G^m$ so that each fiber $\mu(a,b/A)_y$ for $y\in\mu(b/A)$ is either empty or a left coset of $H$. Since any two left cosets of $H$ are equal or disjoint, it follows by Theorem \ref{T: partition} that $\tp(a/Ab)$ is 1-based over $A$.
\end{proof}

\subsection{Few Subgroups}

Hrushosvki and Pillay also prove that 1-based stable groups have few definable subgroups, in the sense that every connected definable subgroup is defined over $\operatorname{acl}(\emptyset)$. To adapt to our setting, we make two changes:

\begin{enumerate}
    \item Instead of connected subgroups, we consider germs of subgroups at the identity.
    \item Instead of asking for the germ to be defined over $\acl(\emptyset)$, we require (as in the definition of topological 1-basedness) that the germ is constant on an open set.
\end{enumerate}

Before proving a precise statement, we need a technical lemma. In the general t-minimal setting (even with independent neighborhoods), one can have a definable set -- say over $A$ -- which is not a d-approximation for any of its elements over $A$. In contrast, Lemma \ref{L: projections of cosets} essentially says that groups and their cosets are always d-approximations at generic points:

\begin{lemma} \label{L: projections of cosets}
Let $X$ be an $A$-type-definable left coset of a subgroup of $G^n$ of dimension $d$. Let $a\in X$ with $\dim(a/A)=d$. Let $b$ be a basis of $a$ over $A$, and $\pi$ the corresponding projection sending $a$ to $b$. Then:
\begin{enumerate}
    \item $\pi$ is finite-to-one on $X$.
    \item $X$ realizes $\operatorname{germ}(a/A)$.
\end{enumerate}
\end{lemma}
\begin{proof}
    Note that (1) implies (2) by compactness, Lemma \ref{L: basic witness}, and Fact \ref{F: int of opens} (as cofinally many $A$-definable supersets of $X$ are d-approximations of $\tp(a/A)$).
    
    We show (1). Since $X$ is a left coset, there is an $A$-type-definable subgroup $H\leq G^n$ such that all non-empty fibers of $\pi$ in $X$ are left cosets of $H$ (here we treat fibers a subsets of $G^n$, not $G^{n-d}$, as this will make the argument easier). So it suffices to show that $H$ is finite.
    
    Now assume $H$ is infinite. Then there is $h\in H$ with $\dim(a,h/A)=\dim(a/A)+\dim(H)>\dim(X)$.
    
    Since the fiber $X_b$ is a left coset of $H$ and contains $a$, it follows that $a\cdot h\in X$, and so $\dim(a\cdot h/A)\leq\dim(X)<\dim(a,h/A)$. On the other hand, we claim that $a\cdot h$ is interalgebraic with $(a,h)$ over $A$, which will give a contradiction. Clearly $a\cdot h\in\acl(A,a,h)$. For the converse, since $\pi(a\cdot h)=b$, we have $b\in\acl(A,a\cdot h)$. But $a\in\acl(Ab)$ by definition; thus $a\in\acl(A,a\cdot h)$, and thus $h\in\acl(A,a\cdot h)$ as well, since $h=a^{-1}\cdot(a\cdot h)$.
\end{proof}

We now give our `few subgroups' result:

\begin{theorem}\label{T: few subgroups}
    Assume $G$ is topologically 1-based. Let $Y\subset G^m$ be $A$-type-definable, let $\{H_y:y\in Y\}$ be an $A$-type-definable family of d-dimensional subgroups of $G^n$, and let $b\in Y$. Then the germ of $H_y$ at the identity is constant for all $y\in\mu(b/A)$.
\end{theorem}
\begin{proof}
    First, we need:
    \begin{claim} There is $a\in H_b$ so that $\dim(a,b/A)=\dim(b/A)+d$. 
    \end{claim}
    \begin{proof}
        Let $u\in\mu^d$ with $\dim(u,b/A)=d+\dim(b/A)$. By Lemma \ref{L: projections of cosets}, there is a finite-to-one projection $\pi:H_b\rightarrow G^d$. Then $\pi(H_b)$ has dimension $d$, so by Lemma \ref{L: big subgroups}, $\pi(H_b)$ contains $\mu^d$. In particular, $u\in\pi(H_b)$. Let $a'\in H_b$ with $\pi(a')=u$. Then $a'$ and $u$ are interalgebraic over $Ab$. So $(a',b)$ and $(u,b)$ are interalgebraic over $A$, and thus $\dim(a',b/A)=d+\dim(d/A)$. Finally, since $\mathcal M$ is sufficiently saturated, there is a tuple $a$ from $\mathcal M$ with $\tp(a,b/A)=\tp(a',b/A)$, and thus $\dim(a,b/A)=d+\dim(b/A)$ as well.
    \end{proof}
    Now fix $a$ as in the claim. Since $G$ is topologically 1-based, Theorem \ref{T: loc lin} gives that $\mu(a,b/A)$ is a left coset of a subgroup. In particular, all of the non-empty fibers $\mu(a,b/A)_y$ for $y\in\mu(b/A)$ are left cosets of the same group, say $L\leq G^n$.
    
    Let $y\in\mu(b/A)$. We show that the germ of $H_y$ at the identity is the same as the germ of $L$ at the identity. As $y$ is arbitrary, this will prove the theorem.
    
    Now by the choice of $a$, we have that $(a,b)$ is additive over $A$. So by Lemma \ref{L: mu generic cont}, there is $x$ so that $(x,y)\in\mu(a,b/A)$. By Corollary \ref{C: fiber germ}, $\operatorname{germ}(x/Ay)$ is realized by the fiber $\mu(a,b/A)_y$, and therefore by a coset of $L$. On the other hand, by Lemma \ref{L: projections of cosets}, $\operatorname{germ}(a/Ab)$ is realized by $H_b$; and since $(x,y)\models\tp(a,b/A)$, $\operatorname{germ}(x/Ay)$ is also realized by $H_y$. So $H_y$ agrees with the coset $x\cdot L$ in a neighborhood of $x$; and thus the germ of $H_y$ at the identity is realized by $x^{-1}(x\cdot L)=L$, as desired.
\end{proof}

\subsection{Local Abelianity}

Recall that every 1-based stable group is abelian-by-finite (see \cite[Theorem 3.2]{HrPi87}). This is shown by applying the `few subgroups' result in that context to the family of conjugation maps.

In the topological setting, the analogous statement is that $G$ contains an open abelian subgroup -- or equivalently, $\mu$ is abelian. We now prove this, using essentially the same strategy from the stable case.

\begin{theorem}\label{T: local abelian} Assume $G$ is topologically 1-based.
\begin{enumerate}
    \item There are a countable parameter set $A$, and an $A$-type-definable open abelian subgroup $H\leq G$.
    \item In particular, $\mu$ is abelian.
\end{enumerate}
\end{theorem}
\begin{proof}
    The family of (graphs of) conjugation maps $x\mapsto axa^{-1}$ is a $\emptyset$-type-definable family of subgroups of $G^2$. Let $b\in G$ with $\dim(b)=1$. Then by Theorem \ref{T: loc lin}, the germ of $x\mapsto yxy^{-1}$ at the identity is constant for $y\in\mu(b/\emptyset)$.
    
    Let $H$ be the \textit{local center} of $G$ -- the set of elements $a\in G$ so that $x\mapsto axa^{-1}$ is the identity on a neighborhood of $e$. Then $H$ is a $\emptyset$-type-definable subgroup of $G$. Since any two elements of $\mu(b/\emptyset)$ induce the same conjugation on a neighborhood of $e$, it follows that $\mu(b/\emptyset)\mu(b/\emptyset)^{-1}\subset H$. In particular, $H$ is infinite, and so $\dim(H)=1$. Let $h\in H$ with $\dim(h)=1$, and then let $u\in\mu$ with $\dim(h,u)=2$ (so $h\in M$ and $u\in N$). So $h$ commutes with an $\mathcal M$-definable neighborhood of the identity, and so $h$ commutes with $u$. Now since $\mathcal M$ is sufficiently saturated, there is $v\in G$ from $\mathcal M$ so that $\dim(h,v)=2$ and $h$ and $v$ commute. 
    
    Since $\dim(h,v)=2$, the $\emptyset$-definable condition `$x$ and $y$ commute' holds on all $(x,y)$ in some $\mathcal M$-definable neighborhood of $(h,v)$. In particular, there are $\mathcal M$-definable open neighborhoods, $U$ of $h$ and $V$ of $v$, so that all elements of $U$ commute with all elements of $V$. But the centralizer of $V$ is a subgroup, so since it has non-empty interior, it contains an $\mathcal M$-definable neighborhood of $e$, say $U'$ (see Lemma \ref{L: big subgroups}). Similarly, the centralizer of $U'$ now contains an $\mathcal M$-definable neighborhood $V'$ of $e$. So all elements of $U'$ commute with all elements of $V'$. It follows that $\mu$ is abelian, because $\mu\subset U'\cap V'$. Thus we have shown (2).

    To show (1), inductively construct a sequence $U_1\supset U_2\supset\dots$ of $\mathcal M$-definable open neighborhoods of $e$ so that $U_1\subset U'\cap V'$ and each $U_{i+1}\cdot U_{i+1},U_{i+1}^{-1}\subset U_i$. Then $\bigcap U_i$ is an open type-definable abelian subgroup defined over a countable set. 
\end{proof}

\subsection{Summarizing}

Finally, we end the paper by giving our main theorem in full. Thus, we now drop all underlying assumptions that were present in this section until now.

\begin{theorem}\label{T: main 1-based}
    Let $\mathcal M$ be a sufficiently saturated t-minimal structure with the independent neighborhood property. Assume $\mathcal M$ is non-trivial and topologically 1-based. Then there are a countable parameter set $A$, and an $A$-type-definable abelian group $G$, such that the following hold:
    \begin{enumerate}
        \item $G$ is open in $\mathcal M$.
        \item $G$ is a topological group with the topology inherited from $M$.
        \item $G$ is locally linear.
    \end{enumerate}
\end{theorem}
\begin{proof}
    By Theorem \ref{T: top gp 1-based}, we can satisfy all requirements other than abelianity and local linearity. But local linearity follows by Theorem \ref{T: loc lin}; and by Theorem \ref{T: local abelian}, after further shrinking we can arrange that $G$ is abelian.
\end{proof}

The following may also be worth noting. We leave the details for the interested reader to verify: 
\begin{remark}
    Let $\mathcal G$ be a $t$-minimal structure with independent neighbourhoods. Assume that $\mathcal G$ expands a group, $(G, \cdot, e)$. Let $\tau$ denote the given definable topology on $G$. By  Lemma \ref{L: loc top gp} $G$ is definably isomorphic to a group $G_d$ locally topological with respect to the topology of the $t$-minimal structure $\mathcal G$. By (the proof of)Lemma \ref{L: mu subgroup} the intersection $\mu_d$ of all $G$-definable neighbourhoods of $d$ (the identity element of $G_d$) is a type-definable topological subgroup of $G_d$. A short argument shows that $\mu_d$ is invariant under conjugation by any element of $G_d$ (in the same model over which $\mu_d$ was defined). By compactness, this implies that conjugation by $g\in G_d$ is continuous at $d$ (for $g$ in any model) -- see e.g., the proof of \cite[Proposition 5.17(3)]{HaHaPeGps}. By \cite[Lemma 2.11]{HaHaPeGps}, there is an (automatically uniform) definable group topology $\tau_d$ on $G_d$ extending the topology on $\mu_d$. Since $\mathcal G$ is $t$-minimal, we get that $G_d$ with $\tau_d$ is visceral. Indeed, any infinite definable set $S$ contains a $\tau$-open set $U$, and for any generic $d'\in U$ we get that $\mu_{d'}\subseteq U$ is open and definably  $\tau_d$-homeomorphic to $\mu_d$, so $\tau_d$ open. So $S$ has non-empty $\tau_d$-interior. By the main result of Appendix B this implies that $(G_d, \tau_d)$ has the independent neighborhood property, so by Lemma \ref{L: unique top} it coincides with $\tau$ away from finitely many points. Since $G_d$ is definably isomorphic to $G$, this isomorphism induces a definable group topology $\rho$ on $G$, and by the same argument  $\rho$ and $\tau$ coincide outside finitely many points. In particular, a topological $t$-minimal group admits exactly one $t$-minimal group topology. 
    \end{remark}
 
\appendix

\section{Proving the Independent Neighborhood Property}

Here, we include a brief appendix to demonstrate that various well-known structures have the independent neighborhood property. It is possible that the facts we present here were already known. In this case, our contribution is to give a unified treatment.

\begin{assumption}
    Throughout the appendix, $\mathcal M$ is a sufficiently saturated t-minimal structure. We do not assume $\mathcal M$ has the independent neighborhood property.
\end{assumption}

\subsection{Broadness at a Point}

In his paper, Johnson defines the class of \textit{broad} subsets of $M^n$. Here we give a similar notion but applying locally near one point. This is not quite a direct translation, because Johnson's notion of broadness is invariant under coordinate permutations and ours is not. Nevertheless, our notion will better suit the argument at hand.

\begin{definition}
    For each $a\in M$, we inductively define the class of \textit{$a$-broad} subsets of $M^n$ as follows:
    \begin{itemize}
        \item If $n=1$, then $B\subset M$ is $a$-broad if $a$ is an accumulation point of $B$, i.e. every neighborhood of $a$ contains a point of $B-\{a\}$.
        \item If $n\geq 1$, then $B\subset M^{n+1}=M^n\times M$ is $a$-broad if there is an $a$-broad set $C\subset M$ such that for each $c\in C$, the fiber $X_c\subset M^n$ is $a$-broad.
    \end{itemize}
\end{definition}

\begin{remark} By induction, note that the assertion `$X$ is $a$-broad' is $\emptyset$-definable in families (one only needs the $\emptyset$-definability of the topology).
\end{remark}

The main elementary fact we need about $a$-broad sets is the following:

\begin{lemma}\label{L: broad union}
    Let $a\in M$, and let $\{B_i:i\in I\}$ be a small collection of subsets of $M^n$. If $B=\bigcup_{i\in I}B_i$ is $a$-broad, then some $B_i$ is $a$-broad. 
\end{lemma}
\begin{proof}
    We induct on $n$. If $n=1$, this is just Fact \ref{F: int of opens}: if no $B_i$ is $a$-broad, we can find open $U_i$ containing $a$ and disjoint from $B_i$. Then $\bigcap U_i$ is open and disjoint from $B$, so $B$ is not $a$-broad, a contradiction.
    
    Now suppose $n=m+1$ where $m\geq 1$. Since $B$ is $a$-broad, there is an $a$-broad set $C\subset M$ so that for all $c\in C$ the fiber $B_c\subset M^m$ is $a$-broad. For each $i\in I$, let $C_i$ be the set of $c\in C$ so that the fiber $(B_i)_c$ is $a$-broad. For each $c\in C$, the fibers $(B_i)_c$ union to the $a$-broad set $B_c$, so by induction one of the $(B_i)_c$ is $a$-broad. Thus $\bigcup C_i=C$, and so by the base case some $C_{i_0}$ is $a$-broad. Then by definition, $B_{i_0}$ is $a$-broad.
\end{proof}

\subsection{Algebraicity over Broad Sets, Part 1}

There are many examples of t-minimal theories admitting locally constant definable functions $f:M\rightarrow M$ with infinite image (indeed, every failure of exchange in a t-minimal theory looks something like this). However, no such function can be a choice function: that is, one cannot hope that the constant value taken on an open set $U\subset M$ always belongs to $U$ (otherwise, the image of $f$ would be infinite and discrete, which is impossible). Lemma \ref{L: broad algebraic} below captures this idea on a general level:

\begin{lemma}\label{L: broad algebraic}
    Let $a\in M$, let $A$ be a parameter set, and let $B\subset M^n$ be $a$-broad. If $a\in\acl(Ab)$ for all $b\in B$, then $a\in\acl(A)$.
\end{lemma}
\begin{proof}
    We induct on $n$. The inductive step is straightforward: assume $n=m+1$ for some $m\geq 1$, and further assume the statement is true for all values less than $n$. Let $C\subset M$ be $a$-broad so that $B_c$ is $a$-broad for each $c\in C$. For each $c\in C$, applying the inductive hypotheses to $B_c$ over the parameter set $Ac$ gives  $a\in\acl(Ac)$. Now applying the base case to $C$ gives that $a\in\acl(A)$.

    The bulk of the work is the case $n=1$. So assume $B\subset M$ is $a$-broad and $a\in\acl(Ab)$ for all $b\in B$. By Lemma \ref{L: broad union}, this algebraicity is witnessed uniformly on a broad set. That is, after potentially replacing  $B$ with a smaller (still $a$-broad) subset, we can find an $A$-definable $X\subset M^2$ with the following properties:
    \begin{enumerate}
        \item $(a,b)\in X$ for all $b\in B$.
        \item The second projection $\pi_2:M^2\rightarrow M$ is at most $k$-to-1 on $X$ for some integer $k$.
    \end{enumerate}

    Let $\pi_1:M^2\rightarrow M$ be the first projection. Let $Y$ be the set of $y\in M$ so that the fiber $\pi_1^{-1}(y)\cap X$ is $y$-broad. Note that $Y$ is $A$-definable and contains $a$ (as $\pi^{-1}(a)\supseteq B$). If $a$ is a boundary point of $Y$, then by t-minimality, $a\in\acl(A)$, and we are done. We show that $a$ cannot be an interior point of $Y$. Assume otherwise, and let $U$ be an open neighborhood of $a$ with $U\subset Y$.

    We now inductively construct elements $a_1,a_2,\dots\in U$ and non-empty open sets $U=U_0\supset U_1\supset U_2\dots$ with the following properties:

    \begin{enumerate}
    \item $a_i\in U_j$ whenever $i>j$.
    \item Each $\{a_i\}\times U_i\subset X$.
    \end{enumerate}

    Suppose $i\geq 1$ and we have built all previous $a_j$ and $U_j$. First, let $V_i$ be the intersection of all previous $U_j$ (so if $i=0$ then $V_i=U_0=U$). Since the previous $U_j$ are non-empty and form a descending chain, $V_i$ is non-empty. Now let $a_i$ be any element of $V_i$, noting that this will ensure (1) is satisfied. 

    Next, by definition $a_i\in U_0\subset Y$, so the fiber $F_i=\pi_1^{-1}(a_i)\cap X$ is $a_i$-broad. In particular, $F_i\cap V_i$ is infinite, so has non-empty interior by t-minimality. Now let $U_i$ be any non-empty open subset of $F_i\cap V_i$, and note that this ensures we satisfy (2). 

    Finally, assume we have built $a_i$ and $U_i$ as above for all $i$. Then by (1) and (2), we have:

    \begin{enumerate}
        \item[(3)] $(a_i,a_j)\in X$ whenever $i<j$. 
    \end{enumerate}

    Indeed, if $i<j$ then (1) gives $a_j\in U_i$, and so (2) gives $(a_i,a_j)\in X$. In particular, (3) implies that the fibers $\pi_2^{-1}(a_j)\cap X$ become arbitrarily large as $j\rightarrow\infty$, which contradicts that $\pi_2$ is $k$-to-1 on $X$.
\end{proof}

\subsection{Algebraicity over Broad Sets, Part 2}

Next we consider a counterpart to Lemma \ref{L: broad algebraic} where we consider two points. That is, suppose $B$ is $a$-broad, and $c\in\acl(Ab)$ for all $b\in B$. Does it follow that $c\in\acl(Aa)$? In fact, we do not know how to prove this, but we do not know of a t-miimal theory where it fails. Instead, we make it a definition:

\begin{definition}\label{D: few broad sets}
    $\mathcal M$ has the \textit{few broad sets property} if the following holds: let $X\subset M^2$ be definable, $\pi_1, \pi_2: M^2\to M$ the projections onto the first and second factors, respectively. Assume that $\pi_2$  is finite-to-one on $X$. Then for each $a\in M$, there are only finitely many $b\in M$ so that the fiber $\{y:(b,y)\in X\}$ is $a$-broad.
\end{definition}

We give some examples:

\begin{example} Suppose $\mathcal M$ satisfies exchange. Then $\mathcal M$ has the few broad sets property. Indeed, let $X$ and $a$ be as in Definition \ref{D: few broad sets}. If infinitely many fibers in $X$ under $\pi_1$ are $a$-broad, then in particular, infinitely many fibers in $X$ under $\pi_1$ are infinite. By exchange, this implies $\dim(X)=2$, which contradicts that $\pi_2$ is finite-to-one on $X$.  
\end{example}

\begin{example}
    Say that $\mathcal M$ has the \textit{few germs property} if for all $a\in M$, there are only a small (i.e. bounded) number of definable germs at $a$. For example:
    \begin{itemize}
        \item Every weakly o-minimal theory has the few germs property. Indeed, in this case the germ of a definable set $X\subset M$ at $a\in X$ has only four possibilities: (i) the single point $a$; (ii) an interval $[a,b)$ with $b>a$; (iii) an interval $(b,a]$ with $b<a$; or (iv) an interval $(b,c)$ with $b<a<c$.
        \item Every C-minimal theory has the few germs property. Indeed, suppose $\mathcal M$ is C-minimal. Then given a definable $X\subset M$ and $a\in X$, either $a$ is isolated in $X$ or $X$ contains a ball around $a$. Thus, there are only two definable germs at $a$.
    \end{itemize}
    Now we show that if $\mathcal M$ has the few germs property, then it has the few broad sets property. Indeed, let $X\subset M^2$ and $a$ be as in Definition \ref{D: few broad sets}. For each $y\in M$, let $B_y=\{a\}\cup\{z:(y,z)\in X\}$, i.e. the fiber above $y$ in $X$ under the first projection (with $a$ added to ensure that $\operatorname{germ}_a(B_y)$ is well-defined; note that adding or removing $a$ does not affect $a$-broadness).
    
    Let $Y$ the (definable) set of $y\in M$ so that $B_y$ is $a$-broad. Toward a contradiction, assume $Y$ is infinite. Then $Y$ is not small, so by the few germs property, there is an infinite $Z\subset Y$ so that $\operatorname{germ}_a(B_y)$ is constant for $y\in Z$. Thus any finitely many $B_y$ for $y\in Z$ have infinite intersection (here we use that all such $B_y$ are $a$-broad). In particular, for any positive integer $k$ we can find sets $S_k,T_k\subset M$ with $|S_k|=|T_k|=k$ and $S_k\times T_k\subset X$. By compactness, we can find infinite sets $S,T\subset M$ with $S\times T\subset X$. This contradicts that $\pi_2$ is finite-to-one on $X$.
\end{example}

\begin{example} For those well-versed in invariant types, one can generalize the previous example as follows: Say that a global 1-type is $a$-broad if every formula in it defines an $a$-broad set. Then say $\mathcal M$ is \textit{locally invariant} if for each $a\in M$, every $a$-broad global 1-type is $a$-invariant (in fact, it is enough to find a small set $B$ so that every $a$-broad 1-type is $Ba$-invariant). Weakly o-minimal and C-minimal theories are easily seen to be locally invariant. Now suppose $\mathcal M$ is locally invariant; we show $\mathcal M$ has the few broad sets property. Let $X$ and $a$ be as in Definition \ref{D: few broad sets}, where $X$ is $A$-definable. In the case that every $a$-broad type is only $Ba$-invariant for a fixed $B$, we include $B$ in $A$; in any case, every $a$-broad type is now $Aa$-invariant.

Now assuming that infinitely many $\pi_1$-fibers in $X$ are $a$-broad, one finds $b\in M-\acl(Aa)$ so that $\pi_1^{-1}(b)\cap X$ is $a$-broad. A compactness argument (using Lemma \ref{L: broad union}) produces a global $a$-broad type $p$ inside the fiber $\pi^{-1}(b)\cap X$. By $Aa$-invariance, and the fact that $b\notin\acl(Aa)$, any realization $x\models p$ satisfies $x\times T\subset X$, where $T$ is the infinite set of realizations of $\tp(b/Aa)$ in $\mathcal M$. But since $p$ is $a$-broad, it is not algebraic. Let $S$ be an infinite set of realizations of $p$ in an elementary extension. Then $S\times T\subset X$, and we again contradict that $\pi_2$ is finite-to-one on $X$. 
\end{example}

Now using a similar argument to Lemma \ref{L: broad algebraic}, we show:

\begin{lemma}\label{L: broad algebraic 2}
    Assume $\mathcal M$ has the few broad sets property. Let $a,c\in M$, $A$ a parameter set, and $B\subset M^n$ an $a$-broad set. If $c\in\acl(Ab)$ for all $b\in B$, then $c\in\acl(Aa)$.
\end{lemma}
\begin{proof}
    Similar to Lemma \ref{L: broad algebraic}. We induct on $n$, the inductive step being essentially identical to that of Lemma \ref{L: broad algebraic}. For the base case, assume $n=1$. As in Lemma \ref{L: broad algebraic}, we can reduce to a single $A$-definable set $X\subset M^2$ witnessing all algebraicities. So if $\pi_1,\pi_2$ are the two projections, then $\pi_2$ is finite-to-one on $X$ and $\pi_1^{-1}(c)\cap X$ is $a$-broad. By the few broad sets property, $c$ is one of only finitely many points with $a$-broad fiber in $X$. Thus $c\in\acl(Aa)$. 
\end{proof}

\subsection{Independent Parameters from a Broad Set}

Assume $\mathcal M$ has the few broad sets property. We now prove that given a tuple and a broad set at any of its coordinates, we can choose parameters from the broad sets without decreasing the dimension of the whole tuple. In many cases, this is sufficient to prove the independent neighborhood property.

We begin by proving the result at a single coordinate. We then conclude the full version inductively.

\begin{lemma}\label{L: ind broad}
    Assume $\mathcal M$ has the few broad sets property. Let $a=(a_1,\dots,a_n)\in M^n$, let $A$ be a parameter set, and let $B\subset M^m$ be $a_i$-broad for some $i$. Then there is $b\in B$ so that $\dim(a/Ab)=\dim(a/A)$.
\end{lemma}

\begin{proof}
    Let $e$ be a basis for $a$ over $A$. Choose $e$ so that if $a_i$ belongs to any basis, then it belongs to $e$.

    Assume $\dim(a/Ab)<\dim(a/A)$ for each $b\in B$. Then for each $b\in B$, there is a coordinate $x\in e$ so that $x\in\acl(Ab\cup(e-\{x\}))$. By Lemma \ref{L: broad union}, after shrinking $B$, we can assume there is a single coordinate $x\in e$ so that $x\in\acl(Ab\cup(e-\{x\}))$ for all $b\in B$. We have two cases:

    \begin{itemize}
        \item Case 1: $x=a_i$. Then applying Lemma \ref{L: broad algebraic} over the parameter set $A\cup(e-\{x\})$, we conclude that $a_i\in A\cup(e-\{a_i\})$. This contradicts that $e$ is a basis for $a$ over $A$.

        \item Case 2: $x\neq a_i$. Then applying Lemma \ref{L: broad algebraic 2} over the parameter set $A\cup(e-\{x\})$, we conclude that $x\in\acl(Aa_i\cup(e-\{x\})$. In particular, let $e'$ be $e$ with $x$ replaced by $a_i$; then $e\in\acl(Ae')$, thus $a\in\acl(Ae')$. Since $e$ and $e'$ have the same length, $e'$ is also a basis for $a$ over $A$. Then by the choice of $e$, $a_i\in e$. Thus $Aa_i\cup(e-\{x\})$ is just $A\cup(e-\{x\})$. So $x\in\acl(A\cup(e-\{x\}))$, and this again contradicts that $e$ is a basis for $a$ over $A$.
    \end{itemize}
\end{proof}

Finally, we conclude:

\begin{theorem}\label{T: ind broad}
    Assume $\mathcal M$ has the few broad types property. Fix $a=(a_1,\dots,a_n)\in M^n$ and a parametet set $A$. For each $i$, let $B_i$ be an $a_i$-broad subset of $M^{m_i}$. Then there is $b=(b_1,\dots,b_n)\in B_1\times\dots\times B_n$ so that $\dim(a/Ab)=\dim(a/A)$.
\end{theorem}
\begin{proof}
    Choose the $b_i$ inductively, at each stage working over $A$ together with all previous $b_j$'s.
\end{proof}

\subsection{Concluding}

Finally, we connect Theorem \ref{T: ind broad} to the topology:

\begin{definition}\label{D: loc broad}
    Say that $\mathcal M$ is \textit{locally broad} if the following holds: let $a\in M$ and $U$ an open neighborhood of $a$. Then there are $m\geq 1$ and an $a$-broad set $B\subset M^m$ so that for every $b\in B$, there is a $b$-definable open set $V_b$ with $a\in V_b\subset U$.
\end{definition}

Roughly speaking, \ref{D: loc broad} says that one can define small neighborhoods of $a\in M$ using parameters near $a$ -- and in addition, $a$-broad many perturbations of those parameters should not move the neighborhood too much (or move $a$ out of it). Weakly o-minimal and C-minimal theories easily have this property, as do most natural examples of t-minimal theories. For example, if $\mathcal M$ is weakly o-minimal and $(b,c)$ is a neighborhood of $a$, then one can define subneighborhoods using parameters $x$ and $y$ whenever $x\in(b,a)$ and $y\in(a,c)$ (and the set of such pairs $(x,y)$ is $a$-broad). The C-minimal case follows by a similar argument.

On the other hand, a real closed field with the Sorgenfrey topology is \textit{not} locally broad: for example, given $a$ and a neighborhood $[a,b)$, the set of $(x,y)$ with $a\in[x,y)\subset[a,b)$ is $\{a\}\times(a,b]$, which is not $a$-broad.

Now we show:

\begin{theorem}\label{T: inp}
    Suppose $\mathcal M$ is locally broad and has the few broad sets property. Then $\mathcal M$ has the independent neighborhood property.
\end{theorem}

\begin{proof}
    Combine the definition of local broadness with Theorem \ref{T: ind broad}.
\end{proof}

\begin{corollary}
    If $\mathcal M$ is weakly o-minimal, or if $\mathcal M$ is C-minimal, then $\mathcal M$ has the independent neighborhood property.
\end{corollary}

\begin{corollary}
    If $\mathcal M$ is locally broad and satisfies exchange, then $\mathcal M$ has the independent neighborhood property.
\end{corollary}

\section{Independent neighborhoods in visceral theories, by Will Johnson}\label{A: visceral}
We show that the independent neighborhood property (Definition \ref{D: inp}) holds in any visceral theory. Recall that a $t$-minimal theory is \textit{visceral} if it admits a definable uniformity (see, e.g., \cite[Definition 1.17]{t-minimal-johnson}). \emph{Dp-minimal} visceral theories were considered in \cite{SimWal}, and the independent neighborhood property was apparently proved for them in \cite[Corollary 3.12]{HaHaPeGps}. 

Work in a monster model $\mathcal{M}$ of a visceral theory $T$ in a language $\mathcal{L}$.  Let $\mathcal{B}$
be a definable basis for the uniformity, chosen so that basic
entourages are symmetric and open.  If $E$ is a basic entourage, and
$a \in \mathcal{M}$, then $E[a]$ denotes the ``open ball'' $\{x \in \mathcal{M} :
(a,x) \in E\}$.  The family $\{E[a] : E \in \mathcal{B}\}$ is a
neighborhood basis of $a$.  More generally if $\bar{a} = (a_1,\ldots,a_n) \in \mathcal{M}^n$, then $E[\bar{a}]$ denotes the product $\prod_{i=1}^n E[a_i]$.  The family $\{E[\bar{a}] : E \in \mathcal{B}\}$ is again a neighborhood basis of $\bar{a}$.
\begin{remark}
  The arguments below generalize to the case of a definable set $D$
  with a visceral uniformity, not necessarily definable in the induced
  structure on $D$.
\end{remark}
Recall that a partial $n$-type $p$ is \emph{broad} if, in some elementary extension, we can find infinite sets $S_1,S_2,\ldots,S_n$ such that every tuple in the product $S_1 \times S_2 \times \cdots \times S_n$ realizes $p$.  For more about broad types and sets, see \cite[\S 2.1]{t-minimal-johnson}.

\begin{lemma} \label{wj-erdos-rado}
  For any small cardinal $\kappa$, there is a small cardinal $\lambda$
  such that if $S_1,\ldots,S_n \subseteq \mathcal{M}$ are sets of size at least $\lambda$ and $A \subseteq \mathcal{M}$ is a  set of size at most $\kappa$, then
  there is $\bar{b} \in \prod_{i=1}^n S_i$ such that $\tp(\bar{b}/A)$ is
  broad.
\end{lemma}
\begin{proof}
  Fix $\kappa$ as in the statement. By the Erd\H{o}s-Rado theorem, choose $\lambda$ such that if
  we color the $n$-element subsets of $\lambda$ with
  $2^{|\mathcal{L}|+\kappa}$-many colors, then there is a homogeneous set of
  size $\aleph_0$.

  Now let $S_1,\ldots,S_n, A \subseteq \mathcal{M}$ be given.  Shrinking the
  $S_i$, we may assume $|S_i| = \lambda$ for each $i$.  Let
  $\{c_{i,\alpha} : \alpha < \lambda\}$ be an enumeration of $S_i$ for
  each $i$.  Consider the map from $n$-element subsets of $\lambda$
  to the type space over $A$ given by
  \begin{equation*}
    \{\alpha_1,\ldots,\alpha_n\} \mapsto \tp(c_{1,\alpha_1},\ldots,c_{n,\alpha_n}/A)
  \end{equation*}
  for $\alpha_1 < \alpha_2 < \cdots < \alpha_n$.  Let $J \subseteq
  \lambda$ be an infinite subset which is homogeneous of color $p$, so
  that
  \begin{equation*}
    \tp(c_{1,\alpha_1},\ldots,c_{n,\alpha_n}/A) = p
  \end{equation*}
  for all $\alpha_1 < \cdots < \alpha_n$ in $J$.  Note that
  $(c_{1,\alpha_1},\ldots,c_{n,\alpha_n}) \in \prod_{i=1}^n S_i$, so
  it remains to show that $p$ is broad, i.e., the set $p(\mathcal{M})$ of
  realizations is broad.  Given any $k$, we can find an increasing
  sequence
  \begin{equation*}
    \beta_{1,1} < \beta_{1,2} < \cdots < \beta_{1,k} < \beta_{2,1} < \cdots < \beta_{2,k} < \cdots < \beta_{n,1} < \cdots < \beta_{n,k}
  \end{equation*}
  in $J$, because $J$ is infinite.  Then every element of
  \begin{gather*}
    \prod_{i=1}^n \{c_{i,\beta_{i,1}},c_{i,\beta_{i,2}},\ldots,c_{i,\beta_{i,k}}\} \\
    = \{c_{1,\beta_{1,1}},c_{1,\beta_{1,2}},\ldots,c_{1,\beta_{1,k}}\} \times
    %% \{c_{2,\beta_{2,1}},c_{2,\beta_{2,2}},\ldots,c_{2,\beta_{2,k}}\} \times
    \cdots
    \times \{c_{n,\beta_{n,1}},c_{n,\beta_{n,2}},\ldots,c_{n,\beta_{n,k}}\}
  \end{gather*}
  has type $p$.  As $k$ was arbitrary, $p(\mathcal{M})$ is broad.
\end{proof}

\begin{lemma}
  Let $(P,\le)$ be an $A$-definable preorder.  Suppose $(P,\le)$ is
  downward directed, in the sense that for any $x,y \in P$, there is
  $z \in P$ with $z \le x$ and $z \le y$.  If $\tp(\bar{b}/A)$ is broad
  and $c \in P$, there is $c' \le c$ such that $\tp(\bar{b}/Ac')$ is
  broad.
\end{lemma}
\begin{proof}
  Let $\kappa = |A| + \aleph_0$.  Let $\lambda$ be as in
  Lemma~\ref{wj-erdos-rado}.  Let $p = \tp(\bar{b}/A)$.  Since $p$ is broad,
  the set $p(\mathcal{M})$ contains a product $S := \prod_{i=1}^n S_i$, with
  $|S_i| = \lambda$ for each $i$.  Note that $\bar{e} \equiv_A \bar{b}$ for
  every $\bar{e} \in S$.  By the axiom of choice, there is a function $f :
  S \to P$ such that
  \begin{equation*}
    (\bar{e},f(\bar{e})) \equiv_A (\bar{b},c) \text{ for all } \bar{e} \in S.
  \end{equation*}
  By saturation and directedness, there is some $c_0 \in P$ such that
  $c_0 \le f(\bar{e})$ for every $\bar{e} \in S$.  By
  Lemma~\ref{wj-erdos-rado}, there is some $\bar{e} \in S$ such that
  $\tp(\bar{e}/Ac_0)$ is broad.  Take $c' \in P$ such that
  \begin{equation*}
    (\bar{e},f(\bar{e}),c_0) \equiv_A (\bar{b},c,c').
  \end{equation*}
  The fact that $c_0 \le f(\bar{e})$ implies $c' \le c$.  The fact
  that $\tp(\bar{e}/Ac_0)$ is broad implies that $\tp(\bar{b}/Ac')$ is
  broad.
\end{proof}

\begin{corollary}\label{wj-dim-case}
  Let $(P,\le)$ be an $A$-definable downward directed preorder.  For
  any $\bar{b} \in \mathcal{M}^n$ and any $c \in P$, there is $c' \le c$ such that
  $\dim(\bar{b}/A) = \dim(\bar{b}/Ac')$.
\end{corollary}
\begin{proof}
  Let $\bar{e}$ be a subtuple of $\bar{b}$ which is $\acl$-independent
  over $A$, and interalgebraic with $\bar{b}$ over $A$, so
  \begin{equation*}
    \dim(\bar{b}/A) = |\bar{e}|.
  \end{equation*}
  Moreover, $\tp(\bar{e}/A)$ is broad.  Take $c' \le c$ such that
  $\tp(\bar{e}/Ac')$ is broad.  Then $\bar{e}$ is $\acl$-independent
  over $Ac'$, and interalgebraic with $\bar{b}$ over $Ac'$, and so
  \begin{equation*}
    \dim(\bar{b}/Ac') = |\bar{e}| = \dim(\bar{b}/A).  \qedhere
  \end{equation*}
\end{proof}

\begin{lemma} \label{wj-close-conjugates}
  Let $p = \tp(\bar{a}/C)$, and let $p(\mathcal{M})$ be the set of realizations.
  Let $k = \dim(\bar{a}/C) = \dim(p) = \dim(p(\mathcal{M}))$.  If $U \ni \bar{a}$ is
  any neighborhood, then $\dim(p(\mathcal{M}) \cap U) = k$.
\end{lemma}
\begin{proof}
  Otherwise, there is a $C$-definable set $D \supseteq p(\mathcal{M})$ such
  that $\dim(D \cap U) < k$ by \cite[Corollary~2.42]{t-minimal-johnson}.  Then
  $\dim_{\bar{a}}(D) < k$, where $\dim_{\bar{a}}(D)$ is the local dimension of $D$
  at $\bar{a}$.  Let $X = \{\bar{b} \in D : \dim_{\bar{b}}(D) < k\} \subseteq D$.
  The set $X$ contains $\bar{a}$ and is $C$-definable (by definability of dimension \cite[Theorem~2.51]{t-minimal-johnson}).  If $\bar{b} \in X$, then $\dim_{\bar{b}}(X) \le
  \dim_{\bar{b}}(D) < k$.  So the local dimension of $X$ is at most $k-1$ at
  any point.  By \cite[Proposition~3.10]{t-minimal-johnson},
  \begin{equation*}
    \dim(X) = \max_{\bar{b} \in X} \dim_{\bar{b}}(X) \le k-1.
  \end{equation*}
  Then $\bar{a}$ belongs to the $(k-1)$-dimensional $C$-definable set $X$,
  contradicting the fact that $\dim(\bar{a}/C) = k$.
\end{proof}

\begin{lemma} \label{wj-strange}
  Let $C$ be a set.  Let $E$ be a $C$-definable basic entourage.  For
  any $\bar{b}$, there is $\bar{c} \in E[\bar{b}] \subseteq \mathcal{M}^n$ such that
  $\dim(\bar{b},\bar{c}/C) = \dim(\bar{b}/C) + \dim(\bar{c}/C)$.
\end{lemma}
\begin{proof}
  Take $E_1$ a basic entourage such that $E_1 \circ E_1 \subseteq E$.
  Let $C_1 \supseteq C$ be a small set defining $E_1$.  Let $k =
  \dim(\bar{b}/C)$.  Let $n = |\bar{b}| = \dim E[\bar{b}] = \dim E_1[\bar{b}]$.
  Let $X$ be the type-definable set of tuples $\bar{b}'$ such that
  $\bar{b}' \equiv_C \bar{b}$ and $\bar{b}' \in E_1[\bar{b}]$.  By
  Lemma~\ref{wj-close-conjugates}, $\dim(X) = \dim(\bar{b}/C) = k$.  Then
  $\dim(X \times E_1[\bar{b}]) = k + n$.  Since the set $X \times
  E_1[\bar{b}]$ is type-definable over $C_1\bar{b}$, there are
  $(\bar{b}',\bar{c}') \in X \times E_1[\bar{b}]$ with
  $\dim(\bar{b}',\bar{c}'/C_1\bar{b}) = k + n$.  Then $\tp(\bar{b}'/C) =
  \tp(\bar{b}/C)$, so $\dim(\bar{b}'/C) = k$.  And $\dim(\bar{c}'/C) \le
  n$ because $\bar{c}'$ is an $n$-tuple.  Then
  \begin{gather*}
    k + n = \dim(\bar{b}',\bar{c}'/C_1\bar{b}) \le \dim(\bar{b}',\bar{c}'/C) \\
    \le \dim(\bar{b}'/C) + \dim(\bar{c}'/C) \le k + n.
  \end{gather*}
  Equality must hold, so $\dim(\bar{b}',\bar{c}'/C) = \dim(\bar{b}'/C) +
  \dim(\bar{c}'/C)$.  By choice of $(\bar{b}',\bar{c}') \in X \times
  E_1[\bar{b}]$, we have $\bar{b}', \bar{c}' \in E_1[\bar{b}]$.  By choice of $E_1$,
  it follows that $\bar{c}' \in E[\bar{b}']$.  To summarize,
  % Since I wrote the appendix, it doesn't make sense to go through and change it from American English to British English
  we have $
  \bar{b}', \bar{c}'$ such that
  \begin{gather*}
    \bar{b}' \equiv_C \bar{b} \\ \bar{c}' \in E[\bar{b}']
    \\ \dim(\bar{b}',\bar{c}'/C) = \dim(\bar{b}'/C) + \dim(\bar{c}'/C)
  \end{gather*}
  Taking $\bar{c}$ such that $(\bar{b}',\bar{c}') \equiv_C (\bar{b},\bar{c})$,
  and moving by an automorphism, we get the desired configuration.
\end{proof}
One might conjecture the following cleaner form of
Lemma~\ref{wj-strange}:
\begin{nontheorem}\label{wj-nontheorem}
  Given a type $\tp(\bar{a}/C)$ and a non-empty open set $U \subseteq
  \mathcal{M}^n$, there is $\bar{b} \in U$ with $\dim(\bar{a},\bar{b}/C) = \dim(\bar{a}/C) +
  \dim(\bar{b}/C)$.
\end{nontheorem}
If the exchange property holds, this is true: take $\bar{b} \in U$ with
$\dim(\bar{b}/C, \ulcorner U \urcorner ,\bar{a}) = n$.  Then $\dim(\bar{b}/C) =
\dim(\bar{b}/C\bar{a}) = n$, so $\dim(\bar{a},\bar{b}/C) = \dim(\bar{a}/C) + \dim(\bar{b}/C)$
by additivity.

On the other hand, if the exchange property fails, so does
Non-Theorem~\ref{wj-nontheorem}.  Specifically, take $a,b_0 \in \mathcal{M} \setminus
\acl(C)$ with $a \in \acl(Cb_0)$ and $b_0 \notin \acl(Ca)$.  Take
$\phi(x,y) \in \tp(a,b_0/C)$ algebraizing $a$ over $b_0$.  The
boundary $\bd(\phi(a,\mathcal{M}))$ is finite and $Ca$-definable, so it
doesn't contain $b_0$, and instead $b_0 \in \inter(\phi(a,\mathcal{M})) =: U$.
If $b \in U$, then $a$ is in the finite set $\phi(\mathcal{M},b)$, so $a \in
\acl(Cb)$ and
\begin{equation*}
  \dim(a,b/C) = \dim(b/C) < \dim(a/C) + \dim(b/C).
\end{equation*}
So we cannot find any $b \in U$ which is ``dimension independent''
from $a$ over $C$.

\begin{theorem} \label{wj-indep-nbhd}
$\mathcal{M}$ has the independent neighborhood property:
  let $\bar{b}$ be a finite tuple in $\mathcal{M}$.  Let $C$ be a small set of
  parameters.  Let $U$ be a neighborhood of $\bar{b}$.  Then there is $C'
  \supseteq C$ such that $\dim(\bar{b}/C') = \dim(\bar{b}/C)$, and a
  $C'$-definable neighborhood $U' \ni \bar{b}$, with $U' \subseteq U$.
\end{theorem}
\begin{proof}
  Let $\{E_p : p \in P\}$ be the definable basis for the uniformity.
  Regard $P$ as a downward-directed preorder by defining
  \begin{equation*}
    p \le q \iff E_p \subseteq E_q.
  \end{equation*}
  Let $p_0$ be so small that $E_{p_0}^2[\bar{b}] \subseteq U$.  By
  Corollary~\ref{wj-dim-case}, there is $p \le p_0$ such that
  $\dim(\bar{b}/Cp) = \dim(\bar{b}/C)$.  By Lemma~\ref{wj-strange},
  there is $\bar{c} \in E_p[\bar{b}]$ such that
  \begin{equation*}
    \dim(\bar{b},\bar{c}/Cp) = \dim(\bar{b}/Cp) + \dim(\bar{c}/Cp).
  \end{equation*}
  By subadditivity of dimension \cite[Proposition~2.31(5)]{t-minimal-johnson},
  \begin{equation*}
    \dim(\bar{b},\bar{c}/Cp) \le \dim(\bar{b}/Cp \bar{c}) + \dim(\bar{c}/Cp) \le \dim(\bar{b}/Cp) + \dim(\bar{c}/Cp),
  \end{equation*}
  and so equality holds and
  \begin{equation*}
    \dim(\bar{b}/C p \bar{c}) = \dim(\bar{b}/Cp) = \dim(\bar{b}/C).
  \end{equation*}
  Because $E_p^2[\bar{b}] \subseteq E_{p_0}^2[\bar{b}] \subseteq U$ and $\bar{c}
  \in E_p[\bar{b}]$, we have $\bar{b} \in E_p[\bar{c}] \subseteq U$.  Take $U' =
  E_p[\bar{c}]$ and $C' = Cp\bar{c}$.  Then $U'$ is $C'$-definable and
  $\dim(\bar{b}/C') = \dim(\bar{b}/C)$.
\end{proof}

\bibliographystyle{plain}
\bibliography{harvard}

\end{document}